\documentclass[10pt,a4paper,oneside,reqno]{amsart}
\usepackage{setspace}
\usepackage[totalwidth=14.5cm,totalheight=22.6cm]{geometry}
\usepackage[T1]{fontenc}
\usepackage{graphicx}
\usepackage{caption}
\usepackage{epsfig}
\usepackage{amsmath,amssymb,amscd,amsthm}
\usepackage{subfigure}
\usepackage{latexsym}
\usepackage{graphics}
\usepackage{colortbl} 
\usepackage{color}
\usepackage{ae}
\usepackage{enumitem}
\usepackage[all]{xy}
\usepackage{scalerel}
\usepackage{tikz}
\usepackage{cancel}
\usepackage{bbm,amsbsy}
\usepackage{mathrsfs}  
\usepackage[colorlinks=true,pagebackref=true]{hyperref}
\usepackage[normalem]{ulem}
\usepackage{nicefrac}
\usepackage{xfrac}
\usepackage{faktor}
\usepackage{tikz-cd} 
\usepackage{systeme}
\usepackage{nomencl}
\usepackage[normalem]{ulem}

\makenomenclature

\makeatletter
\newtheorem*{rep@theorem}{\rep@title}
\newcommand{\newreptheorem}[2]{%
\newenvironment{rep#1}[1]{%
 \def\rep@title{#2 \ref{##1}}%
 \begin{rep@theorem}}%
 {\end{rep@theorem}}}
\makeatother

\makeatletter
\newtheorem*{rep@cor}{\rep@title}
\newcommand{\newrepcor}[2]{%
\newenvironment{rep#1}[1]{%
 \def\rep@title{#2 \ref{##1}}%
 \begin{rep@cor}}%
 {\end{rep@cor}}}
\makeatother

\makeatletter
\newtheorem*{rep@prop}{\rep@title}
\newcommand{\newrepprop}[2]{%
\newenvironment{rep#1}[1]{%
 \def\rep@title{#2 \ref{##1}}%
 \begin{rep@prop}}%
 {\end{rep@prop}}}
\makeatother

\newtheorem{cor}{Corollary}[section]

\newtheorem{theorem}[cor]{Theorem}

\newtheorem{prop}[cor]{Proposition}

 \allowdisplaybreaks
\newrepcor{cor}{Corollary}
\newreptheorem{theorem}{Theorem}
\newrepprop{prop}{Proposition}

\newtheorem{lemma}[cor]{Lemma}

\theoremstyle{definition}
\newtheorem{defi}[cor]{Definition}
\theoremstyle{remark}
\newtheorem{remark}[cor]{Remark}
\newtheorem*{remark*}{Remark}
\newtheorem{example}[cor]{Example}

\newtheorem*{notation*}{Notation}

\newlist{steps}{enumerate}{1}
\setlist[steps, 1]{itemsep=8pt,leftmargin=0cm,itemindent=.5cm,labelwidth=\itemindent,labelsep=0cm,align=left,label = \textbf{\emph{Step \arabic*}:\,}}

\makeatletter
\newcommand{\myitem}[1]{%
\item[#1]\protected@edef\@currentlabel{#1}%
}
\makeatother

\newcommand{\minko}{\mathbb{R}^{1,2}}

\newcommand{\ads}{\mathbb{A}\mathrm{d}\mathbb{S}^3}
\newcommand{\HP}{\mathbb{HP}^3}
\newcommand{\isom}{\mathrm{Isom}}

\newcommand{\inner}[1] {\langle #1 \rangle}
\newcommand{\inners}{\langle \cdot, \cdot \rangle}

\begin{document}\raggedbottom

\setcounter{secnumdepth}{3}
\setcounter{tocdepth}{2}

\title[Infinitesimal earthquake]{The Infinitesimal earthquake theorem for vector fields on the circle}

\author[Farid Diaf]{Farid Diaf}
\address{Farid Diaf: Univ. Grenoble Alpes, CNRS, IF, 38000 Grenoble, France.} \email{farid.diaf@univ-grenoble-alpes.fr}

\thanks{}

\maketitle

\begin{abstract}
We prove that any continuous vector field on a circle is the extension in a suitable sense, of a unique infinitesimal earthquake of the hyperbolic plane. Furthermore, we obtain other extension results when the vector field is assumed only to be upper or lower semicontinuous. This leads to a generalization of Kerckhoff's and Gardiner's infinitesimal earthquake theorems to a broader setting, using a completely novel approach. The proof is based on the geometry of the dual of Minkowski three-space, also called Half-pipe three-geometry. In this way, we obtain a simple characterization of Zygmund vector fields on the circle in terms of width of convex hulls.
\end{abstract}
\tableofcontents
\section{Introduction}
In his $1990$ pioneering paper \cite{Mess}, Mess provided a deep connection between the study of Teichmüller theory of hyperbolic surfaces and three-dimensional Lorentzian geometries of constant sectional curvature. In the case of zero curvature, the \textit{Minkowski space} is the model of flat Lorentzian geometry. Mess was interested in \textit{flat maximal globally hyperbolic spacetimes}, whose study is motivated by general relativity. He showed that the moduli space of these spacetimes is the tangent bundle of the Teichmüller space of the spatial part. Several contributions have been made on this subject, see for example \cite{Note_on_paper_mess,flatspacetimes_bonsante,Barbot_flatspacetime,canorot,BS_flat_conical}.

Mess also studied the case of negative curvature Lorentzian geometry, specifically the so-called \textit{Anti-de Sitter space}, denoted by $\ads$. It is the Lie group $\mathrm{Isom}(\mathbb{H}^2)$ of orientation-preserving isometries of the hyperbolic plane $\mathbb{H}^2$, endowed with its bi-invariant metric induced by its Killing form.

Mess observed that through the \textit{Gauss map} construction, the convex hull in Anti-de Sitter space provides an earthquake map between closed hyperbolic surfaces. This discovery led him to recover of the so-called Thurston's \textit{earthquake theorem} \cite{Thurston}. Roughly speaking, earthquake maps are a continuous version of Dehn twists. Originally developed mainly by Kerckhoff and Thurston, they played a central role in the study of the mapping class group and were a fundamental step towards solving the Nielsen realization problem, as proved by Kerckhoff \cite{Kerearth}. Earthquake maps have been the subject of extensive research in various direction \cite{GHL,Junhu_Zygnorm,Gardiner1999QuasiconformalTT,Saric,miyachisaric,Fixedpoint,earthquekeandparticles}.

Since Mess's work, interest in three-dimensional Anti-de Sitter space has grown, and the Gauss map construction has been generalized to various settings, which have proven to be useful for finding extensions of orientation-preserving circle homeomorphisms to the disk, see for example \cite{Maximalsurface, Areapreserving}. See also \cite{diafseppi2023antide} for a detailed Anti-de Sitter proof of the general version of Thurston's earthquake theorem, which states that every orientation preserving circle homeomorphism is this extension on the boundary of an earthquake of $\mathbb{H}^2$.

\subsection{Main results}
The purpose of this paper is to investigate the infinitesimal counterpart of the aforementioned construction. Specifically, we aim to find a natural way to extend a vector field defined on the circle into the disk using the geometry of the three dimensional \textit{Half-pipe space} $\HP$, also known as the \textit{Co-Minkowski space}. One of the models of Minkowski space $\minko$, is the three-space dimensional vector space $\mathbb{R}^3$ endowed with a bilinear form of signature $(-,+,+)$. The Half-pipe space is the space of all spacelike planes in $\minko$. This space can be identified with $\mathbb{H}^2\times\mathbb{R}$. Indeed, for each pair $(\eta, t)\in\mathbb{H}^2\times\mathbb{R}$, one can associate a spacelike plane in $\minko$ for which the normal is given by $\eta$, and the oriented distance through the normal direction is $t$. Given that the tangent bundle of $\mathbb{S}^1$ is trivial, one can view each vector field $X$ on the circle as a function $\phi_X:\mathbb{S}^1 \to \mathbb{R}$, called the \textit{support function}, where $X(z) = iz\phi_X(z)$ for every $z \in \mathbb{S}^1$. We use the terminology of "support function" even though it is typically assigned to convex sets. Indeed, we will see in Remark \ref{remark_on_support_function}, that the function $\phi_X$, serves as the support function of a convex set in Minkowski space.
Now having identified a vector field $X$ with its support function $\phi_X:\mathbb{S}^1\to\mathbb{R}$, we can see the graph of $\phi_X$ in $\mathbb{S}^1\times\mathbb{R}\cong\partial\mathbb{H}^2\times\mathbb{R}$, can be considered as a curve in the boundary at infinity of $\HP$. The philosophy here is that any surface in $\HP$ that is bounded by that curve will provide an extension to the disk of our vector field $X$.

Applying this construction to the convex hull in $\HP$, we obtain the following result, which can be regarded as the infinitesimal version of Thurston's earthquake theorem in the universal setting (see Theorem \ref{thurs}):

\begin{theorem}\label{TH1}
Let $X$ be a continuous vector field on $\mathbb{S}^1$. Then there exists a left (or right) infinitesimal earthquake $\mathcal{E}$ on $\mathbb{H}^2$ along a geodesic lamination $\lambda$ which extends continuously to $X$ on $\mathbb{S}^1$. Moreover the infinitesimal earthquake is unique, except that there is a range of choices for the infinitesimal earthquake on any leaf of $\lambda$ where $\mathcal{E}$ is discontinuous.
\end{theorem}
In simpler terms, an infinitesimal earthquake is a vector field on the hyperbolic plane that acts as a \textit{Killing} vector field in the strata of the geodesic lamination and "slips" along the leaves of $\lambda$. The geodesic lamination $\lambda$ supports a transverse measure called the \textit{infinitesimal earthquake measure}, which quantifies the amount of shearing between two strata of $\lambda$. The infinitesimal earthquake is continuous except on the \textit{atomic} leaves of $\lambda$, namely, those leaves that have a positive measure.

Infinitesimal earthquakes are essential objects in \textit{infinitesimal Teichmüller theory}. For instance, when the infinitesimal earthquake is invariant under the action of a Fuchsian group, this vector field represents a tangent vector of the Teichmüller space of the hyperbolic surface obtained by quotienting the hyperbolic plane by the underlying Fuchsian group (see Section \ref{equiva_inf_earth} for more details). In this context, it can also be seen as the gradient of the length function, as proven by Wolpert \cite{wolpert}.

Theorem \ref{TH1} can be extended to cases where the vector field $X$ is only lower/upper semicontinuous, meaning that $\phi_X$ is lower/upper semi-continuous. In this situation, we obtain extensions along line segments.

\begin{theorem}\label{THrad}
Let $X$ be a vector field on $\mathbb{S}^1$. Then:
\begin{enumerate}
\item If $X$ is lower semicontinuous, then there exists a left infinitesimal earthquake $\mathcal{E}$ on $\mathbb{H}^2$ that extends radially to $X$.
\item If $X$ is upper semicontinuous, then there exists a right infinitesimal earthquake $\mathcal{E}$ on $\mathbb{H}^2$ that extends radially to $X$.
\end{enumerate}
\end{theorem}

To clarify the statement, consider $\mathbb{D}^2$, the Klein model of the hyperbolic space $\mathbb{H}^2$. We say that a vector field $\hat{X}$ on $\mathbb{D}^2$ \textit{extends radially} to $X$ if, for all $z\in\mathbb{S}^1$ and $x\in \mathbb{D}^2$, we have:  
$$ \lim_{s \to 0^+}\hat{X}((1-s)z+sx)=X(z).$$
Not that one cannot expect that the field $\hat{X}$ extends continuously to $X$, see Remark \ref{derniere_remark}. It is worth noting that Gardiner \cite{Gardinerthurston} proved an infinitesimal earthquake theorem for vector fields on the circle that are Zygmund. These vector fields are, in particular, H\"older continuous for all $0<\alpha<1$. Our Theorem \ref{TH1} is a generalization of Gardiner's infinitesimal earthquake since we only assume continuity of $X$. Furthermore, our approach is entirely novel. In fact, using our approach, we provide the following characterization of \textit{Zygmund} vector fields in terms of the \textit{width} of convex hulls in Half-pipe geometry.

\begin{theorem}\label{TH2}
Let $X$ be a continuous vector field on $\mathbb{S}^1$. Let $\mathcal{E}$ be a left (or right) infinitesimal earthquake along a measured lamination $\lambda$, extending to $X$. Denote: 
\begin{itemize}
    \item $w(X)$ the width of $X.$
    \item $\lVert X\rVert_{cr}$ the cross ratio norm of $X$.
    \item $\lVert\lambda\rVert_{\mathrm{Th}}$ the Thurston norm of $\lambda$.
\end{itemize}
Then the following estimates hold:    $$\frac{1-\tanh(1)}{2\sqrt{2}}\lVert\lambda\rVert_{\mathrm{Th}}\leq w(X)\leq \frac{8}{3}\lVert X\rVert_{cr}.$$
\end{theorem}
Let us explain the statement, concerning the width of a vector field. First, by using the support function of $X$, one can define the \textit{convex hull} of $X$ as the convex hull of its support function $\phi_X:\mathbb{S}^1\to \mathbb{R}$. The boundary of this convex hull in $\HP\cong\mathbb{D}^2\times\mathbb{R}$ consists of two connected components homeomorphic to disks, which we will refer  as $\partial_+\mathcal{C}(X)$ and $\partial_-\mathcal{C}(X)$. Suppose we take points $(\eta, t)\in\partial_+\mathcal{C}(X)$ and $(\eta, s)\in\partial_+\mathcal{C}(X)$. Their dual points are spacelike planes in Minkowski space $\minko$, and these planes are parallel since they have the same normal. Therefore, one can define the distance between $(\eta, t)$ and $(\eta, s)$ as the timelike distance between their spacelike dual planes in $\minko$. The \textit{width} of $X$ is then defined as the supremum of the distance between points $(\eta, t)\in\partial_+\mathcal{C}(X)$ and $(\eta, s)\in\partial_+\mathcal{C}(X)$. Roughly speaking, the cross ratio norm of a vector field measures how far the vector field is from a Killing vector field, in the sense that $\lVert X\rVert_{cr}$ is zero if and only if $X$ is the extension to $\mathbb{S}^1$ of a Killing vector field of $\mathbb{H}^2$. A continuous vector field is \textit{Zygmund} if it has a finite cross-ratio norm. Finally, for a measured geodesic lamination $\lambda$ of $\mathbb{H}^2$, the \textit{Thurston norm} of $\lambda$ is defined as:
\begin{equation}\label{thurstonnorm}     \lVert \lambda\rVert_{\mathrm{Th}}:=\underset{I \ \ \ \ \ \  }{\sup\lambda(I)},          \end{equation} where the supremum is over all geodesic arcs $I$ of unit length that transversely intersect the support of $\lambda.$

It is known from the work of Jun Hu \cite{Junhu_Zygnorm} that $\lVert X\rVert_{cr}\leq C\lVert \lambda\rVert_{\mathrm{Th}}$ holds for some universal constant $C>0$. Combining this result with Theorem \ref{TH2}, we conclude that the three quantities—the Thurston norm, cross ratio norm, and the width—are equivalent and one is finite if and only if all the three quantities are finite. 

It's worth noting that the Anti-de Sitter version of Theorem \ref{TH2} has also been established. Indeed, Bonsante and Schlenker \cite{Maximalsurface} provided a characterization of quasisymmetric circle homeomorphisms in terms of finiteness of the \textit{width} of the convex core of their graph in $\ads$. Later, Seppi \cite{SEP19} provides a quantitative estimate between the width of a quasisymmetric homeomorphism and its cross-ratio norm.
\subsection{Length and Thurston norm}
We now provide an application of the estimates obtained in Theorem \ref{TH2} to derive a result that seems to be non-trivial by using only techniques from hyperbolic geometry. Let $\Sigma_g$ be a connected, oriented, closed surface of genus $g \geq 2$, and denote by $\isom(\mathbb{H}^2)$ the isometry group of $\mathbb{H}^2$. Recall that a representation $\rho: \pi_1(\Sigma_g) \to \isom(\mathbb{H}^2)$ is \textit{Fuchsian} if $\mathbb{H}^2/\rho(\pi_1(\Sigma_g))$ is a closed \textit{hyperbolic} surface:

\begin{prop}\label{lengthvsThurs}
Let $\rho: \pi_1(\Sigma_g) \to \isom(\mathbb{H}^2)$ be a Fuchsian representation, and $\lambda$ be a measured geodesic lamination on the hyperbolic surface $S_{\rho} := \mathbb{H}^2/\rho(\pi_1(\Sigma_g))$. Then, the following inequality holds:
\begin{equation}
\ell_{\rho}(\lambda) \leq \frac{64\pi(3e-1)}{3(e-1)}(g-1) \lVert \lambda\rVert_{\mathrm{Th}}^{\rho}.
\end{equation}
\end{prop}

It is worth recalling that the \textit{length} $\ell_{\rho}(\cdot)$ of a measured geodesic lamination is the unique continuous homogeneous function that extends the length of simple closed geodesics of $S_{\rho}$. The Thurston norm $\lVert \lambda\rVert_{\mathrm{Th}}^{\rho}$ is defined as the Thurston norm (as in \eqref{thurstonnorm}) of the lift of $\widetilde{\lambda}$ in $\mathbb{H}^2$ via the covering map $\mathbb{H}^2\to S_{\rho}$. Proposition \ref{lengthvsThurs} follows from an interpretation obtained by Barbot and Fillastre \cite{barbotfillastre} of the length function in terms of volume of some region in Half-pipe geometry.

It is important to note that the converse inequality does not hold, namely we can not have $\ell_{\rho}(\lambda)\geq C_g\lVert \lambda\rVert_{\mathrm{Th}}^{\rho}$ for some constant $C_g>0$ depending only on the genus of $\Sigma_g$. Indeed, one can consider a sequence of Fuchsian representations $\rho_n$ such that the length of a closed geodesic $\alpha$ tends to $0$. According to the classical collar lemma in hyperbolic geometry, the tubular neighborhood of $\alpha$ is large for such representations. Now, take $\lambda$ to be a measured geodesic lamination supported on $\alpha$ and consider the measure as the intersection number with $\alpha$. Then, a geodesic segment $I$ of length $1$ and transverse to $\alpha$ can never escape the tubular neighborhood of $\alpha$, so $\lambda(I)=1$, and hence the Thurston norm of $\lambda$ is $1$. However, $\ell_{\rho_n}(\lambda)\to 0$, making the inequality $\ell_{\rho}(\lambda)\geq C_g\lVert \lambda\rVert_{\mathrm{Th}}^{\rho}$ impossible.

\subsection{The strategy of the proof}

The key idea behind proving Theorem \ref{TH1} is based on a correspondence between \textit{spacelike} planes in Half-pipe space and points in Minkowski space. A plane in the the projective model of Half-pipe space $\mathbb{D}^2\times\mathbb{R}$ is spacelike if it is not vertical. This correspondence can be seen as the infinitesimal version of the projective duality between points and spacelike planes in $\ads$. Another important aspect is that one of the models of Minkowski space is the Lie algebra $\mathfrak{isom}(\mathbb{H}^2)$ of the Lie group $\mathrm{Isom}(\mathbb{H}^2)$, where each element of $\mathfrak{isom}(\mathbb{H}^2)$ corresponds to a Killing vector field of $\mathbb{H}^2$. As a result of this discussion, we establish a homeomorphism:

\begin{equation}\label{introduction_duality}
\mathcal{K}: \{\mathrm{Spacelike}\ \mathrm{planes}\ \mathrm{in}\ \HP\}\cong \{ \mathrm{Killing}\ \mathrm{vector} \ \mathrm{fields}\ \mathrm{in}\ \mathbb{H}^2\}.
\end{equation}
The approach to prove Theorem \ref{TH1} is as follows: Given a continuous vector field $X$ on $\mathbb{S}^1$ and its support function $\phi_X$, we observe that, up to changing the labeling, the boundary $\partial_-\mathcal{C}(X)$ (resp. $\partial_+\mathcal{C}(X)$) is the graph of a convex function $\phi_X^-:\mathbb{D}^2\to\mathbb{R}$ (resp. concave $\phi_X^+:\mathbb{D}^2\to \mathbb{R}$). Now, we need to define a vector field on $\mathbb{D}^2$ that is a left (resp. right) infinitesimal earthquake.

Let $\eta\in \mathbb{D}^2$ and choose $\mathrm{P}_{(\eta,\phi_X^+(\eta))}$ and $\mathrm{P}_{(\eta,\phi_X^+(\eta))}$ support planes of $\partial_-\mathcal{C}(X)$ and $\partial_+\mathcal{C}(X)$ at $(\eta,\phi_X^-(\eta))$ and $(\eta,\phi_X^+(\eta))$, respectively. This choice of support planes can be made in a canonical way, which will be explained later. Then, using the identification \eqref{introduction_duality}, $\mathcal{K}(\mathrm{P}_{(\eta,\phi_X^+(\eta))})$ and $\mathcal{K}(\mathrm{P}_{(\eta,\phi_X^+(\eta))})$ correspond to Killing vector fields of $\mathbb{H}^2$. We define:
\begin{equation}\label{vector_field_duality}
    \mathcal{E}^{\pm}_X(\eta)=\mathcal{K}(\mathrm{P}_{(\eta,\phi_X^+(\eta))})(\eta).
\end{equation}
Remark that the boundaries $\partial_\pm\mathcal{C}(X)$ are almost everywhere totally geodesic except on a measured geodesic laminations $\lambda^{\pm} $, which is supported where the surface is bent. While proving that $\mathcal{E}_X^-$ (resp. $\mathcal{E}_X^+$) is a left (resp. right) infinitesimal earthquake along $\lambda^{-}$ (resp. $\lambda^{+}$) is relatively easy, the extension of $\mathcal{E}_X^\pm$ to $X$ requires a more technical argument. We prove in Section \ref{Section3} the following:

\begin{prop}\label{extension}
Let $X$ be a vector field on $\mathbb{S}^1$. Then: 
\begin{enumerate}
\item If $X$ is continuous, then $\mathcal{E}_X^+$ and $\mathcal{E}_X^-$ extend continuously to $X$.
\item If $X$ is lower semicontinuous, then $\mathcal{E}_X^-$ extends radially to $X$.
\item If $X$ is upper semicontinuous, then $\mathcal{E}_X^+$ extends radially to $X.$
\end{enumerate}
\end{prop}
Next, in the case where $X$ is a continuous vector field on the circle, we will prove that our construction is the unique way to construct infinitesimal earthquake that extends continuously to $X$ in the following sense: if there exists another left or right infinitesimal earthquake that extends to $X$, it must be equal to $\mathcal{E}_X^{\pm}$ except on the leaves where $\mathcal{E}_X^{\pm}$ is discontinuous. Moreover, from our construction, we can see that the non-uniqueness on leaves of discontinuity is due to the non-uniqueness of support planes, and hence any choice of support plane provides another infinitesimal earthquake. To prove this, we need to use some fundamental properties of \textit{simple infinitesimal earthquakes}, which are, by definition, the restriction of a Killing vector field $K_i$ to a disk $D_i$ ($i = 1, 2$), where $D_1 \cup D_2$ covers the entire disk $\mathbb{D}^2$.

\subsection{Organization of the paper}

Section \ref{section2} reviews definitions and some of the techniques used in the paper. First, we recall the definition of Minkowski space and half-pipe space and explain the duality between these two geometries. Next, we introduce the support function of a vector field and define the width of a vector field. In Section \ref{novsec}, we discuss the notions of earthquake maps and infinitesimal earthquakes, and extract some fundamental properties of simple infinitesimal earthquakes.
In Section \ref{Section3}, we prove Proposition \ref{extension}, which is the extension part of Theorems \ref{TH1} and \ref{THrad}. After that, we show that the vector field obtained from the convex hull construction satisfies the properties of an infinitesimal earthquake.
Section \ref{section4} is devoted to proving the estimate between the width of a vector field, and its cross ratio norm. In Section \ref{section5}, we prove the estimate between the Thurston norm of the measured geodesic lamination, and the width and then proving Theorem \ref{TH2}. Finally, in the last Section \ref{sec7}, we use the estimates of Theorem \ref{TH2} to prove Proposition \ref{lengthvsThurs}.

\subsection{Acknowledgments}

I'm deeply grateful to Andrea Seppi for his constant support during the preparation of this work. I would like to thank him for his invaluable comments, remarks, and the  careful readings of this article.\\
Moreover, I would like to thank Sara Maloni for her interest and remarks on this work, which I initially presented in the Geometry seminar at University of Virginia during my stay in Spring 2023.\\
I would also like to thank François Fillastre for pointing out an application of Theorem \ref{TH2} to prove an estimate between the length function and the Thurston norm as stated in Proposition \ref{lengthvsThurs}.

\section{Preliminaries on three-dimensional geometries}\label{section2}
\subsection{Minkowski geometry} 
The Minkowski space is the vector space $\mathbb{R}^3$ endowed with a non degenerate bilinear form of signature $(-,+,+)$, it can be defined as
$$\minko:=\left( \mathbb{R}^3,  \inner{(x_0,x_1,x_2),(y_0,y_1,y_2)}_{1,2}=-x_0y_0+x_1y_1+x_2y_2\right).$$

The group of isometries of $\minko$ that preserve both the orientation and time orientation is identified as:
$$\mathrm{Isom}(\minko)=\mathrm{O}_0(1,2)\ltimes\minko,$$
where $\mathrm{O}(1,2)$ is the group the linear transformations that preserve the Lorentzian form $\inners_{1,2}$, $\mathrm{O}_0(1,2)$ denote the identity component of $\mathrm{O}(1,2)$, $\minko$ acts by translation on itself. In Minkowski space, there are three types of planes $\mathrm{P}$: \textit{spacelike} when the restriction of the Lorentzian metric to $\mathrm{P}$ is positive definite, \textit{timelike} when the restriction is still Lorentzian or \textit{lightlike} when the restriction is a degenerate bilinear form.

We define the hyperbolic plane as the upper connected component of the two-sheeted hyperboloid,
namely:
\begin{equation}\label{hyperboloid}\mathbb{H}^2:=\{(x_0,x_1,x_2)\in \mathbb{R}^{1,2}\mid -x_0^2+x_1^2+x_2^2=-1, \ x_0>0\}.\end{equation} This hyperboloid is a spacelike surface in $\minko$ which is an isometric copy of the hyperbolic plane embedded in Minkowski space. Recall that a surface $\mathrm{S}$ is $\minko$ is said to be \textit{spacelike} if for every $x\in \mathrm{S}$, the tangent plane $\mathrm{T}_x\mathrm{S}$ is a spacelike plane of $\minko$. The group $\isom(\mathbb{H}^2)$ of orientation preserving isometry of $\mathbb{H}^2$ is thus identified with $\mathrm{O}_0(1,2)$. The geodesics of $\mathbb{H}^2$ are of the form $\mathbb{H}^2 \cap x^{\perp}$ for a spacelike vector $x$ (i.e., $\inner{x,x}_{1,2}>0$), where $x^{\perp}$ denotes the orthogonal complement of $x$ with respect to the bilinear form $\inner{\cdot,\cdot}_{1,2}$. Consider the radial projection $\Pi$ defined on $\{(x_0,x_1,x_2)\in\mathbb{R}^3\mid x_0\neq 0\}$ by:
\begin{equation}\label{radial}
\Pi(x_0,x_1,x_2)=\left( \frac{x_1}{x_0},\frac{x_2}{x_0}   \right),\end{equation} 
The projection $\Pi$ identifies the hyperboloid $\mathbb{H}^2$ with the unit disk $\mathbb{D}^2$ which is the \textit{Klein projective model} of the hyperbolic plane. The boundary at infinity $\partial\mathbb{H}^2$ of the hyperbolic plane is then identified with the unit circle $\mathbb{S}^1$. We now move on to recall the definition of the Minkowski cross product.
\begin{defi}
    Let $x$, $y\in \minko$. The \textit{Minkowski cross product} of $x$ and $y$ is the unique vector $x\boxtimes y$ in $\minko$ such that:
    $$\inner{x\boxtimes y,v}_{1,2}=\det(x,y,v),$$
    for all $v\in \minko.$
\end{defi}
The Minkowski cross product gives rise to an isomorphism $\Lambda: \mathbb{R}^{1,2}\to \mathfrak{isom}(\mathbb{H}^2)$ between the Minkowski space $\minko$ and the Lie algebra of $\mathrm{Isom}(\mathbb{H}^2)$. Since $\mathrm{O}_0(1,2)$ is the isometry group of $\mathbb{H}^2$, $\mathfrak{isom}(\mathbb{H}^2)$ is the algebra of skew-symmetric matrices with respect to $\inner{\cdot,\cdot}_{1,2}$. More precisely we have:
\begin{equation}\label{lievskilling}
    \Lambda(x)(y):=y\boxtimes x.\end{equation}
The isomorphism $\Lambda$ is equivariant with respect to the linear action $\mathrm{O}_0(1,2)$ on $\minko$ and the adjoint action of $\mathrm{O}_0(1,2)$ on $\mathfrak{isom}(\mathbb{H}^2)$, namely for all $x\in \mathbb{R}^{2,1}$, $\mathrm{A}\in\mathrm{O}_0(1,2)$, we have:
\begin{equation}\label{adequivariance}
\Lambda(\mathrm{A}\cdot x)=\mathrm{A}\Lambda(x)\mathrm{A}^{-1}.\end{equation}

Recall that the Lie algebra $\mathfrak{isom}(\mathbb{H}^2)$ can be seen as the space of all \textit{Killing vector fields} on $\mathbb{H}^2$, where a Killing field $X$ is by definition a vector field whose flow is a one-parameter group of isometries of $\mathbb{H}^2$. Indeed each $X\in \mathfrak{isom}(\mathbb{H}^2)$ defines a Killing field on $\mathbb{H}^2$ given by:
$$\overline{X}(p)=\frac{d}{dt}\big\lvert_{t=0} (e^{tX}\cdot p)$$ and any Killing field on $\mathbb{H}^2$ is of this form for a unique $X \in \mathfrak{isom}(\mathbb{H}^2)$. A particular class of Killing fields that will play an important role in this paper are those of the form $\Lambda(x)$ for a spacelike vector $x$ (i.e., $\langle x, x \rangle_{1,2} > 0$). We call such a vector field a \textit{hyperbolic Killing field} of $\mathbb{H}^2$. The terminology comes from the following elementary Lemma.

\begin{lemma}\label{hyperbolic_killing_field}
Let $\sigma$ be a spacelike vector in $\minko$. Then the one-parameter family of isometries $e^{t\Lambda(\sigma)}$ is a hyperbolic isometry with an axis $\ell = \mathbb{H}^2 \cap \sigma^\perp$ and a translation length equal to $\vert t\vert \cdot \sqrt{\langle \sigma, \sigma \rangle_{1,2}}$.
\end{lemma}

\begin{proof}
 Let $\sigma\in\minko$ such that $\inner{\sigma,\sigma}_{1,2}>0$. By transitivity of the action of $\isom(\mathbb{H}^2)$ on the space of geodesics, we may assume that $\sigma=(0,0,a)$ where $\lvert a\rvert=\sqrt{\langle \sigma, \sigma \rangle_{1,2}}$. By elementary computation, one can check that:
$$\exp(t\Lambda(\sigma))=\begin{pmatrix}

 \cosh{(-ta)} & \sinh{(-ta)} & 0\\
 \sinh{(-ta)} & \cosh{(-ta)} &0 \\
 0           &     0&       1
\end{pmatrix}.$$ This completes the proof.
 
\end{proof}

We will call the geodesic $\ell=\mathbb{H}^2\cap \sigma^{\bot}$, the \textit{axis} of the hyperbolic Killing field $\Lambda(\sigma).$ We conclude this section with the following lemma, which expresses the Killing field in the Klein model of hyperbolic space.

\begin{lemma}\label{kil_in_the_disc}
Let $\sigma\in\minko$. The Killing field $\Lambda(\sigma)$ has the following expression in the Klein model:
\begin{equation}\label{Killing_Klein}\begin{array}{ccccc}
 &  & \mathbb{D}^2 & \to & \mathbb{R}^3 \\
 & & \eta & \mapsto & \mathrm{d}_{(1,\eta)}\Pi\left((1,\eta)\boxtimes  \sigma\right) \\
\end{array},\end{equation}
where $\Pi$ is the radial projection defined in \eqref{radial}.
\end{lemma}

\begin{proof}
The Killing field $\Lambda(\sigma)$ in the hyperboloid model is given by:
 \begin{equation}\label{8}\begin{array}{ccccc}
 &  & \mathbb{H}^2 & \to & \mathbb{R}^3 \\
 & & y & \mapsto & y\boxtimes \sigma \\
\end{array}.\end{equation}
Now, since the radial projection $\Pi$ identifies the hyperboloid $\mathbb{H}^2$ with the Klein model $\mathbb{D}^2$, the Killing field $\Lambda(\sigma)$ in $\mathbb{D}^2$ is simply the push-forward of the formula \eqref{8} by the radial projection $\Pi$, which gives us the formula \eqref{Killing_Klein}.
\end{proof}

\subsection{Half-pipe geometry as dual of Minkowski Geometry}
In this section, we will present the so called \textit{Half-pipe} space. Following 
\cite{danciger_thesis}, it is defined as
$$\HP :=\{[x_0,x_1,x_2,x_3]\in\mathbb{RP}^3, \ -x_0^2+x_1^2+x_2^2<0\}.$$
The boundary at infinity $\partial\HP$ of $\HP$ is given by: 
$$\partial\HP=\{[x]\in \mathbb{RP}^3, \ -x_0^2+x_1^2+x_2^2=0\}.$$ The Half-pipe space has a natural identification with the dual of Minkowski space, namely the space of spacelike planes of Minkowski space. More precisely, we have the homeomorphism
\begin{equation}\label{hpminko}
    \mathcal{D}:\ \HP\cong \{\mathrm{Spacelike}\ \mathrm{planes}\ \mathrm{in}\ \minko\}
\end{equation}
which associates to each point $[x,t]$ in $\HP$, the spacelike plane of $\minko$ defined as:
 \begin{equation}\label{4}
    \mathrm{P}_{[x,t]}=\{  y \in \minko : \langle x,y \rangle_{1,2} = t         \}
\end{equation}
The homeomorphism $\mathcal{D}$ extends to a homeomorphism between $\partial\HP\setminus[0,0,0,1]$ and the space of \textit{lightlike} planes in Minkowski space $\minko$ using the same formula \eqref{4}. Another interesting model of $\HP$ derived from this duality is given by the diffeomorphism: $\HP\to\mathbb{H}^2\times\mathbb{R}$ defined by: 
\begin{equation}\label{L}
[x,t]\to([x],\mathrm{L}([x,t])),\end{equation}
where $\mathrm{L}([x,t])$ is the \textit{height function}, which is defined as the signed distance of the spacelike plane $\mathrm{P}_{[x,t]}$ to the origin along the future normal direction. It can be checked by elementary computation that: \begin{equation}\label{formule_L}
    \mathrm{L}([x,t])=\frac{t}{\sqrt{-\langle x,x\rangle}}_{1,2}.\end{equation}
We will call \textit{geodesics} (resp. \textit{planes}) of $\HP$ the intersection of lines (resp. planes) of $\mathbb{RP}^3$ with $\HP.$ We will also use the following terminology:
\begin{itemize}
\item[\textbullet] A geodesic in $\HP$ of the form $\{*\}\times \mathbb{R}$ is called a \textit{fiber}.
\item[\textbullet] A geodesic in $\HP$ which is not a fiber is called a \textit{spacelike} geodesic.
\item[\textbullet] A plane in $\HP$ is \textit{spacelike} if it does not contain a fiber.
\end{itemize}
From this, we can define a dual correspondence to the identification \eqref{hpminko} as follows:
\begin{equation}\label{D_star}
  \mathcal{D}^*: \ \minko\cong\{ \mathrm{Spacelike}\ \mathrm{planes}\ \mathrm{in}\ \HP  \}
\end{equation}
which associate to each vector $v$ in $\minko$, the spacelike plane in $\HP$ given by: 
$$\mathrm{P}_v:=\{[x,t]\in\HP \mid \inner{x,v}_{1,2}=t   \}.$$
Now, let us denote by $\isom(\HP)$ the group of transformations given by: 
$$\begin{bmatrix}
\mathrm{A} & 0  \\
v & 1 
\end{bmatrix},$$
where $\mathrm{A}\in \mathrm{O}_{0}(1,2)$ and $v\in \mathbb{R}^2$. Observe that the group $\isom(\HP)$ preserves the orientation of $\HP$ and the oriented fibers. The map $\mathcal{D}$ in \eqref{hpminko} induces an isomorphism between 
$\isom(\mathbb{R}^{1,2})$ and $\isom(\HP)$ given by (See \cite[Section 2.8]{riolo_seppi})

\begin{equation}\label{groupeduality}  
\begin{array}{ccccc}
\mathrm{Is}: &  & \isom(\mathbb{R}^{1,2}) & \to & \isom(\HP) \\
 & & (\mathrm{A},v) & \mapsto & \begin{bmatrix}
\mathrm{A} & 0  \\
^T v\mathrm{J}\mathrm{A} & 1 
\end{bmatrix}, \\
\end{array}
\end{equation}
where $\mathrm{J}=\mathrm{diag}(-1,1,1).$ By the duality approach, one can define the notion of \textit{angle} between spacelike planes of $\HP$:
\begin{prop}\cite[Lemma 4.10]{riolo_seppi}\label{hpangle}
 Let $v_1$, $v_2\in\minko$ and $\mathrm{P}_{v_1}$, $\mathrm{P}_{v_2}$ be their dual spacelike planes in $\HP$. Then $\mathrm{P}_{v_1}$ and $\mathrm{P}_{v_2}$ intersect along a spacelike geodesic in $\HP$ if and only if $v_1-v_2$ is spacelike in $\mathbb{R}^{1,2}$.\\ In that case, we define the angle between $\mathrm{P}_{v_1}$ and $\mathrm{P}_{v_2}$ as the non-negative real number
    $$\theta=\sqrt{\langle v_1-v_2, v_1-v_2        \rangle_{1,2}}.$$
\end{prop}
Now, we move on to describe the $\textit{Klein model}$ of the Half-pipe space which will be useful in this paper. It is defined as the cylinder $\mathbb{D}^2\times\mathbb{R}$ obtained by projecting $\HP$ in the affine chart  $\{x_0\neq 0\}$: \begin{equation}\label{7}
    \begin{array}{ccccc}
&  & \HP & \to & \mathbb{D}^2\times\mathbb{R} \\
 & & [x_0,x_1,x_2,x_3] & \mapsto & (\frac{x_1}{x_0},\frac{x_2}{x_0},\frac{x_3}{x_0}). \\
\end{array}\end{equation}
By the above discussion, it is clear that is this model the correspondence $\eqref{D_star}$ between Minkowski space and spacelike planes of $\HP$ associates to each vector $v\in\minko$, the non vertical plane
\begin{equation}\label{formule_dual_plan}
    \mathrm{P}_v:=\{(\eta,t)\in\mathbb{D}^2\times\mathbb{R} \mid \langle (1,\eta),v  \rangle_{1,2}=t  \}\end{equation} and this plane corresponds to the graph of an affine functions over $\mathbb{D}^2$.
We finish these preliminaries on Half-pipe geometry by a Lemma that describes the action of the isometries of $\HP$ in the Klein model. 
\begin{lemma}\cite[Lemma 2.26]{barbotfillastre}\label{rotationHP}
    Let $(\eta, t)\in \mathbb{D}^2\times\mathbb{R}$, $\mathrm{A}\in \mathrm{O}_{0}(1,2)$ and $v\in \mathbb{R}^{1,2}$. Then the isometry of
Half-pipe space defined by $\mathrm{Is}(\mathrm{A},v)$ acts on the Klein model $\mathbb{D}^2\times\mathbb{R}$ as follows:
$$\mathrm{Is}(\mathrm{A},v)\cdot(\eta,t)=\left( \mathrm{A}\cdot \eta, \frac{t}{-\langle \mathrm{A}\begin{pmatrix}
1 \\
\eta
\end{pmatrix} ,(1,0,0)\rangle_{1,2}}+\langle (1,\mathrm{A}\cdot \eta), v\rangle_{1,2} \right),$$
where $\mathrm{A}\cdot \eta$ is the image of $\eta$ by the isometry
of $\mathbb{D}^2$ induced by $\mathrm{A}.$
\end{lemma}

\subsection{Vector fields on the circle and Half-Pipe geometry}
A fundamental step in the proof of the main theorem involves associating a subset of $\partial\HP$ to any vector field on $\mathbb{S}^1$. Let us consider 
$$N=\{(x,y,z)\in\minko, \ -x^2+y^2+z^2=0\}\setminus\{0\},$$ so that $\mathbb{P}(N)\cong\mathbb{S}^1$. The next Lemma, proved in \cite{BS17}, states that there is a natural identification between vector fields on $\mathbb{S}^1$ and $1$-homogeneous functions on $N$.
\begin{lemma}\label{homogene}\cite[Lemma $2.23$]{BS17}
    There is a $1-$to$-1$ correspondence between vector fields $X$ on $\mathbb{S}^1$ and $1$-homogeneous functions $\Phi:N\to\mathbb{R}$ satisfying the following property: For any $\mathcal{C}^1$ spacelike section $s:\mathbb{S}^1\to N$ of the radial projection $\Pi:N\to\mathbb{S}^1$ and for $v$ the unit tangent vector field to
$s$, then 
\begin{equation}\label{secctionhomogene}
s_*(X(z))=\Phi(s(z))v(s(z))
\end{equation}
\end{lemma}
In this paper, we will mostly work in the Klein model $\mathbb{D}^2\times\mathbb{R}$ of the half-pipe space, so we need this definition.

\begin{defi}\label{supportfunction}
Given a vector field $X$ on $\mathbb{S}^1$ and $\Phi:N\to \mathbb{R}$ a one-homogeneous function as in \eqref{secctionhomogene}, the \textit{support function} of $X$ is the function $\phi_X:\mathbb{S}^1\to \mathbb{R}$ defined by:
$$\phi_X(z)=\Phi(1,z).$$
Therefore, the vector field $X$ can be written as: \begin{equation}\label{supportfunction_vs_vectorfield}
    X(z)=iz\phi_{X}(z)
\end{equation}
\end{defi}
\begin{remark}\label{remark_on_support_function}
Usually, the term "support function" is used to describe the support plane of a convex set in affine space. In our case, if $\phi_X:\mathbb{S}^1\to\mathbb{R}$ is a function, then we can define the \textit{domain of dependence} in Minkowski space as follows: $$\mathcal{D}=\bigcap_{z\in\mathbb{S}^1}\{\sigma\in \minko\mid\ \inner{(1,z),\sigma}_{1,2}<\phi_X(z) \}.$$
The domain of dependence $\mathcal{D}$ is a convex set which is the intersection of half-spaces bounded by the lightlike planes $\{\sigma\in \minko\mid\ \inner{(1,z),\sigma}_{1,2}=\phi_X(z)\}$. For more details on domain of dependence in Minkowski space, we refer the reader to \cite[Section $2.2$]{BS17} 
\end{remark}

Let us now study the support function of Killing vector field, we will show they are affine maps. To prove this, we need the following Lemma.
\begin{lemma}\label{crosspructlemma}
Let $z=(x,y)\in \mathbb{S}^1$ and $v=(-y,x)\in \mathrm{T}_{z}\mathbb{S}^1$ be a unit tangent vector at $z$ oriented in the counterclockwise direction. Then for any $\sigma$ in $\minko$, we have
    $$\inner{(1,z),\sigma}_{1,2}=\inner{(1,z)\boxtimes \sigma,(0,v)}_{1,2}.$$
\end{lemma}
The proof of Lemma \ref{crosspructlemma} is given in the proof of Proposition $5.2$ in \cite{BS17}. Since the lemma will be used later, for the sake of clarity, we will provide its proof.

\begin{proof}[Proof of Lemma \ref{crosspructlemma}]
Let $p=(1,0,0)$ and consider $w=(1,z)-p$ so that $\left(p,w,(0,v)\right)$ is an oriented orthonormal basis. Hence we have 
\begin{equation}\label{cross_product_basis}
    p\boxtimes w=(0,v),\ w\boxtimes (0,v)=-p,\ (0,v)\boxtimes p=w.\end{equation}
Suppose that $\sigma=ap+bw+c(0,v)$, then we get 
\begin{equation}
    (1,z)\boxtimes\sigma=(b-a)(0,v)-cw-cp=(b-a)(0,v)-c(1,z).\end{equation} Since $\inner{(0,v),(1,z)}_{1,2}=0$ then $\inner{(1,z)\boxtimes\sigma,(0,v)}_{1,2}=b-a$, on the other hand 
$$\inner{(1,z),\sigma}_{1,2}=\inner{p+w,ap+bw+c(0,v)}_{1,2}=b-a,$$ this completes the proof.
\end{proof}By an elementary computation, we have 
\begin{equation}\label{differentialformula}
    \begin{array}{ccccc}
\mathrm{d}\Pi_{(1,x,y)} & : & \mathrm{T}_{(1,x,y)} N & \to \mathbb{R}^2&  \\
 & &v  & \mapsto & \inner{v,(0,-y,x)}_{1,2}(-y,x).\\
\end{array}\end{equation}where $\Pi$ is the radial projection defined in \eqref{radial}. 
This leads us to the following corollary.
\begin{cor}\label{support_function_of_killing} 
Let $\sigma\in \mathbb{R}^{1,2}$. The support function $\phi_{\Lambda(\sigma)}:\mathbb{S}^1\to \mathbb{R}$ of the Killing vector field $\Lambda(\sigma)$ is given by: 
$$\phi_{\Lambda(\sigma)}(z)=\inner{(1,z),\sigma}_{1,2}.$$
\end{cor}
\begin{proof}
Recall that from Lemma \ref{kil_in_the_disc}, the Killing vector field $\Lambda(\sigma)$ can be written in the Klein model as: $$z\to\mathrm{d}_{(1,z)}\Pi\left((1,z)\boxtimes  \sigma\right).$$ Let $z=(x,y)$, according to definition of support function \eqref{supportfunction_vs_vectorfield}, we need to show that:
 $$\mathrm{d}_{(1,z)}\Pi\left((1,z)\boxtimes  \sigma\right)=\inner{(1,z),\sigma}_{1,2}(-y,x).$$
 Since $(1,z)\in N$ and $(1,z)\boxtimes \sigma\in \mathrm{T}_{(1,z)}N$, then by formula \eqref{differentialformula}, we have 
 $$\mathrm{d}_{(1,z)}\Pi\left((1,z)\boxtimes \sigma\right)=\inner{(1,z)\boxtimes \sigma,(0-y,x)}_{1,2}(-y,x).$$
Therefore, by Lemma \ref{crosspructlemma} we get \begin{equation}\label{pi(z,1)}
    \mathrm{d}_{(1,z)}\Pi\left((1,z)\boxtimes  \sigma\right)=\inner{(1,z),\sigma}_{1,2}(-y,x),\end{equation}
and this concludes the proof.\end{proof}
\begin{remark}\label{remarqueabc}
Let $a$, $b$, and $c$ be three real numbers. We can deduce from Corollary \ref{support_function_of_killing} the existence of a Killing vector field $X$ for which the support function $\phi_X$ satisfies the following:
\begin{equation}\label{abc}
\phi_X(1,0)=a, \ \phi_X(0,1)=b, \ \phi_X(-1,0)=c.
\end{equation}
To prove this, observe that the three points $(1,0,a)$, $(0,1,b)$, and $(-1,0,c)$ span a unique spacelike plane in $\mathbb{D}^2\times\mathbb{R}$. Hence, this plane is of the form $\mathrm{P}_\sigma$ for some $\sigma\in\minko$, where we recall that from \eqref{formule_dual_plan}, we have:
$$\mathrm{P}_\sigma=\{(\eta,t)\in\mathbb{D}^2\times\mathbb{R}\mid \inner{(1,\eta),\sigma}_{1,2}=t\}.$$
Now, take $X$ to be the Killing vector field $\Lambda(\sigma)$. From Corollary \ref{support_function_of_killing}, the support function of $X$ is given by: $$\phi_X(z)=\inner{(1,z),\sigma}_{1,2}$$ and by the construction of the plane $\mathrm{P}_\sigma$, the function $\phi_X$ satisfies the relations \eqref{abc}. 
\end{remark}
In the next lemma, we will study the behavior of a support function under the action of $\mathrm{O}_0(1,2)\ltimes\minko$ on vector fields of the circle. Recall that the linear part acts on the space of vector fields by pushforward, and the translation part acts by adding a Killing vector field.
\begin{lemma}\label{equii}
    Let $X$ be a vector field of $\mathbb{S}^1$ and $\phi_X$ be its support function. Let $\mathrm{A}\in\mathrm{O}_0(1,2)$ and $\sigma\in\minko$. The support function $\phi_{\mathrm{A}_*X+\Lambda(\sigma)}$ of the vector field $\mathrm{A}_*X+\Lambda(\sigma)$ satisfies:\begin{equation}
          \mathrm{gr}(\phi_{\mathrm{A}_*X+\Lambda(\sigma)})=\mathrm{Is}(\mathrm{A},\sigma)\mathrm{gr}(\phi_X).
    \end{equation}
    
\end{lemma}
\begin{proof}
 First we deduce from Corollary \ref{support_function_of_killing}, that the support function of $X+\Lambda(\sigma)$ is given by: 
 $$\phi_{X+\Lambda(\sigma)}(z)=\phi_X(z)+\inner{(1,z),\sigma}_{1,2}$$ and hence by Lemma \ref{rotationHP}:
  \begin{equation}\label{20}
        \mathrm{gr}(\phi_{X+\Lambda(\sigma)})=\mathrm{Is}(0,\sigma)\mathrm{gr}(\phi_X).
    \end{equation} 
Next, by  an elementary computation, one may deduce from Lemma \ref{homogene} that for a given vector field $X$ and $\Phi$ as in \eqref{secctionhomogene}, and for an isometry $\mathrm{A}$ of the hyperbolic plane, the function $\Phi\circ \mathrm{A}^{-1}:N\to\mathbb{R}$ is the $1-$homogeneous function corresponding to the vector field $\mathrm{A}_*X.$ Hence, the support function of $\mathrm{A}_*X$ is then given by: 
$$
    \phi_{\mathrm{A}_*X}=\Phi\circ \mathrm{A}^{-1}(1,z)=\Phi\left(-\langle \mathrm{A}^{-1}\begin{pmatrix}
1 \\
z 
\end{pmatrix} ,(1,0,0)\rangle_{1,2}  \begin{pmatrix}
1 \\
\mathrm{A}^{-1}\cdot z 
\end{pmatrix}  \right).$$
So, by homogeneity of $\Phi$, we obtain

\begin{equation}\label{17}\phi_{\mathrm{A}_*X}(z)=-\langle \mathrm{A}^{-1}\begin{pmatrix}
1 \\
z 
\end{pmatrix} ,(1,0,0)\rangle_{1,2} \phi_X(\mathrm{A}^{-1}\cdot z).\end{equation} Therefore using Lemma \ref{rotationHP}, we conclude that:
\begin{equation}\label{22}
    \mathrm{gr}(\phi_{\mathrm{A}_*X})=\mathrm{Is}(\mathrm{A},0)\mathrm{gr}(\phi_X).
\end{equation}
The proof of Lemma is then completed by combining equations \eqref{20} and \eqref{22}.
\end{proof}

Lemma \ref{homogene} allows us to view each vector field $X$ of $\mathbb{S}^1$ as a subset of $\partial\HP$. Indeed, by one homogeneity of the map $\Phi:N\to \mathbb{R}$ as in \eqref{secctionhomogene}, one may define the \textit{graph} of $X$ as the subset of $\partial\HP$ given by
\begin{equation}
    \mathcal{GR}(X):=\{[(x_0,x_1,x_2, \Phi(x_0,x_1,x_2))] \mid (x_0,x_1,x_2)\in N  \}
\end{equation}
Now, we want to take the convex hull of $\mathcal{GR}(X)$ in projective space. Recall that the convex hull of a set in projective space can be defined if it is contained in an affine chart.

\begin{lemma}
Let $X$ be a vector field on the circle, and let $\Phi:N\to\mathbb{R}$ be the one homogeneous function as in \eqref{secctionhomogene}. Then:
\begin{enumerate}
\item The graph $\mathcal{GR}(X)$ of $X$ is contained in an affine chart.
\item Let $\mathcal{C}(X)$ be the convex hull of $\mathcal{GR}(X)$. Then $\mathcal{C}(X)\subset \overline{\HP}\setminus\{[0,0,0,1]\}$.
\end{enumerate}
\end{lemma}

\begin{proof}
It is enough to remark that for $(x_0,x_1,x_2)\in N$, $x_0\neq 0$, thus $N\subset \{x_0\neq 0\}$. Hence, $\mathcal{GR}(X)$ is contained in the affine chart $\{x_0\neq 0\}$. To prove the second item, note that $\overline{\HP}\setminus\{[0,0,0,1]\}$ is contained in the affine chart $\{x_0\neq 0\}$; moreover, it is convex since it is equal to $\overline{\mathbb{D}^2}\times\mathbb{R}$.
Now since $\mathcal{GR}(X)$ is contained in $\overline{\HP}\setminus{[0,0,0,1]}$ and $C(X)$ is the smallest convex set containing $\mathcal{GR}(X)$, then $\mathcal{C}(X)\subset\overline{\HP}\setminus\{[0,0,0,1]\}$, this completes the proof.
\end{proof}
Notice that the graph of the vector field $X$ corresponds to the graph of the support function $\phi_X$ in the affine chart $\{x_0\neq0\}$.

We denote by $\partial\mathcal{C}(X)$ the boundary of $\mathcal{C}(X)$ in $\overline{\HP}$. If $X$ is continuous, then the graph of $\phi_X$ forms a Jordan curve. By the Jordan-Brouwer separation theorem, $\partial\mathcal{C}(X)\setminus \mathcal{GR}(X)$ has two connected components, which we denote as $\partial_+\mathcal{C}(X)$ and $\partial_-\mathcal{C}(X)$ , respectively. Following the work of \cite{barbotfillastre}, we now define the \textit{width} of a vector field

\begin{defi}\label{width}
 Let $X$ be a continuous vector field of $\mathbb{S}^1$. Then the \textit{width} of $X$
is defined as:
$$w(X):=\underset{[(x,s)] \in \partial_- \mathcal{C}(X), \ [(x,t)]\in \partial_+\mathcal{C}(X)}{\sup} \vert \mathrm{L}([x,t])-\mathrm{L}([x,s])\vert  \in [0,+\infty],$$ where $\mathrm{L}:\HP\to\mathbb{R}$ is height function defined in \eqref{formule_L}
\end{defi}
The quantity $\mathrm{L}([x,t])-\mathrm{L}([x,s])$ can be interpreted as the timelike distance between the spacelike planes $\mathrm{P}_{[x,t]}$ and $\mathrm{P}_{[x,s]}$ of Minkowski space dual to $[x,t]$ and $[x,s]$ respectively (see Equation \eqref{4}). The next Lemma asserts that the width is an invariant under isometries of $\HP$.

\begin{lemma}\label{width_equivariant}
    Let $X$ be a continuous vector field. Then for any $\mathrm{A}\in\mathrm{O}_0(1,2)$ and $v\in\minko$, we have 
    $$w(\mathrm{A}_*X+\Lambda(v))=w(X).$$
\end{lemma}

\begin{proof}
Let $\Phi:N\to\mathbb{R}$ be a one-homogeneous function associated to the vector field $X$ as in \eqref{secctionhomogene}. Then, by Corollary \ref{support_function_of_killing} and Lemma \ref{equii}, the function $\Phi\circ\mathrm{A}^{-1}+\inner{\cdot,v}_{1,2}$ is the one-homogeneous function associated with the vector field $\mathrm{A}_*X+\Lambda(v)$. Hence,\begin{align*}
\mathcal{GR}(\mathrm{A}_*X+\Lambda(v))&=\{[(x_0,x_1,x_2, \Phi\circ\mathrm{A}^{-1}(x_0,x_1,x_2)+\inner{(x_0,x_1,x_2),v}_{1,2})] \mid (x_0,x_1,x_2)\in N  \}\\
&=\{[\mathrm{A}\cdot (x_0,x_1,x_2), \Phi(x_0,x_1,x_2)+\inner{\mathrm{A}\cdot(x_0,x_1,x_2),v}_{1,2})] \mid (x_0,x_1,x_2)\in N  \}.
\end{align*}
Hence, we deduce that $\mathcal{GR}(\mathrm{A}_*X+\Lambda(v))=\mathrm{Is}(\mathrm{A},v)\mathcal{GR}(X)$. Since isometries of $\HP$ preserve convex sets, we deduce that:
$$\partial_\pm\mathcal{C}(\mathrm{A}_*X+\Lambda(v))=\mathrm{Is}(\mathrm{A},v)\partial_\pm\mathcal{C}(X).$$
To complete the proof, we need to show that the \textit{distance along the fibers} is preserved under the isometry group of $\HP$, namely:
\begin{equation}\label{equivariant_length}
\mathrm{L}(\mathrm{Is}(\mathrm{A},v)[x,t])-\mathrm{L}(\mathrm{Is}(\mathrm{A},v)[x,s])= \mathrm{L}([x,t])-\mathrm{L}([x,s]),
\end{equation}
and this follows directly from Lemma \ref{rotationHP} and the formula in \eqref{L} of $\mathrm{L}$.
\end{proof}

Now, we will explain how the boundary $\partial_\pm\mathcal{C}(X)$ can be seen as a graph of a real function over $\mathbb{D}^2$. To do this, we need to recall some notions in convex analysis. For more detailed discussion, readers can refer to \cite{Rockconvex}. Given a convex (resp. concave) function $u:\mathbb{D}^2\to \mathbb{R}$, the \textit{boundary value} of $u$ is the function on $\mathbb{S}^1$ whose value at $z\in\mathbb{S}^1$ is given by: \begin{equation}\label{boundaryvalue}
    u(z) = \lim_{s \to 0^+}u((1-s)z+sx),\ \mathrm{for} \ \mathrm{any}\ x\in\mathbb{D}^2.\end{equation}
The value of the limit \eqref{boundaryvalue} is independent of the choice of $x\in \mathbb{D}^2$ (see \cite[Section 4]{ConvexAnal}). In fact if $u:\mathbb{D}^2\to\mathbb{R}$ is a convex (resp. concave) function then the function $u:\overline{\mathbb{D}^2}\to\mathbb{R}$ obtained by extending $u$ to $\mathbb{S}^1$ by \eqref{boundaryvalue} is a lower semicontinuous (resp. upper semicontinuous) function. For any function $\phi: \mathbb{S}^1 \to \mathbb{R}$, define two functions $\phi^-, \phi^+: \overline{\mathbb{D}^2} \to \mathbb{R}$ as follows:
\begin{equation}\label{phi-}
\phi^-(z)=\sup\{ a(z)\mid \ a:\mathbb{R}^2\to\mathbb{R} \ \mathrm{is}\ \mathrm{an} \ \mathrm{affine}\ \mathrm{function}\ \mathrm{with}  \ a_{\mid\mathbb{S}^1}\leq\phi   \},\end{equation}
\begin{equation}\label{phi+}
\phi^+(z)=\inf\{ a(z)\mid \ a:\mathbb{R}^2\to\mathbb{R} \ \mathrm{is}\ \mathrm{an} \ \mathrm{affine}\ \mathrm{function}\ \mathrm{with}  \ \phi\leq a_{\mid\mathbb{S}^1}   \}.\end{equation}
We need the following properties of the maps $\phi^-$ and $\phi^+$:\label{page 8}
\begin{enumerate}
    \item[(P1)]\label{P1} We have $\phi^-\leq\phi^+$, moreover, $\phi^-$ (resp. $\phi^+$)  is lower semicontinuous and convex (resp. upper semicontinuous and concave). As a consequence, $\phi^\pm$ are continuous in the unit disc $\mathbb{D}^2$.
\item[(P2)]\label{P2} If $\phi$ is lower semicontinuous (resp. upper semicontinuous), then the boundary value of $\phi^-$ (resp. $\phi^+$) is equal to $\phi$ on $\mathbb{S}^1$ (see \cite[Lemma 4.6]{ConvexAnal}).
\item[(P3)]\label{P3} If $\phi: \mathbb{S}^1\to\mathbb{R}$ is a continuous function, then $\phi^-$ and $\phi^+$ are continuous in $\overline{\mathbb{D}^2}$, see \cite[Proposition 4.2]{ConvexAnal}.
\item[(P4)]\label{P4} If $u:\mathbb{D}^2\to\mathbb{R}$ is a convex (resp. concave) function with boundary value $u |_{\mathbb{S}^1}\leq \phi$ (resp. $\phi\leq u |_{\mathbb{S}^1} $), then $u\leq\phi^-$ (resp. $\phi^+\leq u$) in $\mathbb{D}^2$, see \cite[Corollary 4.5]{ConvexAnal}.
\item[(P5)]\label{P5} If $X$ is a continuous vector field on $\mathbb{S}^1$, and $\phi_X$ is its support function (which is continuous), then the boundary of $\partial\mathcal{C}(X)$ is the union of the graphs $\phi_X^-$ and $\phi_X^+$, see \cite[Lemma 2.41]{barbotfillastre}.
\end{enumerate}
Now, we consider the epigraphs of $\phi^+$ and $\phi^-$ defined as follows: $$\mathrm{epi}^+(\phi^-) = \{(\eta,t)\in\overline{\mathbb{D}^2}\times\mathbb{R} \mid t\geq\phi^-(\eta)\} \ \mathrm{and}\  \mathrm{epi}^-(\phi^+) = \{(\eta,t)\in\overline{\mathbb{D}^2}\times\mathbb{R} \mid t\leq\phi^+(\eta)\}.$$ Then the first property $(\mathrm{P}1)$ on $\phi^-$ and $\phi^+$ is equivalent to saying that $\mathrm{epi}^+(\phi^-)$ and $\mathrm{epi}^-(\phi^+)$ are closed and convex.

Based on the last property $(\mathrm{P}5)$, we will always assume that for a continuous vector field $X$ on $\mathbb{S}^1$, and up to changing the labeling of $\partial_\pm\mathcal{C}(X)$, the following holds: $$\partial_{-}\mathcal{C}(X)=\mathrm{gr}(\phi_X^- |_{\mathbb{D}^2}) \ \mathrm{and}\  \partial_{+}\mathcal{C}(X)=\mathrm{gr}(\phi_X^+ |_{\mathbb{D}^2}).$$
And hence the width of $X$ can be written as:
\begin{equation}\label{width_with_function}
    w(X):=\underset{\eta\in\mathbb{D}^2}{\sup} \left( \mathrm{L}(\eta,\phi_X^+(\eta))-\mathrm{L}(\eta,\phi_X^{-}(\eta))\right).\end{equation}

We now proceed to define the notion of upper and lower semi-continuous vector fields on the circle. This definition will, a priori, depend on the orientation of the circle. Let us fix the standard orientation of $\mathbb{S}^1$. Suppose $V$ is a smooth section of $\mathrm{T}\mathbb{S}^1$ that is positively oriented, meaning that $\mathrm{det}(z,V(z))>0$. Let $X$ be a vector field on $\mathbb{S}^1$. Then, there exists a function $f:\mathbb{S}^1\to\mathbb{R}$ such that
\begin{equation}\label{XfV}
X=fV.
\end{equation}
We then define the following:
\begin{defi}
Given a vector field $X$ of $\mathbb{S}^1$ and a function $f:\mathbb{S}^1\to\mathbb{R}$ as in \eqref{XfV}. Then we say that:
\begin{itemize}
    \item $X$ is \textit{counter-clockwise lower semicontinuous} if $f$ is a lower semicontinuous function.
\item $X$ is \textit{counter-clockwise upper semicontinuous} if $f$ is a upper semicontinuous function.
\end{itemize}
\end{defi}

The definitions of counter-clockwise lower and upper semi-continuous vector fields do not depend on the choice of the positively oriented smooth section $V$. Therefore, if we take $V(z)=iz$, which is a positively oriented section of $\mathrm{T}\mathbb{S}^1$, then counter-clockwise lower/upper semicontinuous vector field $X$ is equivalent to the fact that the support function $\phi_X:\mathbb{S}^1\to\mathbb{R}$ is lower/upper semicontinuous. In the remainder of the paper, we will omit the word "counter-clockwise" and simply refer to such vector fields as lower/upper semi-continuous vector fields.

We conclude this section by introducing two types of extensions of vector fields from the disk to the circle, which are necessary for our purposes.

\begin{defi}\label{2.17}
Let $X$ be a vector field on $\mathbb{S}^1$, and let $\hat{X}$ be a vector field on $\mathbb{D}^2$. We define the following:
    \begin{itemize}
        \item $\hat{X}$ \textit{extends continuously} to $X$, if for any sequence $(\eta_n)_{n\in\mathbb{N}}$ in $\mathbb{D}^2$ converging to $\eta_{\infty}$, we have
        $$\hat{X}(\eta_n)\to X(\eta_{\infty}).$$
\item $\hat{X}$ \textit{extends radially} to $X$ if $\hat{X}$ extends to $X$ along line segments in $\mathbb{D}^2$, meaning that for each $z\in\mathbb{S}^1$, the following limit holds:
\begin{equation}\label{extendfunction}
    \lim_{s \to 0^+}\hat{X}((1-s)z+sx)=X(z),\ \mathrm{for} \ \mathrm{any}\ x\in\mathbb{D}^2.\end{equation} 
       
    \end{itemize}
\end{defi}
Note that in Definition \ref{2.17}, we don't assume that $\hat{X}$ is continuous on the disk $\mathbb{D}^2$. As we will see later, in Theorems \ref{TH1} and \ref{THrad}, the infinitesimal earthquakes are not continuous in $\mathbb{D}^2$, but they extend (continuously or radially) to the circle.

\begin{remark}\label{derniere_remark}
 Clearly, if $\hat{X}$ extends continuously to $X$, then $\hat{X}$ extends radially to $X$. However, the converse does not hold in general. In fact, if $\hat{X}$ extends radially to $X$ and $X$ is not continuous at $z_0\in \mathbb{S}^1$, then $\hat{X}$ does not extend continuously on $z_0$. Indeed, without loss of generality, assume that $X(z_0)=0$, then since $X$ is not continuous at $z_0$, then there is a sequence of points $(z_n)_{n\in \mathbb{N}}$ in $\mathbb{S}^1$ such that: 
 \begin{equation}\label{epsilon}
     \lVert X(z_n)\rVert_{euc}\geq \epsilon,
 \end{equation} for some $\epsilon>0$, where $\lVert \cdot\rVert_{euc}$ denotes the Euclidean norm in $\mathbb{R}^2$. Let $\ell_n$ denote the segment of extremity $(0,0)$ and $z_n$. By continuity of $\hat{X}$ on $\ell_n$, we conclude that for each $n\in \mathbb{N}$, there exists $\delta_n>0$ converging to zero such that: 
 \begin{equation}\label{equationradiallyyy}
     \mathrm{If} \ \eta\in\ell_n\ \mathrm{and}\ \lVert\eta-z_n\rVert_{euc}\leq\delta_n, \ \mathrm{then}\ \lVert\hat{X}(\eta)-X(z_n)\rVert_{euc} <\frac{\epsilon}{2}.
 \end{equation}
 Now take $x_n=(1-\delta_n)z_n\in \mathbb{D}^2$ which lies on the segment $\ell_n$ and satisfies $\lVert\eta-z_n\rVert_{euc}\leq\delta_n$. Thus by \eqref{equationradiallyyy}, we have: $$\lVert \hat{X}(x_n)-X(z_n)\rVert_{euc}\leq\frac{\epsilon}{2}.$$ Therefore, based on \eqref{epsilon}, we deduce that: $$\lVert\hat{X}(x_n)\rVert_{euc}\geq \frac{\epsilon}{2}.$$ Hence, $\lim\limits_{n \rightarrow +\infty} \lVert \hat{X}(x_n)\rVert_{euc}\neq 0$ and so $\hat{X}$ does not extend continuously on $z_0$. 
\end{remark}

\begin{example}
Let $X$ be a lower semicontinuous vector field on the circle, and consider $\phi_X:\mathbb{S}^1\to \mathbb{R}$ its support function. As example of vector field $\hat{X}$ that extends radially to $X$, one can take $\hat{X}(z)=iz\phi_X^-(z)$. It is easy to see $\hat{X}$ extends radially to $X$, since $\phi_X^-$ extends radially to $\phi_X$ (see Property $(\mathrm{P}2)$ in \ref{P2}).\end{example}

\section{Earthquakes and infinitesimal earthquakes}\label{novsec}

\subsection{Earthquake in hyperbolic plane}
In this section, we will recall the notion of earthquake map
of $\mathbb{H}^2$ and introduce the infinitesimal earthquake on $\mathbb{H}^2$. First we need to introduce the notion of measured geodesic lamination of $\mathbb{H}^2$. Recall that each oriented geodesic in $\mathbb{H}^2$ is uniquely determined by the pair of its endpoints
on $\mathbb{S}^1$, the initial point and the end point. Then the space $\mathcal{G}$ of unoriented geodesics in $\mathbb{H}^2$ is homeomorphic to $(\mathbb{S}^1\times\mathbb{S}^1\setminus \mathrm{diag})/\sim$, where $(a, b) \sim (b, a)$ and $\mathrm{diag}=\{ (a,a) \mid a\in\mathbb{S}^1  \}$.

\begin{defi}
A \textit{geodesic lamination} on $\mathbb{H}^2$ is a closed subset of $\mathcal{G}$ such that its elements
are pairwise disjoint complete geodesics of $\mathbb{H}^2$ . A \textit{measured geodesic lamination} is a locally finite
Borel measure on $\mathcal{G}$ such that its support is a geodesic lamination.
\end{defi}
For a measured geodesic lamination $\lambda$ of $\mathbb{H}^2$, we denote by $\vert \lambda\vert$ the \textit{support} of $\lambda$. The geodesics in $\vert \lambda\vert$ are called \textit{leaves}. The connected components of the complement $\mathbb{H}^2\setminus\vert\lambda\vert$ are called gaps. The \textit{strata} of $\lambda$ are the leaves and the gaps. It is worth to remark that each measured geodesic lamination $\lambda$ induces a transverse measure to its support. Namely an assignment of a positive Borel measure to each closed arc $I$ whose support is $\vert\lambda\vert\cap I$. More precisely, the measure of $I$ that we keep denoted by $\lambda(I)$ is the $\lambda$-measure of the set of geodesic in $\mathcal{G}$ which intersect the arc $I$. We have also the converse construction, namely the assignment of a positive Borel measure on each closed arc transverse to a geodesic lamination $\vert\lambda\vert$ and invariant under homotopy with respect to $\vert\lambda\vert$, gives rise to a measured lamination $\lambda$ whose support is $\vert\lambda\vert$ (see \cite[Section 1]{Bonahon_holder} for more details). Denote by $\mathcal{ML}(\mathbb{H}^2)$ the set of measured geodesic laminations of $\mathbb{H}^2$. Then we have:
\begin{defi}\label{thursnorm2}
Let $\lambda$ be a measured geodesic lamination of $\mathbb{H}^2$, then the \textit{Thurston norm} of $\lambda$ is defined as:
$$     \lVert \lambda\rVert_{\mathrm{Th}}:=\underset{I \ \ \ \ \ \  }{\sup\lambda(I)},          $$
where $I$ varies over all geodesic segments of length $1$ transverse to the support of the 
geodesic lamination $\lambda.$ 
\end{defi}
A measured geodesic lamination is \textit{bounded} if its Thurston norm is finite. We denote the space of bounded measured geodesic laminations by $\mathcal{ML}_b(\mathbb{H}^2).$

\begin{defi}\label{earthquake}
    A left (resp. right) earthquake $\mathrm{E}$ along a geodesic lamination $\lambda$ is a bijective map $\mathrm{E} : \mathbb{H}^2 \to\mathbb{H}^2$ such that for any stratum $S$ of $\lambda$, the restriction $\mathrm{E}_{S}$ of $\mathrm{E}$ to $S$ is equal to the restriction of
an isometry of $\mathbb{H}^2$, and 
for any two strata $S$ and $S'$ of $\lambda$, the \textit{comparison map}
$$\mathrm{Comp}(S,S'):=(\mathrm{E}|_{S})^{-1}\circ \mathrm{E}|_{S'}$$ is the restriction on an isometry $\gamma$ of $\mathbb{H}^2$, such that:
\begin{itemize}
\item $\gamma$ is different from the identity, unless possibly when one of the two strata $S$ and $S'$ is contained in the closure of the other;
\item 
when it is not the identity, $\gamma$ is a hyperbolic isometry whose axis $\ell$ weakly separates $S$ and $S'$;
\item furthermore, $\gamma$
translates to the left (resp. right), seen from $S$ to $S'$. 
\end{itemize}
\end{defi}
Let us explain the meaning of the last condition. Consider a smooth arc $c:[0,1]\to\mathbb{H}^2$ such that $c(0)\in S$ and $c(1)\in S^{'}$. Suppose also that the image of $c$ intersects $\ell$ exactly at one point, denoted as $x_0=c(t_0)$. Let $v=c'(t_0)\in \mathrm{T}_{x_0} \mathbb{H}^2$ be the tangent vector at the intersection point. Let $w \in\mathrm{T}_{x_0} \mathbb{H}^2$ be a vector tangent to the geodesic $\ell$ pointing towards $\gamma(x_0)$. Then we say that $\gamma$ translates to the left (resp. right) when seen from $S$ to $S^{'}$ if the vectors $v$ and $w$ form a positive (resp. negative) basis of $\mathrm{T}_{x_0} \mathbb{H}^2$, with respect to the standard orientation of $\mathbb{H}^2$.\\
Given an earthquake $\mathrm{E}$ of $\mathbb{H}^2$ along a geodesic lamination $\lambda$, one can construct a measured lamination supported on $\lambda$ called \textit{earthquake measure} associated to $\mathrm{E}$, we keep denoting such measure by $\lambda$. This measures the amount of shearing along the support of the earthquake along the support of $\lambda$. For example if the support of $\lambda$ consists of finite number of geodesics, then for each closed oriented arc $I$ transverse to $\lambda$, let $S_1,\cdots S_k$ denote the consecutive strata of $\lambda$ intersecting $I$. The $\lambda$-measure of $I$ is given by: 
$$\sum_{i=1}^{k-1}\mathrm{T}(\mathrm{Comp}(S_i,S_{i+1}))$$
where $\mathrm{T}(\gamma)$ denotes the \textit{signed translation length}
of the hyperbolic isometry $\gamma$. The reader can consult \cite{Thurston,miyachisaric} for the general construction of earthquake measure. It turns out that an earthquake measure $\lambda$ uniquely determines an earthquake map denoted by $\mathrm{E}^{\lambda}:\mathbb{H}^2\to\mathbb{H}^2$ up to post-composition by an isometry of $\mathbb{H}^2$, see \cite[Proposition 1.6.1: Metering earthquakes]{Thurston}.

Thurston \cite{Thurston} proved that, for a given left or right earthquake $\mathrm{E}$, even if $\mathrm{E}$ itself may not be continuous on $\mathbb{H}^2$, $\mathrm{E}$ extends uniquely
to an orientation-preserving homeomorphism of $\mathbb{S}^1$ which is the ideal boundary of the hyperbolic plane $\mathbb{H}^2$. We denote such extension by $\mathrm{E}|_{\mathbb{S}^1}$. Moreover Thurston proved the converse statement.
\begin{theorem}\cite[Geology is transitive]{Thurston}\label{thurs}
    Let $\Phi:\mathbb{S}^1\to\mathbb{S}^1$ be an orientation-preserving homeomorphism. Then there exists a left (resp. a right) earthquake map $\mathrm{E}$ of $\mathbb{H}^2$ along a geodesic lamination $\lambda$, which extends continuously to $\Phi$ on $\mathbb{S}^1.$ Moreover the earthquake is unique, except that there is a range of choices for the earthquake on any leaf of $\lambda$ where $\mathrm{E}$ is discontinuous. 
\end{theorem}
The previous discussion allows us to define an injective map:
\begin{equation}\label{earthqaukemeasure}
\mathrm{EM}: \mathrm{Isom}(\mathbb{H}^2)\backslash \mathrm{Homeo}_+(\mathbb{S}^1)\to\mathcal{ML}(\mathbb{H}^2)\end{equation}
which associates to each class of orientation preserving homeomorphism $\Phi$ of $\mathbb{S}^1$ the earthquake measure $\lambda$ such that $\mathrm{E}^{\lambda}|_{\mathbb{S}^1}=\Phi.$ The map $\mathrm{EM}$ is not surjective, namely one can find a measured geodesic lamination which does not arise as an earthquake measure, see example in \cite{Thurston}. Nevertheless the set of bounded lamination is in the image of the map $\mathrm{EM}$. To explain this we need to recall some terminology. Recall that given an orientation-preserving homeomorphism $\Phi: \mathbb{RP}^1\to \mathbb{RP}^1$ (where $\mathbb{RP}^1\cong\mathbb{S}^1$), the \textit{cross-ratio norm} is defined as: 
 $$\lVert  \Phi \rVert_{cr}=\underset{\mathrm{cr}(Q)=1\ \ \ \ \ \ \ \ \ \ \ \ \ }{\sup \vert\ln{\mathrm{cr}\left(\Phi(Q)\right)}    \vert }\in[0,+\infty],$$ where $Q=[a,b,c,d]$ is a quadruple of points on $\mathbb{RP}^1$ and $$\mathrm{cr}(Q)=\frac{  (b-a)(d-c) }{ (c-b)(d-a)   }$$ is the cross ratio of $Q$.
\begin{defi}
    Let $\Phi:\mathbb{S}^1\to\mathbb{S}^1$ be an orientation preserving homeomorphism, then we say that $\Phi$ is \textit{quasisymmetric} if the cross-ratio norm $\lVert \Phi\rVert_{cr}$ of $\Phi$ is finite.
\end{defi}
Denote $\mathcal{QS}(\mathbb{S}^1)$ the subgroup of $\mathrm{Homeo}_+(\mathbb{S}^1)$ consisting of quasisymmetric homeomorphism of $\mathbb{S}^1$, then we define the \textit{universal Teichm\"uller space} as the space of quasisymmetric homeomorphisms of the circle up to post-composition with an isometry of $\mathbb{H}^2$:
$$\mathcal{T}(\mathbb{H}^2)=\mathrm{Isom}(\mathbb{H}^2)\backslash \mathcal{QS}(\mathbb{S}^1).$$ We then have the following Theorem anticipated by Thurston which characterises quasisymmetric homeomorphisms in term of earthquake measures. 
\begin{theorem}\cite{GHL,Saric}\label{boundedinfi}
    The map $\mathrm{EM}$ \eqref{earthqaukemeasure} is a bijection form the universal Teichm\"uller space $\mathcal{T}(\mathbb{H}^2)$ to the set of bounded measured laminations $\mathcal{ML}_b(\mathbb{H}^2)$.
\end{theorem}

\subsection{Infinitesimal earthquake} 
In this section we are going to discuss the notion of \textit{infinitesimal earthquake}. Let $X$ be a vector field on $\mathbb{S}^1$, then the cross-ratio distortion norm of $X$ is defined as 
$$\lVert  X \rVert_{cr}=\underset{\mathrm{cr}(Q)=1\ \ \ \ \ \ \ \ \ \ \ \ \ }{\sup \vert X[Q]\vert} \in[0,+\infty],$$ where 
$$X[Q]=\frac{X(b)-X(a)}{b-a}-\frac{X(c)-X(b)}{c-b}+\frac{X(d)-X(c)}{d-c}-\frac{X(a)-X(d)}{a-d},$$
for $Q=[a,b,c,d]$. For instance, one can check that $X$ is an extension to $\mathbb{S}^1$ of a Killing vector field of $\mathbb{H}^2$ if and only if $\lVert X\rVert_{cr}=0.$
\begin{defi}
   A continuous vector field $X$ of $\mathbb{S}^1$ is \textit{Zygmund} if and only if the cross-ratio distortion norm of $X$ is finite.
   In that case, we say that the support function $\phi_X:\mathbb{S}^1\to\mathbb{R}$ is in the Zygmund class.
\end{defi}
It has been proved that a vector field $X$ of $\mathbb{S}^1$ is Zygmund if and only if there exist a smooth path $t\to f_t$ of quasisymmetric maps such that $f_0=\mathrm{Id}$ and $\frac{d}{dt}\big\lvert_{t=0} f_t=X$ (see \cite[Section 16.6, Theorem 5]{Gardiner1999QuasiconformalTT}). In fact, Gardiner, Hu and Lakic \cite{GHL} have proved that for a given measured geodesic lamination $\lambda\in \mathcal{ML}_b(\mathbb{H}^2)$, the path $t\to\mathrm{E}_l^{t\lambda}\mid_{\mathbb{S}^1}$ is differentiable. Moreover its derivative at $t=0$ is a Zygmund vector field, called the \textit{bounded left infinitesimal earthquake}, and we denote it by
\begin{equation}\label{bounded_infinitesimal_earthquake}\Dot{\mathrm{E}_l^{\lambda}}(z)=\frac{d}{dt}\big\lvert_{t=0}\mathrm{E}_l^{t\lambda}(z),\end{equation} for any $z\in \mathbb{S}^1$. Similarly, one can define the bounded right infinitesimal earthquake $\Dot{\mathrm{E}_l^{\lambda}}$ by taking, in \eqref{bounded_infinitesimal_earthquake}, the derivative of the right infinitesimal earthquake $\mathrm{E}_r^{t\lambda}$. We have the following converse statement proved by Gardiner: 
\begin{theorem}\cite{Gardinerthurston}\label{GardThurs}
    Let $X$ be a Zygmund vector field of $\mathbb{S}^1$, then there exists a bounded measured geodesic
lamination $\lambda$ such that $X$ is the bounded left (resp. right) infinitesimal earthquake along $\lambda$, namely
$$X=\Dot{\mathrm{E}_l^{\lambda}} \ (\mathrm{resp}. \ X=\Dot{\mathrm{E}_r^{\lambda}} ).$$
\end{theorem}
Theorem \ref{GardThurs} can be seen as the analogue of Thurston's earthquake Theorem \ref{thurs}. However, unlike Thurston's theorem, which deals with circle homeomorphisms, Gardiner's theorem only deals with the case of Zygmund vector fields, which are the infinitesimal version of quasisymmetric maps. To overcome this problem, we will first define infinitesimal earthquakes in a more general and intrinsic way, generalizing bounded infinitesimal earthquakes.
\begin{defi}\label{infiearth}
A left (resp. right) infinitesimal earthquake $\mathcal{E}$ is a vector field on $\mathbb{H}^2$ such that there exists a  geodesic lamination $\lambda$ for which
the restriction $\mathcal{E}_{S}$ of $\mathcal{E}$ to any stratum $S$ of $\lambda$ is equal to the restriction of
a Killing field of $\mathbb{H}^2$, and 
for any two strata $S$ and $S'$ of $\lambda$, the \textit{comparison vector field}
$$\mathrm{Comp}(S,S'):=\mathcal{E}|_{S'}-\mathcal{E}|_{S}$$ is the restriction of a Killing field $\mathfrak{a}$ of $\mathbb{H}^2$, such that:
\begin{itemize}
\item $\mathfrak{a}$ is different from zero, unless possibly when one of the two strata $S$ and $S'$ is contained in the closure of the other;
\item 
when it is not the zero, 
 $\mathfrak{a}$ is a hyperbolic Killing field whose axis $\ell$ weakly separates $S$ and $S'$;
\item furthermore, $\mathfrak{a}$
translates to the left (resp. right), seen from $S$ to $S'$. 
\end{itemize}
\end{defi}
The last condition means that the one-parameter family of hyperbolic isometries given by $\{\exp(t\mathfrak{a})\}_{t\geq 0}$ consists of hyperbolic transformations that translate to the left (resp. right) seen from $S$ to $S^{'}$ in the sense of Definition \ref{earthquake}.
As an example of an infinitesimal earthquake, the bounded infinitesimal earthquake is viewed as a vector field on $\mathbb{H}^2$. Indeed, according to Epstein and Marden \cite{Epsmarden}, for each $z\in \mathbb{H}^2$, the path $t\to\mathrm{E}^{t\lambda}(z)$ is differentiable at $t=0$. Therefore, by the definition of an earthquake map, the vector field $z\to\frac{d}{dt}\big\lvert_{t=0}\mathrm{E}^{t\lambda}(z)$ satisfies Definition \ref{infiearth}. As a result, an infinitesimal earthquake on $\mathbb{H}^2$ may not necessarily be continuous.

\subsection{Simple infinitesimal earthquake}\label{simple}
In this section, we will examine the situation of an infinitesimal earthquake for which the associated geodesic lamination $\lambda$ is represented by a single geodesic $\ell$. We will refer to such a vector field as a \textit{simple infinitesimal earthquake} and denote it as $\mathcal{E}^{\ell}$. We will extract two fundamental properties of these vector fields, which will later allow us to prove certain properties required to establish that a vector field is a left (or right) infinitesimal earthquake.

First, remark that, up to the action of the isometry group of the hyperbolic plane, we can assume that $\ell=\mathbb{H}^2\cap e_3^{\bot}$, where $e_3 = (0,0,1)$. Let $S_1$ and $S_2$ be the two connected components of $\mathbb{H}^2\setminus\lambda$ such that:
$$S_1\subset\{(\eta_1,\eta_2)\in \mathbb{D}^2\mid \eta_2<0  \},\ S_2\subset\{(\eta_1,\eta_2)\in \mathbb{D}^2\mid \eta_2>0\}.$$
Since the restriction of $\mathcal{E}^{\ell}$ to the strata $S_1$ is a Killing vector field, say $\mathrm{K}$, then up to considering $\mathcal{E}^{\ell}-\mathrm{K}$ instead of $\mathcal{E}^{\ell}$, one may assume that the restriction of $\mathcal{E}^{\ell}$ to $S_1$ is identically zero. Therefore, based on the definition of an infinitesimal earthquake, the vector field $\mathcal{E}^{\ell}$ should have the following form:
$$
    \mathcal{E}^{\ell}(\eta)=\begin{cases}
			0, & \text{if $\eta\in S_1$}\\
            \mathrm{d}\Pi_{(1,\eta)}((1,\eta)\boxtimes ae_3)&\text{if $\eta\in\ell$}\\
             \mathrm{d}\Pi_{(1,\eta)}((1,\eta)\boxtimes be_3) &\text{if $\eta\in S_2$}
		 \end{cases}$$
To determine whether $\mathcal{E}^{\ell}$ is a left or a right infinitesimal earthquake, we should compute the comparison vector field. For instance let $Y=\mathrm{Comp}(S_1,S_2)$ be the comparison vector field between $S_1$ and $S_2$. Then
$$Y(\eta)=\mathrm{d}\Pi_{(1,\eta)}((1,\eta)\boxtimes be_3).$$
By Lemma \ref{hyperbolic_killing_field}, $Y$ is the hyperbolic killing field $\Lambda(be_3)$. To determine the sense of translation, remark that: $$\exp\left(t\Lambda(be_3)\right)=\begin{pmatrix}

 \cosh{(bt)} & -\sinh{(bt)} & 0\\
 -\sinh{(bt)} & \cosh{(bt)} &0 \\
 0           &     0&       1
\end{pmatrix}.$$ So, the one-parameter family of isometries $\exp\left(t\Lambda(be_3)\right)$ can be written in the Klein model $\mathbb{D}^2$ as follows:
$$\begin{array}{ccccc}
\exp\left(t\Lambda(be_3)\right) & : & \mathbb{D}^2 & \to &\mathbb{D}^2 \\
 & & (\eta_1,\eta_2) & \mapsto & \left( \frac{-\sinh(bt)+\cosh(bt)\eta_1}{\cosh(bt)-\sinh(bt)\eta_1},  \frac{\eta_2}{\cosh(bt)-\sinh(bt)\eta_1}    \right). \\
\end{array}$$
We can deduce from this that the comparison vector field $\mathrm{Comp}(S_1,S_2)$ translates to the left (resp. right), seen from $S_1$ to $S_2$ if and only if $b\geq 0$ (resp. $b\leq 0$) see Figure \ref{leftrightpicture}. Similarly, we can prove that the comparison vector fields 
\begin{itemize}
    \item $\mathrm{Comp}(S_1,\ell)$ translates to the left (resp. right), seen from $S_1$ to $\ell$ if and only if $a\geq0$ (resp. $a\leq0$).
    \item $\mathrm{Comp}(\ell,S_2)$ translates to the left (resp. right), seen from $\ell$ to $S_2$ if and only if $b\geq a$ (resp. $b\leq a)$.
\end{itemize}
In the end, we conclude that $\mathcal{E}^{\ell}$ is a left (resp. right) infinitesimal earthquake if and only if $0\leq a\leq b$ (resp. $b\leq a\leq 0).$
 \begin{figure}[h]
\centering
\includegraphics[width=.5\textwidth]{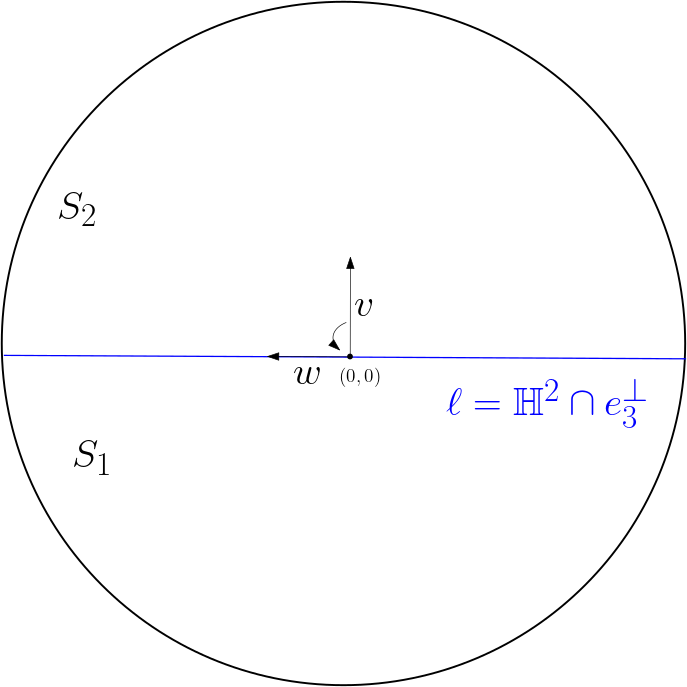}
\caption{The vector $v\in\mathrm{T}_{(0,0)}\mathbb{H}^2$ pointing from $S_1$ towards $S_2$, the vector $w$ tangent to the geodesic $\ell$ at $(0,0)$ and pointing towards $\exp(t\Lambda(be_3))(0,0)$ for $b$ and $t$ positive. Therefore, $(v, w)$ forms a positive basis with respect to the standard orientation of $\mathbb{H}^2.$}\label{leftrightpicture}
\end{figure}  

Now using Corollary \ref{support_function_of_killing}, it is not difficult to see that the vector field $\mathcal{E}^{\ell}$ extends to a vector field on $\mathbb{S}^1$, and its support function, denoted by $\phi_{\ell}:\mathbb{S}^1\to \mathbb{R}$, is defined as follows:
 $$\phi_{\ell}(\eta_1,\eta_2)=\begin{cases}
			0, & \text{if $\eta_2\leq0$}\\
          b\eta_2 &\text{if $\eta_2>0$}
		 \end{cases}.$$
The formula for $\phi_{\ell}$ also makes sense in the disk $\mathbb{D}^2$. Therefore, let us denote by $\overline{\phi_{\ell}}:\mathbb{D}^2\to\mathbb{R}$ the function defined by:
\begin{equation}\label{phil}
\overline{\phi_{\ell}}(\eta_1,\eta_2)=\begin{cases}
		    0 &\text{if $(\eta_1,\eta_2)\in\mathbb{D}^2$ $\mathrm{and}\ \eta_2\leq 0$}\\
           b\eta_2 &\text{if $(\eta_1,\eta_2)\in\mathbb{D}^2$ $\mathrm{and}\ \eta_2>0$}
		 \end{cases}\end{equation}
It worth noting that $\overline{\phi_{\ell}}$ is a convex (resp. concave) function if and only if $b\geq 0$ (or $b\leq 0$). The above discussion leads us to the following Lemma.
\begin{lemma}\label{left_convex}
 Let $\ell$ be a geodesic in hyperbolic plane and $\mathcal{E}^{\ell}$ be a simple infinitesimal earthquake along $\ell.$ Consider $\overline{\phi_{\ell}}:\mathbb{D}^2\to\mathbb{R}$ the extension to $\mathbb{D}^2$ of the support function of $\mathcal{E}^{\ell}$ as in \eqref{phil}. Then: 
 \begin{itemize}
     \item If $\mathcal{E}^{\ell}$ is a left infinitesimal earthquake then $\overline{\phi}_{\ell}$ is a convex function.
        \item If $\mathcal{E}^{\ell}$ is a right infinitesimal earthquake then $\overline{\phi}_{\ell}$ is a concave function.    
 \end{itemize}
\end{lemma}

Another property satisfied by simple infinitesimal earthquakes and used later in Section \ref{unique} is the following.
\begin{lemma}\label{2.27.}
Let $\mathcal{E}^{\ell}$ be a left (resp. right) simple infinitesimal earthquake along $\ell$. Let $S_1$ and $S_2$ be two strata of the geodesic lamination $\ell$. Consider $Y:=\mathrm{Comp}(S_1,S_2)$ the comparison vector field and $\phi_Y:\mathbb{S}^1\to\mathbb{R}$ its support function. Then, for all $x\in S_2$, we have:
    $$\phi_Y(x)\geq 0,\ (\mathrm{resp.}\ \phi_Y(x)\leq 0).$$ Similarly, for $x\in S_1$, we have: $\phi_Y(x)\leq 0,\ (\mathrm{resp.}\ \phi_Y(x)\geq 0).$  
\end{lemma}
\begin{proof}
The proof basically follows from the previous discussion. We will only treat the case of a left simple infinitesimal earthquake since the proof of the other situation is analogous. So, up to the action of $\mathrm{O}_0(1,2)\ltimes\minko$, one may assume that: $$
    \mathcal{E}^{\ell}(\eta)=\begin{cases}
			0, & \text{if $\eta\in S_1$}\\
            \mathrm{d}\Pi_{(1,\eta)}((1,\eta)\boxtimes ae_3)&\text{if $\eta\in\ell$}\\
             \mathrm{d}\Pi_{(1,\eta)}((1,\eta)\boxtimes be_3) &\text{if $\eta\in S_2$}
		 \end{cases}$$
   for some $0\leq a\leq b$. Here, the geodesic lamination $\ell$ has three strata: the two connected components of $\mathbb{H}^2\setminus \ell$ and the leaf $\ell$ itself. Let $S_1=\{(\eta_1,\eta_2)\mid \eta_2<0\}$ and $S_2=\{(\eta_1,\eta_2)\mid \eta_2>0\}$. Then the support function $\phi_Y$ of the comparison vector field $Y=\mathrm{Comp}(S_1,S_2)$ is given by: $$\phi_Y(\eta_1,\eta_2)=b\eta_2.$$ The proof is completed in this case since $b>0$. The other choices of $S_1$ and $S_2$ can be handled in a similar way. 
\end{proof}

\section{Proof of the infinitesimal earthquake Theorem}\label{Section3}
In this section, we will prove Theorems \ref{TH1} and \ref{THrad}. Recall that for a convex subset $K$ of an affine space, we say that an affine plane $\mathrm{P}$ is a support plane of $K$ (at $x\in\partial K$) if $\mathrm{P}$ contains $x$ and if all of $K$ is contained in one of the two closed half-spaces bounded by $\mathrm{P}$. We will adopt the same terminology for convex sets in $\HP$. Specifically, for a convex set $K$ in $\HP$, we say that a plane $\mathrm{P}\subset\HP$ is a \textit{support plane} of $K$ at $p\in \partial K$ if $p\in\overline{\mathrm{P}}\subset\overline{\HP}$ and $\mathrm{P}$ is disjoint from the interior of $K$.
Let $X$ be a vector field on $\mathbb{S}^1$, and let $\phi_X$ be its support function. Consider $\phi_X^{-}$ and $\phi_X^{+}$, defined as maps from $\overline{\mathbb{D}}^2$ to $\mathbb{R}$, as given in equations \eqref{phi-} and \eqref{phi+}. It is worth noting that if $\mathrm{P}$ is a support plane of the convex set $\mathrm{epi}^{\mp}(\phi_X^{\pm})$, then $\mathrm{P}$ is necessarily non-vertical, making it a spacelike plane in $\HP$ (see Figure \ref{supportplane}).
Therefore, if $\sigma\in\minko$, then the spacelike plane $\mathrm{P}_{\sigma}\subset\HP$ dual to $\sigma$ is a support plane at a point in $\mathrm{gr}(\phi_X^{-}|_{\mathbb{D}^2})$ (resp. $\mathrm{gr}(\phi_X^{+}|_{\mathbb{D}^2})$) if for all $\eta\in \mathbb{D}^2$,
\begin{equation}\label{Condition_de_plan_support}
\inner{(1,\eta),\sigma}_{1,2}\leq\phi_X^-(\eta)\ \mathrm{resp.}\ \inner{(1,\eta),\sigma}_{1,2}\geq\phi_X^+(\eta),\end{equation} with equality at some $\eta_0\in \mathbb{D}^2.$ \begin{figure}[htb]
\centering
\includegraphics[width=.6\textwidth]{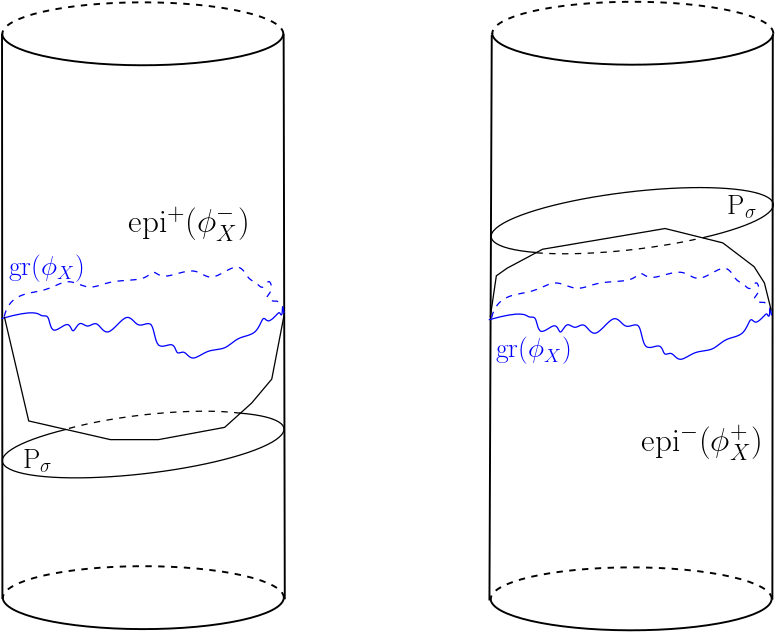}
\caption{On the left a support plane of $\mathrm{epi}^{+}(\phi_X^{-})$ at point of $\mathrm{gr}(\phi_X^{-}|_{\mathbb{D}^2}).$ On the right a support plane of $\mathrm{epi}^{-}(\phi_X^{+})$ at point of $\mathrm{gr}(\phi_X^{+}|_{\mathbb{D}^2})$.}\label{supportplane}
\end{figure}
The following Lemma shows that two support planes should necessarily intersect.
\begin{lemma}\label{planes_intersect}
 Let $X$ be a vector field of $\mathbb{S}^1$ which is not Killing. Let $\phi_X$ be its support function. Then any two spacelike support planes of $\mathrm{epi^{\mp}}(\phi^{\pm}_X)$ at points of $\mathrm{gr}(\phi_X^{\pm}\mid_{\mathbb{D}^2})$ intersect in a spacelike geodesic of $\HP$.
\end{lemma}

\begin{proof}
    Let us consider the case of $\mathrm{epi^+}(\phi_X^-)$, the other case being analogous. Let $\sigma_1,\sigma_2\in\minko$ such that $\mathrm{P}_{\sigma_1}$ and $\mathrm{P}_{\sigma_2}$ are support planes intersecting $\mathrm{gr}(\phi_X^-|_{\mathbb{D}^2})$. Suppose by contradiction that $\mathrm{P}_{\sigma_1}$ and $\mathrm{P}_{\sigma_2}$ are disjoint. Then for all $\eta\in\mathbb{D}^2$ either 
    $$\inner{(1,\eta),\sigma_1}_{1,2}< \inner{(1,\eta),\sigma_2}_{1,2} \ \mathrm{or} \ \inner{(1,\eta),\sigma_2}_{1,2}<\inner{(1,\eta),\sigma_1}_{1,2}.$$
    Without loss of generality, we assume that $\inner{(1,\eta),\sigma_1}_{1,2}<\inner{(1,\eta),\sigma_2}_{1,2}$ for all $\eta\in \mathbb{D}^2.$ Since $\mathrm{P}_{\sigma_2}$ is a support plane of $\mathrm{epi^{\mp}}(\phi^{\pm}_X)$, then: 
    $$\inner{(1,\eta),\sigma_2}_{1,2}\leq\phi_X^-(\eta),$$ hence $$\inner{(1,\eta),\sigma_1}_{1,2}<\phi_X^-(\eta),$$ for all $\eta\in \mathbb{D}^2.$ This is a contradiction since the equation $\inner{(1,\eta),\sigma_1}_{1,2}=\phi_X^-(\eta)$ has at least one solution in $\mathbb{D}^2$. Therefore the intersection is non empty and so it is a spacelike geodesic.
\end{proof}
We will also need this Lemma proved in  \cite{ConvexAnal} about the piecewise linear structure of the graph of $\phi_X^{\pm}$.

\begin{lemma}\cite[Lemma 4.9]{ConvexAnal}\label{structure_of_convex_core}
Let $X$ be a vector field of $\mathbb{S}^1$ and $\phi_X:\mathbb{S}^1\to\mathbb{R}$ be its support function. Let $\sigma\in\minko$ such that $\mathrm{P}_{\sigma}$ is a support plane of $\mathrm{epi}^{\mp}(\phi_X^{\pm})$. Then
\begin{enumerate}
\item \label{item1} If $\phi_X$ is lower semicontinuous, then $\{\eta\in\overline{\mathbb{D}^2}\mid \phi^-_X(\eta)=\langle(1,\eta),\sigma\rangle_{1,2} \}$ is the convex hull of the set $\{\eta\in\mathbb{S}^1\mid \phi^-_X(\eta)=\langle(1,\eta),\sigma\rangle_{1,2} \}.$ 
 \item \label{item2} If $\phi_X$ is upper semicontinuous, then $\{\eta\in\overline{\mathbb{D}^2}\mid \phi^+_X(\eta)=\langle(1,\eta),\sigma\rangle_{1,2} \}$ is the convex hull of the set $\{\eta\in\mathbb{S}^1\mid \phi^+_X(\eta)=\langle(1,\eta),\sigma\rangle_{1,2} \}.$
\end{enumerate}
\end{lemma}
Consider $\mathrm{P}_{\sigma}$ a support plane of $\mathrm{epi}^{\mp}(\phi_X^{\pm})$, then if $\mathrm{P}_{\sigma}\cap \mathrm{epi}^{\mp}(\phi^{\pm}_X)$ is line, it is called a \textit{bending line}, other wise $\mathrm{P}_{\sigma}\cap \mathrm{epi}^{\mp}(\phi^{\pm}_X)$ has non empty interior and it is called a \textit{flat piece}. We will also call bending line a boundary geodesic of a flat piece. Lemma \ref{structure_of_convex_core} implies that each point $(\eta,\phi_X^{\pm}(\eta))$ in the graph of $\phi_X^{\pm}$ (which is the boundary of the convex $\mathrm{epi}^{\mp}(\phi_X^{\pm})$ ) is either in a bending line or in the interior of a  flat piece $\mathrm{P}_{\sigma}\cap \mathrm{epi}^{\mp}(\phi_X^{\pm})$, in that case the plane $\mathrm{P}_{\sigma}$ is the unique support plane at $(\eta,\phi_X^{\pm}(\eta)).$ If $\mathcal{L}$ is a bending line, the support plane at points of $\mathcal{L}$ is in general not unique. Following \cite[Section II.1.6]{Epsmarden}, we have:
\begin{lemma}\label{extremelemma}
Let $X$ be a vector field on $\mathbb{S}^1$, and let $\phi_X : \mathbb{S}^1 \to \mathbb{R}$ be its support function. Let $\mathcal{L}$ be a bending line of $\mathrm{epi}^{\mp}(\phi_X^{\pm})$. Consider the set:
$$\mathcal{S}(\mathcal{L})=\{\sigma\in\minko\mid \mathrm{P}_{\sigma}\ \mathrm{support} \ \mathrm{planes}\ \mathrm{at} \ \mathcal{L} \}.$$
Then, either $\mathcal{S}(\mathcal{L})$ is a single point, or it constitutes a compact spacelike segment. We refer to the two planes corresponding to the endpoint of this segment as the \textit{extreme support planes}.
\end{lemma}

\begin{proof}
Let us first show that $\mathcal{S}(\mathcal{L})$ is a compact set. Let $\mathrm{P}_{\sigma_n}$ be a sequence of support planes at $\mathcal{L}$. By compactness of the space of planes in projective space, we can extract a subsequence such that the planes $\mathrm{P}_{\sigma_n}$ converge to some plane $\mathrm{P}$. Now, note that the condition to be a support plane is a closed condition, hence $\mathrm{P}$ is a support plane at points on $\mathcal{L}$, and thus it is necessarily a spacelike plane of the form $\mathrm{P}_{\sigma}$, so $\sigma_n\to\sigma$. 

Assume that $\mathcal{S}(\mathcal{L})$ is not a point, and we will show that $\mathcal{S}(\mathcal{L})$ is an arc. Without loss of generality, we may assume, through the action of the isometry group $\isom(\HP)$, that the geodesic $\mathcal{L}$ is given by:
$$\mathcal{L}=\{ (x,0,0)\mid x\in(-1,1)   \}\subset\mathbb{D}^2\times\mathbb{R}.$$
Let $\sigma=(\sigma_0,\sigma_1,\sigma_2)\in\minko$ such that $\mathcal{L}\subset \mathrm{P}_{\sigma}$, then $\inner{(1,x,0),\sigma}_{1,2}=0$. Hence we have $-\sigma_0+x\sigma_1=0$ for all $x\in(-1,1)$. This implies that $\sigma_0=\sigma_1=0$, and so
$$\mathcal{S}(\mathcal{L})\subset \{ (0,0,s)\mid s\in\mathbb{R} \}.$$ Since $\mathcal{S}(\mathcal{L})$ is a compact subset of a spacelike geodesic in $\minko$, in order to show that $\mathcal{S}(\mathcal{L})$ is a segment, it is enough to show that $\mathcal{S}(\mathcal{L})$ is convex, meaning that if $(0,0,s_1)$ and $(0,0,s_2)$ are in $\mathcal{S}(\mathcal{L})$, then for any $t\in [0,1]$, the point $(0,0,(1-t)s_1+ts_2)$ is contained in $\mathcal{S}(\mathcal{L})$. This follows easily from the support plane condition \eqref{Condition_de_plan_support}. In the end, we have shown that $\mathcal{S}(\mathcal{L})$ is a compact connected arc. 
\end{proof}
Recall that $\Pi$ is the radial projection defined in \eqref{radial}. We then define:
\begin{equation}\label{infearth_convexhull}\begin{array}{ccccc}
\mathcal{E}_{X}^- & : & \mathbb{D}^2 & \to & \mathbb{R}^2 \\
 & & \eta & \mapsto & \mathrm{d}_{(1,\eta)}\Pi\left((1,\eta)\boxtimes  \sigma\right), \\
\end{array}\end{equation}where $\sigma\in \mathbb{R}^{2,1}$ is a point for which the dual  spacelike plane $\mathrm{P}_{\sigma}\subset\HP$ is a support plane of $\mathrm{epi}^+(\phi_X^-)$ at $(\eta,\phi_X^-(\eta))$. In cases where $(\eta,\phi_X^{\pm}(\eta))$ lies on a bending line with multiple support planes, we make a canonical choice as follows: Let $\mathrm{P}_{\sigma_1}$ and $\mathrm{P}_{\sigma_2}$ be the two extremal support planes of $\mathrm{epi}^+(\phi_X^-)$ at $(\eta,\phi_X^-(\eta))$ as provided in Lemma \ref{extremelemma}. We then choose the medial plane $\mathrm{P}_{\frac{\sigma_1+\sigma_2}{2}}$ as the support plane at $(\eta,\phi_X^-(\eta))$.
Analogously, we define 
\begin{equation}\label{supearth_convexhull}\begin{array}{ccccc}
\mathcal{E}_{X}^+ & : & \mathbb{D}^2 & \to & \mathbb{R}^2 \\
 & & \eta & \mapsto & \mathrm{d}_{(1,\eta)}\Pi\left((1,\eta)\boxtimes  \sigma\right), \\
\end{array}\end{equation}
This time, $\mathrm{P}_{\sigma}$ is a support plane of $\mathrm{epi}^-(\phi_X^+)$ at $(\eta,\phi_X^+(\eta))$, and similarly, we take the medial support plane if there are several support planes.
Notice that the vector fields $\mathcal{E}_X^\pm$ are the same vector fields \eqref{vector_field_duality} given in the Introduction and are defined in the Klein model of the hyperbolic plane.
\begin{remark}
In the construction of $\mathcal{E}_X^\pm$, we \textit{choose} the medial support plane when multiple support planes are present. This choice is not very significant; any support plane can be chosen. In fact, the left and right infinitesimal earthquake that extend to $X$ is unique up to choice that has to be made at every bending line where it is discontinuous.
\end{remark}

\begin{example}\label{exx1}
  Let $v\in\minko$, and let $\phi_{\Lambda(v)}:z\mapsto\inner{(1,z),v}_{1,2}$ be the support function of the Killing vector field $\Lambda(v)$. We claim that the vector $\mathcal{E}_{\Lambda(v)}^\pm$ is exactly the Killing vector field $\Lambda(v)$ written in the Klein model. Indeed, since $\phi_{\Lambda(v)}$ is the restriction of an affine map to $\mathbb{S}^1$, then $\phi_{\Lambda(v)}^{\pm}(\eta)=\inner{(1,\eta),v}_{1,2}$ for all $\eta\in\overline{\mathbb{D}^2}$.
Hence, the graph of $\phi_{\Lambda(v)}^{\pm}|_{\mathbb{D}^2}$ is exactly the spacelike plane $\mathrm{P}_{v}$ in $\HP$ dual to $v$. Clearly, for any $\eta\in\mathbb{D}^2$, the plane $\mathrm{P}_v$ is the unique support plane at $(\eta,\phi_{\Lambda(v)}^{\pm}(\eta))$. Therefore, by definition, we have $\mathcal{E}_{\Lambda(v)}^{\pm}(\eta)=\mathrm{d}_{(1,\eta)}\Pi\left((1,\eta)\boxtimes v\right)$, which, according to Lemma \ref{kil_in_the_disc}, is the Killing vector field $\Lambda(v)$. By similar arguments, one may also check that simple infinitesimal earthquakes can be written as $\mathcal{E}_X^{\pm}$ for a vector field $X$ for which, its support function $\phi_X$ is the restriction to $\mathbb{S}^1$ of a piecewise affine map.
\end{example}
\subsection{Extension to the boundary}
We start investigating the extension property of the vector field $\mathcal{E}_X^{\pm}$.
\begin{prop}\label{extension}
Let $X$ be a vector field on $\mathbb{S}^1$. Then:
\begin{enumerate}
 \item \label{1} If $X$ is continuous, then $\mathcal{E}_X^+$ and $\mathcal{E}_X^-$ extend continuously to $X$.
    \item \label{2} If $X$ is lower semicontinuous, then $\mathcal{E}_X^-$ extends radially to $X$.
    \item \label{3} If $X$ is upper semicontinuous, then $\mathcal{E}_X^+$ extends radially to $X.$
\end{enumerate}
\end{prop}
To prove Proposition \ref{extension}, we need the following Lemma.
\begin{lemma}\label{4.7}
    Let $\mathrm{A}\in \mathrm{O}_0(1,2)$ and $v\in\minko$, then the following holds:
    \begin{equation}\label{38}
        \mathcal{E}^{\pm}_{\mathrm{A}_*X+\Lambda(v)}=\mathrm{A}_*\mathcal{E}_{X}^{\pm}+\mathcal{E}^{\pm}_{\Lambda(v)}.
    \end{equation}
\end{lemma}

\begin{proof}
To prove the lemma, let's begin by remarking that the relation \eqref{38} is trivially satisfied when $X$ is a Killing vector field. In this case, as shown in Example \ref{exx1}, the vector field $\mathcal{E}_X^{\pm}$ is simply the Killing vector field $X$.
Secondly, by Lemma \ref{equii}, we have:
$$\mathrm{gr}(\phi_{\mathrm{A}_*X+\Lambda(v)})=\mathrm{Is}(\mathrm{A},v)\mathrm{gr}(\phi_X).$$
Finally, to conclude the proof, we need to observe that the vector field $\mathcal{E}_X^{\pm}$ is constructed using Killing vectors coming from the support planes of $\mathrm{epi}^{\mp}(\phi_X^{\pm})$. We must prove that our choice of support plane is equivariant under the action of $\mathrm{O}_0(1,2)\ltimes\minko$. Observe that if a point $p\in\mathrm{gr}(\phi_X)$ has a unique support plane $\mathrm{P}$, then $\mathrm{Is}(\mathrm{A},v)\mathrm{P}$ is also the unique support plane at $\mathrm{Is}(\mathrm{A},v)p$.
Now, in cases where several support planes exist, we have chosen the medial support plane, and it is evident that this choice is equivariant with respect to the action of $\mathrm{O}_0(1,2)\ltimes\minko$. This concludes the proof.\end{proof}
We can now make this Remark.
\begin{remark}\label{EXequivariant}
  The property of continuous and radial extension of $\mathcal{E}_X^{\pm}$ is invariant under the action of $\mathrm{O}_0(1,2)\ltimes\minko$. Indeed if $\mathrm{A}\in \mathrm{O}_0(1,2)$ and $v\in\minko$, then by Lemma \ref{4.7}
    \begin{equation}
        \mathcal{E}^{\pm}_{\mathrm{A}_*X+\Lambda(v)}=\mathrm{A}_*\mathcal{E}_{X}^{\pm}+\mathcal{E}^{\pm}_{\Lambda(v)},
    \end{equation}
Now, from Example \ref{exx1}, $\mathcal{E}_{\Lambda(v)}^{\pm}$ is a Killing vector field of $\mathbb{H}^2$, and thus, it extends continuously to $\mathbb{S}^1$. Therefore, $\mathcal{E}_X^{\pm}$ extends continuously (resp. radially) to vector fields on $\mathbb{S}^1$ if and only if $\mathcal{E}_{\mathrm{A}_*X+\Lambda(v)}^{\pm}$ extends continuously (resp. radially) to vector field of $\mathbb{S}^1$.
\end{remark}
We start first by providing the proof of statement \eqref{1} of Proposition \ref{extension}.

\begin{proof}[Proof of Proposition \ref{extension}-\eqref{1}]
Before presenting the proof, it is essential to note that in this case, since $\phi_X$ is continuous, then $\phi_X^{\pm}:\overline{\mathbb{D}^2}\to\mathbb{R}$ is continuous until the boundary, see Property $(\mathrm{P}3)$ in Section \ref{P3}, page 12.\\ Let $\eta_n=(x_n,y_n)\in \mathbb{D}^2$ be a sequence converging  to $\eta_{\infty}=(x,y)\in \mathbb{S}^1$. The goal is to show that: 
    $$\mathcal{E}_X^{\pm}(\eta_n)\to X(\eta_{\infty}).$$
    We will only give the proof for the vector field $\mathcal{E}_X^-$ as the proof for
$\mathcal{E}_X^+$ can be obtained in the same way.
Let $\mathrm{P}_{\sigma_n}$ be a sequence of spacelike support planes of $\mathrm{epi}^+(\phi_X^-)$ at $(\eta_n,\phi_X^-(\eta_n))$. Up to extracting a subsequence, there are two possibilities: either $\sigma_n\to\sigma_{\infty}$
and $\mathrm{P}_{\sigma_n}$ converges to the spacelike support plane $\mathrm{P}_{\sigma_{\infty}}$ , or $\sigma_n$ diverges in $\minko$ and $\mathrm{P}_{\sigma_n}$
converges to a vertical plane. We will treat these two situations separately.\\ Consider the case where $\sigma_n \to \sigma_{\infty}$. Then, by the continuity of the Minkowski cross product and the differential of the radial projection $\Pi$, we obtain:
$$\mathcal{E}_X^-(\eta_n)=\mathrm{d}_{(1,\eta_n)}\Pi\left((1,\eta_n)\boxtimes \sigma_n\right)\to\mathrm{d}_{(1,\eta_{\infty})}\Pi\left((1,\eta_{\infty})\boxtimes \sigma_{\infty}\right).$$ Now, from the equation \eqref{pi(z,1)} of Corollary \ref{support_function_of_killing}, we get: $$\mathrm{d}_{(1,\eta_{\infty})}\Pi\left((1,\eta_{\infty})\boxtimes \sigma_{\infty}\right)=\inner{(1,\eta_{\infty}),\sigma_{\infty}}_{1,2}(-y,x).$$ However, by continuity of $\phi_X^-$ on $\overline{\mathbb{D}}^2$, we have $\phi_X^-(\eta_n)\to\phi_X^-(\eta_{\infty})=\phi_X(\eta_{\infty})$. Thus $(\eta_{\infty}, \phi_X(\eta_{\infty})) \in \mathrm{P}_{\sigma_{\infty}}$, then $\langle (1,\eta_{\infty}), \sigma_{\infty}\rangle_{1,2} = \phi_X(\eta_{\infty})$, and so: $$\mathrm{d}_{(\eta_{\infty},1)}\Pi\left((1,\eta_{\infty})\boxtimes  \sigma_{\infty}\right)=\phi_X(\eta_{\infty})(-y,x)=i\cdot\eta_{\infty}\phi_X(\eta_{\infty})=X(\eta_{\infty}).$$
and the proof is complete in that case. \\
Now, let's move on to the other case, where $\sigma_n$ diverges. It follows from Remark \ref{EXequivariant} that up to the action of $\mathrm{O}_0(1,2)\ltimes\minko$ on $X$, one may assume that:
$$\eta_{\infty}=(1,0),\ \phi_X(1, 0) = \phi_X(0, 1) = \phi_X(0, -1) = 0,$$ see Remark \ref{remarqueabc}. The goal now is to show that: \begin{equation}\label{convergence_to_zero}
   \mathcal{E}_X^-(\eta_n)\to 0.\end{equation}\color{black} 
\color{black}Let $p = (1, 0, 0)$, $r_n = \sqrt{x_n^2 + y_n^2}$, and $v_n = \frac{1}{r_n}(-y_n, x_n)$. Consider $w_n = \frac{1}{r_n}\eta_n$ so that $(p, (0,w_n), (0,v_n))$ is an oriented orthonormal basis, see Figure \ref{Bonpicture}.
\begin{figure}[htb]
\centering
\includegraphics[width=.5\textwidth]{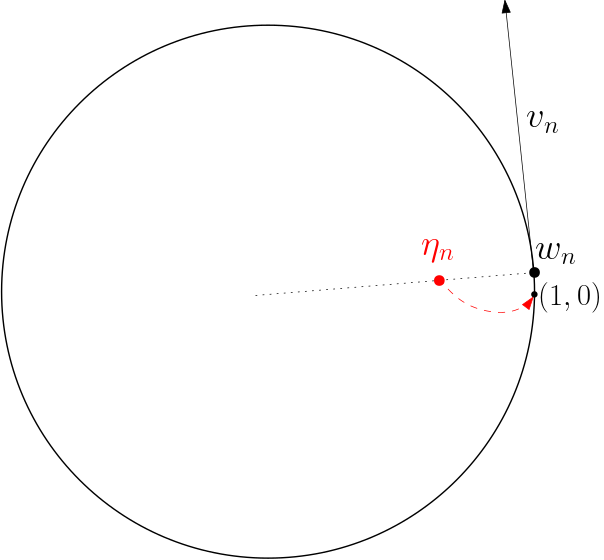}
\caption{The sequences $w_n$ and $v_n$ are such that $(p, (0, w_n), (0, v_n))$ forms an oriented orthonormal basis.}\label{Bonpicture}
\end{figure}

\noindent We can write $\sigma_n$ in this basis as: $$\sigma_n=a_np+b_n(0,w_n)+c_n(0,v_n).$$ Since $(1,\eta_n)=p+(0,\eta_n)$, then by a computation similar to Lemma \ref{crosspructlemma} (see formulas \eqref{cross_product_basis}), we get:
$$(1,\eta_n)\boxtimes\sigma_n= -r_nc_np-c_n(0,w_n)-(r_na_n-b_n)(0,v_n).$$
Now, by an elementary computation, one may check that the differential of the radial projection $\Pi$ satisfies the following:
$$\mathrm{d}_{(1,x,y)}\Pi(v_0,v_1,v_2)=\left( -x v_0+v_1, -y v_0+v_2       \right),$$ for all $(x,y)\in\mathbb{D}^2$ and $(v_0,v_1,v_2)\in\mathbb{R}^3$. This implies that $$\mathrm{d}_{(1,\eta_n)}\Pi(p) = -\eta_n,\  \mathrm{d}_{(1,\eta_n)}\Pi(0,w_n) = w_n, \ \mathrm{d}_{(1,\eta_n )}\Pi(0,v_n) = v_n,$$ 
therefore:
$$
    \mathrm{d}_{(1,\eta_n)}\Pi( (1,\eta_n)\boxtimes\sigma_n )=c_n(r_n^2-1)w_n+(b_n-r_na_n)v_n,
$$
and since $(\eta_n, \phi_X^-(\eta_n)) \in \mathrm{P}_{\sigma_n}$, then $\langle (1,\eta_n), \sigma_n \rangle_{1,2} = \phi_X^-(\eta_n)$, so: $$\phi_X^-(\eta_n)=r_nb_n-a_n.$$ As consequence, we have
\begin{equation}
\mathcal{E}_X^-(\eta_n)=\mathrm{d}_{(1,\eta_n)}\Pi((1,\eta_n)\boxtimes\sigma_n)=r_n\phi^-(\eta_n)v_n+(1-r_n^2)b_nv_n-(1-r_n^2)c_nw_n.
\end{equation}
Since $\phi_X^-(\eta_n) \to \phi_X(1, 0) = 0$, in order to prove \eqref{convergence_to_zero}, it is enough to show that:
\begin{equation}\label{necessary_condition}
    (1-r_n^2)b_n\to0 \ \mathrm{and}\ (1-r_n^2)c_n\to 0.
\end{equation}
To prove this, consider the following estimate obtained from the support plane property:
$$\label{support_plane_estimate}
\mathrm{For}\ \mathrm{all} \ \eta\in \mathbb{D}^2,\ \inner{(1,\eta),\sigma_n}_{1,2}\leq \phi_X^-(\eta),$$
which can be written as:
\begin{equation}\label{support_plane_estimate}
    -a_n+\frac{b_nx_n-c_ny_n}{r_n}x+\frac{b_ny_n+c_nx_n}{r_n}y\leq\phi_X^-(x,y)
\end{equation}
for all $(x,y)\in \overline{\mathbb{D}^2}$. Applying the inequality \eqref{support_plane_estimate} to $(x,y)\in\{   (1,0),(0,1),(0,-1)\}$, we get
\begin{equation}\label{ineq}
    0\leq\displaystyle\left\lvert \frac{b_ny_n+c_nx_n}{r_n}\right\rvert\leq a_n, \ \ \frac{b_nx_n-c_ny_n}{r_n}\leq a_n.\end{equation} Now, denote by $\overline{\mathrm{P}}_{\sigma_n}$ the closure of the spacelike plane $\mathrm{P}_{\sigma_n}$ in $\overline{\mathbb{D}^2}\times\mathbb{R}$, namely:
    $$\overline{\mathrm{P}}_{\sigma_n}=\left\{ (x,y,t)\in\overline{\mathbb{D}^2}\times\mathbb{R}\mid \inner{(1,x,y),\sigma_n}_{1,2}=t  \right\},$$ then, we have:
\begin{equation}
\overline{\mathrm{P}}_{\sigma_n}=\left\{ (x,y,t)\in\overline{\mathbb{D}^2}\times\mathbb{R}\mid -\frac{t}{a_n}+\frac{b_nx_n-c_ny_n}{r_na_n}x + \frac{b_ny_n+c_nx_n}{r_na_n}y=1  \right\}.\end{equation}

Now, remark that both conditions that $(\eta_n,\phi_X^-(\eta_n)) \in \mathrm{epi}^+(\phi_X^-)$ and that $\overline{\mathrm{P}}_{\sigma_n}$ is a support plane of $\mathrm{epi}^+(\phi_X^-)$ are closed conditions, hence up to subsequence the plane $\overline{\mathrm{P}}_{\sigma_n}$ will converge to a support plane $\mathrm{P}_{\infty}\subset \overline{\mathbb{D}^2}\times\mathbb{R}$ of $\mathrm{epi}^+(\phi_X^-)$ at $\eta_{\infty}=(1,0)$, by assumption, such support plane is vertical since it is not spacelike. This implies that $\mathrm{P}_{\infty}=\{(x, y, t)\in\overline{\mathbb{D}^2}\times\mathbb{R} \mid x = 1\}$, see Figure \ref{extensionpicture}. 
\begin{figure}[t]
\centering
\includegraphics[width=.4\textwidth]{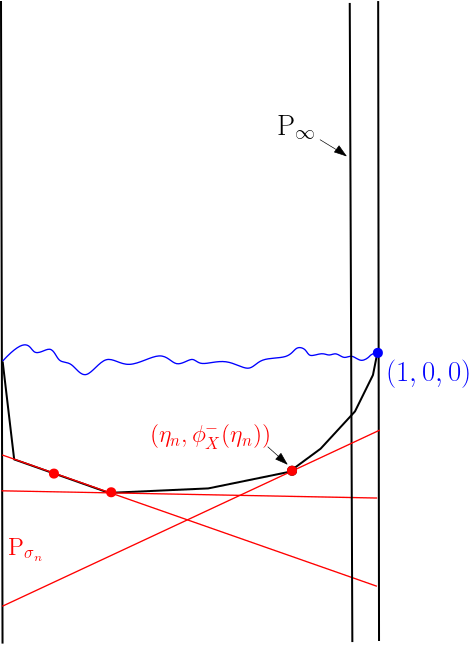}
\caption{A two dimensional picture illustrating the proof, when the sequence of support planes $\mathrm{P}_{\sigma_n}$ (in red) does not converge to a spacelike plane in $\HP$, it will converges to a vertical plane, which is a support for $\mathrm{epi}^+(\phi_X^-)$. From the image, the euclidean normal of the limiting plane $\mathrm{P}_{\infty}$ must be $(1,0,0)$. otherwise, it would intersect transversely the cylinder and would fail to be a support plane of $\mathrm{epi^+}(\phi_X^-)$. Now, since $\mathrm{P}_{\infty}$ contains the point $(1, 0, 0)$, it follows that $\mathrm{P}_{\infty}$ is equal to the plane $x=1$.}\label{extensionpicture}
\end{figure}

As a consequence, given that $a_n\geq 0$ and satisfies the estimates \eqref{ineq}, we can deduce the following limits:
\begin{equation}\label{44}
a_n\to +\infty,\ \frac{b_nx_n-c_ny_n}{a_n}\to1,\ \frac{b_ny_n+c_nx_n}{a_n}\to 0.
\end{equation} But \begin{equation}\label{45}
    r_nb_n-a_n=\phi_X^-(\eta_n),
\end{equation} and so since $\phi_X^-(\eta_n)\to 0$ and $r_n\to 1$, then equation \eqref{44} implies that: 

\begin{equation}\label{46}
    \frac{b_n}{a_n}\to 1.
\end{equation}
Now, from \eqref{44}, we have $\frac{b_ny_n+c_nx_n}{a_n}\to 0$, thus we get by \eqref{46}:

\begin{equation}\label{47}
\frac{c_n}{a_n}\to 0.\end{equation}
Next, using \eqref{45} we can rewrite the condition \eqref{support_plane_estimate} as follows:
\begin{equation}\label{48}
    \underbrace{\left(\frac{x_n}{r_n}x+\frac{y_n}{r_n}y-1\right)}_{\mathcal{Q}_n(x,y)}b_n+\underbrace{\left( \frac{x_n}{r_n}y-\frac{y_n}{r_n}x   \right)}_{\mathcal{R}_n(x,y)}c_n+(1-r_n)b_n\leq\phi_X^-(x,y)-\phi_X^-(\eta_n),\end{equation}
for all $(x, y) \in \overline{\mathbb{D}^2}$ where 
$$\begin{cases}
\mathcal{Q}_n(x,y)=\frac{x_n}{r_n}x+\frac{y_n}{r_n}y-1,\\
\mathcal{R}_n(x,y)=\frac{x_n}{r_n}y-\frac{y_n}{r_n}x.
\end{cases}$$
The goal now is to bound $(1-r_n)b_n$ from above and below by an expression that converges to zero. To achieve this, observe that the point $(\frac{x_n}{r_n}, \frac{y_n}{r_n}) \in \overline{\mathbb{D}^2}$ is the solution to the system of equations given by:
$$\mathcal{Q}_n(x,y)=\mathcal{R}_n(x,y)=0.$$
Substituting this solution into \eqref{48}, we obtain:
$$(1-r_n)b_n\leq \phi_X^-(\frac{x_n}{r_n},\frac{y_n}{r_n})-\phi_X^-(\eta_n).$$ Next observe that the point $\left(\frac{x_n}{r_n}(2r_n-1),\frac{y_n}{r_n}(2r_n-1)\right) \in \overline{\mathbb{D}^2}$ is the solution to the system of equations given by:
$$\mathcal{Q}_n(x,y)=-2(1-r_n),\ \mathcal{R}_n(x,y)=0.$$
Substituting $(\frac{x_n}{r_n}(2r_n-1),\frac{y_n}{r_n}(2r_n-1))\in \overline{\mathbb{D}^2}$ into \eqref{48}, we obtain:
$$\phi_X^-(\eta_n)-\phi_X^-\left(\frac{x_n}{r_n}(2r_n-1),\frac{y_n}{r_n}(2r_n-1)\right)\leq (1-r_n)b_n.$$ In the end we have the following 
\begin{equation}\label{equationofextension}
    \phi_X^-(\eta_n)-\phi_X^-\left(\frac{x_n}{r_n}(2r_n-1),\frac{y_n}{r_n}(2r_n-1)\right)\leq (1-r_n)b_n\leq\phi_X^-\left(\frac{x_n}{r_n},\frac{y_n}{r_n}\right)-\phi_X^-(\eta_n).\end{equation} 
Recall that $\phi_X^{-}:\overline{\mathbb{D}^2}\to\mathbb{R}$ is continuous until the boundary, so passing to the limit in \eqref{equationofextension}, we can conclude:
$$(1-r_n)b_n\to 0.$$ Now, since $\frac{c_n}{b_n}\to 0$ by \eqref{47}, we also have:
$$(1-r_n)c_n\to 0.$$ This completes the proof of \eqref{necessary_condition} and, consequently, the proof of the first statement of the Proposition.\end{proof}

\begin{proof}[Proof of Proposition \ref{extension} (\ref{2}) and (\ref{3})]
Again, one only needs to prove the second statement of Proposition \ref{extension}, as the third statement can be obtained in an analogous way. Let $x\in\mathbb{D}^2$ and $\eta_{\infty}\in\mathbb{S}^1$. We need to show that if $\eta_n$ is a sequence in the segment $[x,\eta_{\infty})$ converging to $\eta_{\infty}$, then $\mathcal{E}_X^-(\eta_n)$ converges to $X(\eta_{\infty})$. By the equivariance property of $\mathcal{E}_X^-$ (see Remark \ref{EXequivariant}), we may assume that $x=(0,0)$ and $\eta_{\infty}=(1,0)$. Moreover,
\begin{align*}
\phi_X(1,0) &= \phi_X(0,1) = \phi_X(0,-1) = 0.
\end{align*}
Take $\eta_n=(x_n,0)$ a sequence in the segment $[(0,0),(1,0))$ converging to $(1,0)$ and keep the same notation as in the proof of the first statement of Proposition \ref{extension}, one has
\begin{equation}
\mathcal{E}_X^-(\eta_n) = r_n\phi^-(\eta_n)v_n + (1-r_n^2)b_nv_n - (1-r_n^2)c_nw_n.
\end{equation}
Since $\phi_X$ is lower semicontinuous, then the boundary value of $\phi_X^-$ is equal to $\phi_X$ (see property $(\mathrm{P}2)$ in Section \ref{P2}, Page $12$) hence $\phi_X^-(\eta_n)\to\phi_X(\eta_{\infty})=0$. A computation similar to that in the proof of the first statement of Proposition \ref{extension} (see limits \eqref{46} and \eqref{47}) shows that: \begin{equation}\label{51}
    \frac{c_n}{b_n}\to 0.
\end{equation} Moreover, we have the inequalities:
\begin{equation}\label{equationofextension2}
\phi_X^-(\eta_n) - \phi_X^-\left(\frac{x_n}{r_n}(2r_n-1),0\right) \leq (1-r_n)b_n \leq \phi_X^-\left(\frac{x_n}{r_n},0\right) - \phi_X^-(\eta_n),
\end{equation}
as in \eqref{equationofextension}. Now, remark that the points $\left(\frac{x_n}{r_n},0\right)$ and $\left(\frac{x_n}{r_n}(2r_n-1),0\right)$ are in the segment $[(0,0),(1,0))$ and hence, by the lower semicontinuity of $\phi_X$, we get 
$$\phi_X^-(\frac{x_n}{r_n},0)\to\phi_X(1,0)=0, \ \mathrm{and}\ \phi_X^-(\frac{x_n}{r_n}(2r_n-1),0)\to\phi_X(1,0)=0.$$
Therefore, passing to the limit in \eqref{equationofextension2}, we obtain that $(1-r_n)b_n\to 0$, and so by \eqref{51} we get $(1-r_n)c_n\to0$. In conclusion we have $\mathcal{E}_X^-(\eta_n)\to 0$ and this completes the proof.

\end{proof}

\subsection{Infinitesimal earthquake properties}
In this section, we will prove that the vector fields $\mathcal{E}_X^{\pm}$, satisfies the properties defining infinitesimal earthquakes.

\begin{prop}\label{earthquake_properties}
Let $X$ be a vector field of the circle. Then
\begin{itemize}
    \item If $X$ is lower semicontinuous, then $\mathcal{E}_X^-$ is a left infinitesimal earthquake.
    \item If $X$ is upper semicontinuous, then $\mathcal{E}_X^+$ is a right infinitesimal earthquake.
\end{itemize}  
\end{prop}

\begin{proof}   
First, we need to define the geodesic lamination $\lambda^{\pm}$ associated to $\mathcal{E}_X^{\pm}$. For each $v \in \minko$ such that the spacelike plane $\mathrm{P}_v$ is a support plane of $\mathrm{epi}^{\mp}(\phi_X^\pm)$ at a point of $\mathrm{gr}(\phi_X^\pm|_{\mathbb{D}^2})$, we denote by $S_X^{\pm}(v)$ the set:\begin{equation}\label{strata}
S_X^{\pm}(v)=\{\eta\in\mathbb{D}^2\mid \phi^\pm_X(\eta)=\langle(1,\eta),v\rangle_{1,2} \}.\end{equation} The image of this set under $\phi_X^\pm$ corresponds to $\mathrm{P}_v \cap \mathrm{gr}(\phi_X^\pm|_{\mathbb{D}^2})$. 
Therefore if $X$ is lower semicontinuous, then we define the geodesic lamination $\lambda^{-}$ as the collection of all the connected components of $S_X^{-}(v) \setminus \mathrm{Int}(S_X^{-}(v))$, as $\mathrm{P}_v$ varies over all support planes of $\mathrm{gr}(\phi_X^{\pm}|_{\mathbb{D}^2})$. These collection are in fact geodesic because of item \eqref{item1} of Lemma \ref{structure_of_convex_core}. Similarly if $X$ is upper semicontinuous, we define $\lambda^{+}$ as the collection of all the connected components of $S_X^{+}(v) \setminus \mathrm{Int}(S_X^{+}(v))$, as $\mathrm{P}_v$ varies over all support planes of $\mathrm{gr}(\phi_X^{+}|_{\mathbb{D}^2})$, and it is a collection of geodesic because of item \eqref{item2} of Lemma \ref{structure_of_convex_core}. The closeness of $\lambda^{\pm}$ as a subset of $\mathbb{D}^2$ follows from the continuity of $\phi^{\pm}_X$ in the disk, and the fact that being a support plane is a closed condition.

Having defining the geodesic lamination, we will now verify the earthquake properties. We will provide the proof for $\mathcal{E}_X^-$ as the proof for $\mathcal{E}_X^+$ is analogous.  
By definition of $\mathcal{E}_X^-$ and by Lemma \ref{kil_in_the_disc}, the restriction of $\mathcal{E}_X^-$ to each stratum of $\lambda^-$ is a Killing vector field. So, we only need to verify that the comparison vector field satisfies the necessary properties. Let $S_X^-(v)$ and $S_X^-(w)$ be two strata of $\lambda^-$. Since $\mathrm{P}_v$ and $\mathrm{P}_w$ are support planes of $\mathrm{epi^+}(\phi_X^-)$, then for $\eta \in \mathbb{D}^2$, we have:
\begin{equation}\label{41}
\inner{(1,\eta),v}_{1,2}\leq\phi^-_X(\eta), \ \mathrm{and}\ \inner{(1,\eta),w}_{1,2}\leq\phi^-_X(\eta).
\end{equation}
Let us start with the trivial case where $v=w$. In this case, we clearly have: $$\mathrm{Comp}(S_X^-(v),S_X^-(w))=0.$$ In fact, it is allowed in Definition \ref{infiearth} for the comparison vector field to be zero in this case.

Now we treat the remaining case where $v \neq w$. According to Lemma \ref{planes_intersect}, the planes $\mathrm{P}_v$ and $\mathrm{P}_w$ intersect along a spacelike geodesic in $\HP$. Furthermore, by Proposition \ref{hpangle}, $v-w$ is a spacelike vector in $\minko$. Thus, the orthogonal of $v-w$ defines a geodesic in $\mathbb{H}^2$, which we will denote by $\ell$. 

Consider $S_1$ and $S_2$, the two connected components of $\mathbb{H}^2 \setminus \ell$, such that $S_X^-(v) \subset S_1$ and $S_X^-(w) \subset S_2$. Now, consider the vector field $\mathcal{E}^{\ell}$ defined by: \begin{equation}\label{simplemedium}
    \mathcal{E}^{\ell}(\eta)=\begin{cases}
			\mathrm{d}_{(1,\eta)}\Pi\left((1,\eta)\boxtimes  v\right), & \text{if $\eta\in S_1$}\\
            \mathrm{d}_{(1,\eta)}\Pi\left((1,\eta)\boxtimes  \frac{v+w}{2}\right)&\text{if $\eta\in \ell$}\\
           \mathrm{d}_{(1,\eta)}\Pi\left((1,\eta)\boxtimes  w\right) &\text{if $\eta\in S_2$}
		 \end{cases}\end{equation}
Denote by $\mathrm{Comp}_{\ell}$ the comparison vector field of $\mathcal{E}^{\ell}$, then we have
\begin{equation}\label{42}
    \mathrm{Comp}_{\ell}(S_1,\ell)(\eta)=\mathrm{Comp}_{\ell}(\ell,S_2)(\eta)=\mathrm{d}_{(1,\eta)}\Pi\left((1,\eta)\boxtimes  \frac{w-v}{2}\right).
\end{equation}

\begin{equation}\label{43}
    \mathrm{Comp}_{\ell}(S_1,S_2)(\eta)=\mathrm{d}_{(1,\eta)}\Pi\left((1,\eta)\boxtimes  (w-v)\right).
\end{equation}

Remark that both vector fields in $\eqref{42}$ and $\eqref{43}$ are hyperbolic Killing vector field with axis $\ell$ (see Lemma \ref{hyperbolic_killing_field}). Moreover $\mathrm{Comp}_{\ell}(S_1,\ell)$ translates to the left (resp. right), seen from $S_1$ to $\ell$ if and only if $\mathrm{Comp}_{\ell}(\ell,S_2)$ translates to the left (resp. right), seen from $\ell$ to $S_2$ if and only if $\mathrm{Comp}_{\ell}(S_1,S_2)$ seen from $S_1$ to $S_2$. Therefore we deduce that $\mathcal{E}^{\ell}$ is either a left or a right simple infinitesimal earthquake.
Since the comparison vector $\mathrm{Comp}(S_X^-(v),S_X^-(v))$ coincides with the comparison vector field $\mathrm{Comp}_{\ell}(S_1,S_2)$, then it is enough to prove that $\mathcal{E}^{\ell}$ is a left infinitesimal earthquake to deduce that $\mathcal{E}_X^-$ is a left infinitesimal earthquake.

Consider $\phi_l:\mathbb{S}^1\to\mathbb{R}$ the support function of $\mathcal{E}^{\ell}$ and assume by contradiction that $\mathcal{E}^{\ell}$ is a right infinitesimal earthquake, then by Lemma \ref{left_convex} the function $\overline{\phi_l} : \mathbb{D}^2 \to \mathbb{R}$ defined by:

$$\overline{\phi_l}(\eta)=\begin{cases}
			\inner{(1,\eta),v}_{1,2}, & \text{if $\eta\in S_1\cup \ell$}\\
           \inner{(1,\eta),w}_{1,2} &\text{if $\eta\in S_2$}
		 \end{cases}$$
is a concave function. However, from equation \eqref{41}, $\overline{\phi_{\ell}}\leq \phi_X^-$, thus: 
 $$\mathrm{epi^+}(\phi_X^-)\subset\mathrm{epi^+}(\overline{\phi_l})$$
As $\phi_X^-$ is convex, the epigraph of $\phi_X^-$ is also convex. Consequently, the epigraph of $\overline{\phi_l}$ must also be convex. In the end $\overline{\phi}_{\ell}$ is convex and concave, and thus from the formula of $\overline{\phi_{\ell}}$, the vector $v$ should be equal to $w$, and this is a contradiction. Thus $\mathcal{E}^{\ell}$ is a left infinitesimal earthquake and hence $\mathrm{Comp}(S_X^-(v),S_X^-(v))$ translates to the left seen from $S_X^-(v)$ to $S_X^-(w)$, this concludes the proof.
\end{proof}
The proof of Theorem \ref{THrad} is then completed since it  follows from Propositions \ref{extension} and \ref{earthquake_properties}.

\subsection{Uniqueness of infinitesimal earthquake extension}\label{unique}
In this section, we will show that the vector fields $\mathcal{E}_X^{+}$ and $\mathcal{E}_X^-$ constructed in \eqref{infearth_convexhull} are is some sense the unique way to get left and right infinitesimal earthquakes on $\mathbb{D}^2$ that extend to $X$. More precisely, we have: 

\begin{prop}\label{uniqness_of_earthquake}
Let $X$ be a continuous vector field of $\mathbb{S}^1$. Let $\mathcal{E}_X^{\pm}$ be the infinitesimal earthquake along the geodesic lamination $\lambda^{\pm}$ extending to $X.$
Assume that there exists another left (resp. right) infinitesimal earthquake $\mathcal{E}_1^-$ (resp. $\mathcal{E}_1^+$) which extends continuously to $X$. Denote by $\lambda_1^{\pm}$ the geodesic lamination associated to $\mathcal{E}_1^{\pm}$. Then 
\begin{itemize}
    \item The geodesic lamination $\lambda_1^{\pm}$ coincides with $\lambda^{\pm}.$
    \item The vector field $\mathcal{E}_1^{\pm}$ coincides with $\mathcal{E}_X^{\pm}$ in all gaps of $\lambda^{\pm}$ and in all leaves of $\lambda^{\pm}$ where $\mathcal{E}_X^{\pm}$ is continuous.
\end{itemize}

\end{prop}
The next lemma characterizes the continuity of $\mathcal{E}_X^{\pm}$ in terms of the uniqueness of support planes.
\begin{lemma}\label{carachterisation_conti_bending}
    Let $X$ be a continuous vector field and $\phi_X$ be its support function. We take $\eta\in \mathbb{D}^2$, then $\mathcal{E}^{\pm}_X$ is continuous on $\eta$ if and only if there is a unique support plane of $\mathrm{epi}^{\mp}(\phi_X^{\pm})$ at $(\eta,\phi_X^{\pm}(\eta)).$ 
\end{lemma}
\begin{proof}
We will focus on proving the Lemma for the left infinitesimal earthquake $\mathcal{E}_X^-$ since the proof for the right infinitesimal earthquake $\mathcal{E}_X^+$ is analogous. Let $\eta\in \mathbb{D}^2$ and assume that there is a unique support plane $\mathrm{P}$ of $\mathrm{epi}^{\mp}(\phi_X^{\pm})$ at $(\eta,\phi_X(\eta))$, we will prove that the left infinitesimal earthquake $\mathcal{E}_X^{-}$ is continuous on $\eta.$ Let $\eta_n$ be a sequence of points converging to $\eta.$ Let $\sigma_n\in \minko$ be such that $\mathrm{P}_{\sigma_n}$ is a support plane at $(\eta_n,\phi_X(\eta_n))$ and 
$$\mathcal{E}_{X}^+(\eta_n)= \mathrm{d}_{(1,\eta)}\Pi\left((1,\eta)\boxtimes  \sigma_n\right).$$ Up to subsequence, the plane $\mathrm{P}_{\sigma_n}$ converges to a plane $\mathrm{Q}$ which is a support plane at $(\eta,\phi_X(\eta))$, and so by the uniqueness of support plane $\mathrm{Q}=\mathrm{P}$ and hence $\mathcal{E}_X^-(\eta_n)\to\mathcal{E}_X^-(\eta).$

Now, let's prove the converse statement. By contradiction, assume that $\mathcal{E}_X^-$ is continuous at $\eta$, and there exist several support planes at $(\eta,\phi_X^-(\eta))$. Therefore, the point $(\eta,\phi_X^-(\eta))$ must be on a bending line $\mathcal{L}$. We take $\mathrm{P}_v$ and $\mathrm{P}_w$ the two extreme support planes at $(\eta,\phi_X^-(\eta))$, where $v$ and $w$ are different vectors in $\minko$, and $\mathrm{P}_v\cap\mathrm{P}_w=\mathcal{L}$. Denote by $\ell$ the projection of $\mathcal{L}$ onto $\mathbb{D}^2$, namely:
\begin{equation}\label{elll}
    \ell=\{ \eta\in\mathbb{D}^2\mid \inner{(1,\eta),v}_{1,2}=\inner{(1,\eta),w}_{1,2}  \}.
\end{equation}
We consider $S_X^-(v)$ and $S_X^-(w)$ the sets defined in \eqref{strata}. Now, orient the geodesic $\ell$ and assume that $S_X^-(v)$ is to the left of $\ell$ while $S_X^-(w)$ is to the right of $\ell$. According to \cite[II.1.7.3, Limits of support planes]{Epsmarden}, we have the following:
\begin{enumerate}
    \item \label{item11} If $\eta_n$ is a sequence converging to $\eta$ from the left, and $\mathrm{P}_n$ is a support plane at $(\eta,\phi_X^-(\eta_n))$. Then $\mathrm{P}_n$ converges to $\mathrm{P}_v$, see Figure \ref{continuityleft}. 
    \item \label{item22} Similarly, if $\eta_n$ is a sequence converging to $\eta$ from the right, and $\mathrm{P}_n$ is a support plane at $(\eta,\phi_X^-(\eta_n))$. Then $\mathrm{P}_n$ converges to $\mathrm{P}_w$.
\end{enumerate}
\begin{figure}[h]
\centering
\includegraphics[width=.4\textwidth]{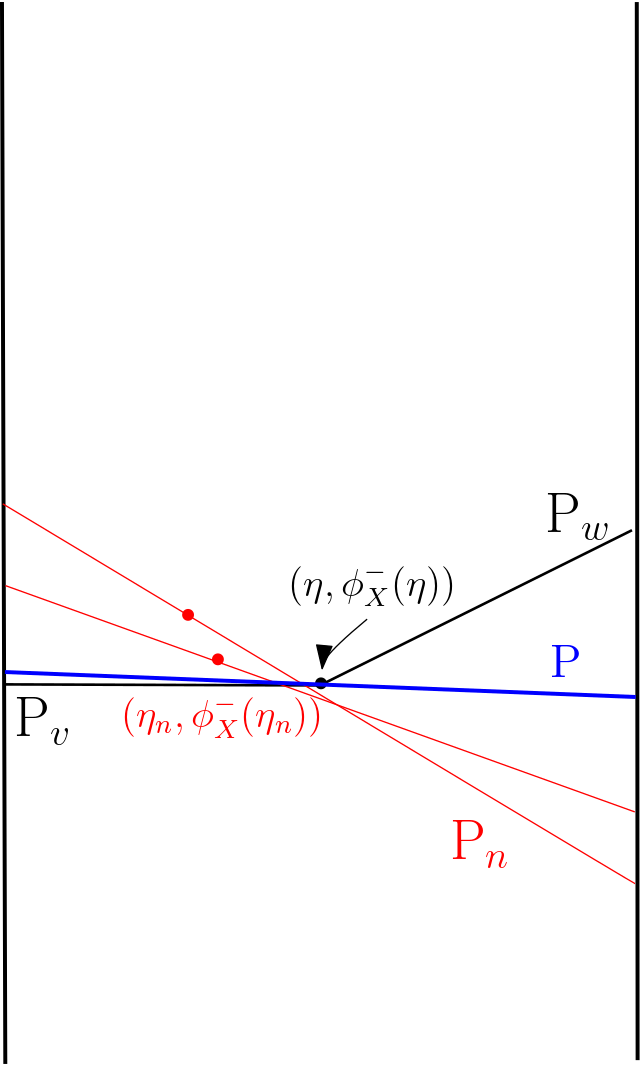}
\caption{A two dimensional picture illustrating the result used. If $\eta_n$ is in the left of $\ell$, (In this two dimensional picture $\ell$ corresponds to the point $(\eta,\phi_X^-(\eta)))$, then from the configuration, the limit of $\mathrm{P}_n$ is a support plane $\mathrm{P}$ at $(\eta,\phi_X^-(\eta))$ such that the left half of $\mathrm{P}$ containing $(\eta,\phi_X^-(\eta))$ is above $\mathrm{P}_v$, and this contradicts the extremality of $\mathrm{P}_v$.}\label{continuityleft}
\end{figure}

Now, we will apply this fact to find a contradiction of the continuity of $\mathcal{E}_X^-$. Take $\eta_n^1$ as a sequence to the left of $\ell$ converging to $\eta$. By the first item, we have:
\begin{equation}\label{contileft}
\mathcal{E}_X^-(\eta_n^1)\to \mathrm{d}_{(1,\eta)}\Pi\left((1,\eta)\boxtimes v\right).
\end{equation}
Similarly, if we take a sequence $\eta_n^2$ converging to $\eta$ from the right, then by the second item, we get:
\begin{equation}\label{continuright}
\mathcal{E}_X^-(\eta_n^2)\to \mathrm{d}_{(1,\eta)}\Pi\left((1,\eta)\boxtimes w\right).
\end{equation}
Hence, for the continuity of $\mathcal{E}_X^-$, we would require:
$$\mathrm{d}_{(1,\eta)}\Pi\left((1,\eta)\boxtimes  (v-w)\right)=0.$$ However, this leads to a contradiction, as it implies that $(1,\eta)\boxtimes (v-w)=0$. This is not possible unless $v=w$. Indeed since $\eta\in\ell$, then from \eqref{elll} $(1,\eta)$ is orthogonal to $v-w$ and, therefore $(1,\eta)\boxtimes (v-w)\neq 0$.

\end{proof}

\begin{proof}[Proof of Proposition \ref{uniqness_of_earthquake}] The proof is inspired from \cite{Thurston}, and the reader can also refer to \cite{mareike}. We treat the case of left infinitesimal earthquakes as the case of right infinitesimal earthquakes can be treated by the same way. First let us prove that the two geodesic laminations $\lambda^-$ and $\lambda^-_1$ are equal. Let $S$ be a stratum for $\lambda_1^-$. Since $\mathcal{E}_1^-$ is a left infinitesimal earthquake, there exists a Killing vector field $\mathrm{K}_{S}$ on $\mathbb{H}^2$ such that $\mathcal{E}^-_{1}|_{S}=\mathrm{K}_{S}$. Let us consider $\mathcal{Y}=\mathcal{E}^-_1-\mathrm{K}_S$. Then $\mathcal{Y}$ is a left infinitesimal earthquake that extends continuously to $Y=X-\mathrm{K}_S$ on $\mathbb{S}^1$. Furthermore, by definition, $\mathcal{Y}$ vanishes on the stratum $S$.

In terms of functions, if $\phi_X$ is the support function of $X$, then the support function of $Y$ is given by: 
$$\phi_Y(z)=\phi_X(z)-\inner{(1,z),v}_{1,2},$$ where $\mathrm{K}_S=\Lambda(v)$ (see Corollary \ref{support_function_of_killing}). We now claim that for all $z\in\mathbb{S}^1,$
\begin{equation}\label{equationextremleft}
\phi_{Y}(z)\geq 0.
\end{equation}
Denote by $\overline{S}$ the closure of $S$ in $\overline{\mathbb{D}^2}$, then $\phi_Y$ vanishes at all points contained in $\overline{S}\cap\mathbb{S}^1$, making the inequality \eqref{equationextremleft} trivially satisfied. For $x\in \mathbb{S}^1\setminus\overline{S}$, we distinguish two cases:
The first case is when $x$ is contained in the closure of some other stratum $S'$ of $\lambda_1^-$, different from $S$. Consider $\mathrm{K}_{S'}$ to be a Killing vector field of $\mathbb{H}^2$ such that $\mathcal{E}^-_{1}|_{S'}=\mathrm{K}_{S'}$. Then,
$$Y(x)=\mathcal{E}_1(x)-\mathrm{K}_{S}(x)=\mathrm{K}_{S'}(x)-\mathrm{K}_{S}(x)=\mathrm{Comp}_1(S,S')(x),$$ where $\mathrm{Comp}_1$ is the comparison vector field of the infinitesimal earthquake $\mathcal{E}_1$. Since $x\in \overline{S'}$ then by Lemma \ref{2.27.}, $\phi_Y(x)\geq 0$.

The second case is when $x$ is not contained in the closure of any stratum. In that case, there exists a sequence $(x_n)_{n\in\mathbb{N}}\in \mathbb{S}^1$ converging to $x$, with $x_n\in\ell_n$ for a leaf $\ell_n$ of $\lambda^-_1$ for all $n\in\mathbb{N}$, thus by continuity of $\phi_Y$, we have
$$\phi_Y(x)=\lim_{n \to +\infty} \phi_Y(x_n).$$ As previously shown, $\phi_Y(x_n)\geq 0$, and thus we get $\phi_Y(x)\geq 0$. This concludes the proof of the claim. So we have proved that: 
$$\phi_X(z)\geq \inner{(1,z),v}_{1,2},\ \mathrm{for}\ \mathrm{all}\ z\in \mathbb{S}^1;$$
according to the properties of the function $\phi_X^-$ (see Property $(\mathrm{P}4)$ in Section \ref{P4}), we have
\begin{equation}\label{supportplane_in_uniq_thm}
    \phi_X^-(\eta)\geq \inner{(1,\eta),v}_{1,2},\ \mathrm{for}\ \mathrm{all}\ \eta\in \mathbb{D}^2.
\end{equation}
Denote by $\partial S=\overline{S}\setminus S\subset \mathbb{S}^1$, then we have 
$$\partial S\subset \{\eta\in\mathbb{S}^1\mid \phi_X^-(\eta)=\langle(1,\eta),v\rangle_{1,2} \}.$$ From the definition of a stratum, the closure $\overline{S}$ of $S$ is compact, connected, then $\overline{S}$ is equal to the convex hull of $\partial S.$ Therefore, by Lemma \ref{structure_of_convex_core}, we deduce that: 
$$\overline{S}\subset \{\eta\in\overline{\mathbb{D}^2}\mid \phi_X^-(\eta)=\langle(1,\eta),v\rangle_{1,2} \}.$$
In conclusion, we have:
\begin{equation}\label{stratum_inclusion}
S\subset S_{X}^-(v).
\end{equation} where
 $$S_{X}^-(v)=\{\eta\in\mathbb{D}^2\mid \phi_X^-(\eta)=\langle(1,\eta),v\rangle_{1,2} \}.$$
In particular, the last inclusion \eqref{stratum_inclusion} implies that we have equality in equation \eqref{supportplane_in_uniq_thm} for points contained in $S$. As a consequence, $\mathrm{P}_v$ is a support plane of $\mathrm{epi}^+(\phi_X^-)$ at points of $S$ and so any stratum of $\lambda_1^-$ is contained in a stratum of $\lambda^-$.

Now, observe that if $S$ is a gap, then $S_{X}^-(v)$ has non empty interior and so for each point $\eta\in S$, the plane $\mathrm{P}_v$ is the unique support plane at $(\eta,\phi_X^-(\eta))$, and so by construction, $\mathcal{E}_X^-(\eta)=\Lambda(v)(\eta)=\mathrm{K}_S(\eta)$ for all $\eta\in S$, hence $\mathcal{E}_X^-$ and $\mathcal{E}_1^-$ coincide on any gap of $\lambda_1^-$.

We claim now that $\lambda^-=\lambda_1^-$. To prove this, remark that if $S$ and $T$ are two gaps of $\lambda_1^-$ contained in $S_{X}^-(v)$, then since $\mathcal{E}_X^-$ and $\mathcal{E}_1^-$ coincide on any gap of $\lambda_1^-$, we have,
$$\mathcal{E}^-_{1}|_{S}=\mathcal{E}^-_{1}|_{T}$$
and hence $\mathrm{Comp}_1(S,T)=\mathrm{Id}$. This can only occur if one stratum is contained in the closure of the other. As both $S$ and $T$ are gaps, it follows that $S= T$. So for each $v\in\minko$ such that $\mathrm{P}_v$ is a support plane of $\mathrm{epi}^+(\phi_X^-)$, the set $S_X^-(v)$ contains at most one gap of $\lambda^-_1$ and any gap of $\lambda_1^-$ is contained in a unique $S_X^-(v)$, therefore the fact that the strata of $\lambda_1^-$ cover all of $\mathbb{D}^2$ then gives us $\lambda_1^-=\lambda^-$.

Having showing that $\lambda_1^-=\lambda^-$, we will now prove the second statement. Note that we have already proved that $\mathcal{E}_1^-$ and $\mathcal{E}_X^-$ coincide on all gaps of $\lambda^-.$ Now will show that they coincide on leaves where $\mathcal{E}_X^-$ is continuous. Consider such leaf that we denote by $S$, then since $S$ is a particular stratum, then we have seen from \eqref{stratum_inclusion} that if $\mathcal{E}^-_1=\Lambda(v)$ on $S$, then 
$\mathrm{P}_v$ is a support plane of $\mathrm{epi}^+(\phi_X^-)$. Now by Lemma \ref{carachterisation_conti_bending}, such support plane is unique, and hence by definition, we have $\mathcal{E}_X^-=\Lambda(v)$, and so $\mathcal{E}_X^-$ and $\mathcal{E}_1^-$ coincide. This concludes the proof.
\end{proof}
The proof of Theorem \ref{TH1} is thus completed.
\subsection{Equivariant infinitesimal earthquakes}\label{equiva_inf_earth}
Let $\Sigma_{g,n}$ be a connected oriented surface of hyperbolic type with genus $g\geq 0$ and $n\geq0$ punctures. We denote by $\pi_1(\Sigma_{g,n})$ the fundamental group of $\Sigma_{g,n}$. In this section, we will see how to associate, to each infinitesimal deformation of a complete hyperbolic metric on $\Sigma_{g,n}$, an equivariant infinitesimal earthquake in a suitable sense. For this purpose, we need to recall some definitions.  A representation $\rho:\pi_1(\Sigma_{g,n})\to\isom(\mathbb{H}^2)$ is \textit{Fuchsian} if $\rho$ is discrete and faithful and which maps loops around punctures to parabolics isometries of $\mathbb{H}^2$. The \textit{Teichmüller space} $\mathcal{T}(\Sigma_{g,n})$ of $\Sigma_{g,n}$ is the space of Fuchsian representations up to conjugacy.
\begin{defi}
   Let $\rho: \pi_1(\Sigma_{g,n})\to\isom(\mathbb{H}^2)$ be a Fuchsian representation. Then we say that map $\tau:\pi_1(\Sigma_{g,n})\to\mathfrak{isom}(\mathbb{H}^2)$ is a cocyle for $\rho$ if $\tau$ satisfies the following condition:  \begin{equation}\label{cocylecondition}
       \tau(\alpha\beta)=\tau(\alpha)+\mathrm{Ad}\rho(\alpha)\cdot\tau(\beta),\end{equation}
for all $\alpha$, $\beta\in \pi_1(\Sigma_{g,n})$. Moreover, we say that $\tau$ is \textit{admissible} if for every loop $\alpha_i$ around a puncture, there exists $\mathfrak{a_i}\in \mathfrak{isom}(\mathbb{H}^2)$ such that: $$\tau(\alpha_i)=\mathrm{Ad}(\rho(\alpha_i))\cdot\mathfrak{a_i}-\mathfrak{a_i}.$$
\end{defi}
Furthermore, we define:
\begin{itemize}
    \item $\mathrm{Z}^1_{\mathrm{Ad}\rho}(\pi_1(\Sigma_{g,n}), \mathfrak{isom}(\mathbb{H}^2))$: the space of admissible cocycles.
\item $\mathrm{B}^1_{\mathrm{Ad}\rho}(\pi_1(\Sigma_{g,n}), \mathfrak{isom}(\mathbb{H}^2))$: the space of \textit{coboundaries} for $\rho$. These are admissible cocycles $\tau:\pi_1(\Sigma_{g,n})\to \mathfrak{isom}(\mathbb{H}^2)$ of the form: $$\tau(\alpha)=\mathrm{Ad}(\rho(\alpha))\cdot\mathfrak{a}-\mathfrak{a},$$ for all $\alpha\in \pi_1(\Sigma_{g,n})$.
\item We then define the vector space $\mathrm{H}^1_{\mathrm{Ad}\rho}(\pi_1(\Sigma_{g,n}),\mathfrak{isom}(\mathbb{H}^2))$ as the quotient 
$$\mathrm{H}^1_{\mathrm{Ad}\rho}(\pi_1(\Sigma_{g,n}),\mathfrak{isom}(\mathbb{H}^2))=\frac{\mathrm{Z}^1_{\mathrm{Ad}\rho}(\pi_1(\Sigma_{g,n}),\mathfrak{isom}(\mathbb{H}^2))}{\mathrm{B}^1_{\mathrm{Ad}\rho}(\pi_1(\Sigma_{g,n}),\mathfrak{isom}(\mathbb{H}^2))}.$$
 \end{itemize}
In fact, an admissible cocycle should be thought of as a derivative of a smooth path of Fuchsian representations. If $t \to \rho_t$ is such a smooth path with $\rho_0 = \rho$, then one can check that:
\begin{equation}\label{derivative}
\tau(\alpha) = \frac{d}{dt}|_{t=0} \rho_t(\alpha) \rho(\alpha)^{-1},
\end{equation}
satisfies the cocycle condition \eqref{cocylecondition}. Now, since there is only one conjugacy class of parabolic isometries in $\mathrm{Isom}(\mathbb{H}^2)$, then for each loop $\alpha_i$ around a puncture, there exists a smooth family of isometries $g_t^i$ such that $\rho_t(\alpha_i) = (g_t^i)^{-1} \rho(\alpha_i) g_t^i$ for some smooth path $g_t^i$ of isometries of $\mathbb{H}^2$. Taking the derivative in \eqref{derivative}, we get:
$$\tau(\alpha_i) = \mathrm{Ad}(\rho(\alpha_i)) \cdot \mathfrak{a_i} - \mathfrak{a_i},$$ where $\mathfrak{a_i}$ is the tangent vector to $g_t^i$ at $t=0$ and so $\tau$ is an admissible cocycle.
A celebrated theorem of Goldman \cite{Goldman_tangent_space} asserts that, when $\Sigma_{g,n}$ is closed (i.e, $n=0$), the tangent space of $\mathcal{T}(\Sigma_{g,n})$ at $[\rho]$ is identified with $\mathrm{H}^1_{\mathrm{Ad}\rho}(\pi_1(\Sigma),\mathfrak{isom}(\mathbb{H}^2))$. See \cite[Section 4.2]{arnaud} for a discussion of Goldman's Theorem in the case of punctured surfaces.

Following \cite{DGK_complete_lorentz}, we have this definition.
\begin{defi}
 Let $\rho:\pi_1(\Sigma_{g,n})\to\isom(\mathbb{H}^2)$ be a Fuchsian representation and $\tau$ a cocyle. We say that a vector field $V$ on $\mathbb{H}^2$ is $(\rho, \tau)$-\textit{equivariant} if for all $\alpha\in\pi_1(\Sigma_{g,n})$, we have:
\begin{equation}V=\rho(\alpha)_*V+\tau(\alpha).\end{equation}
\end{defi}

\begin{remark}
 Let $\rho_t$ be a smooth family of Fuchsian representations with $\rho_0=\rho$. Consider the admissible cocycle $\tau(\alpha) = \frac{d}{dt}|_{t=0} \rho_t(\alpha) \rho(\alpha)^{-1}.$ Then for any smooth family $f_t:\mathbb{H}^2\to\mathbb{H}^2$ such that $f_t$ is $(\rho,\rho_t)$ equivariant and $f_0=\mathrm{id}_{\mathbb{H}^2}$. The derivative $V(\eta):=\frac{d}{dt}|_{t=0} f_t(\eta)$ is $(\rho,\tau)-$equivariant vector field on $\mathbb{H}^2$. Indeed, let $\eta\in\mathbb{H}^2$ and $\alpha\in\pi_1(\Sigma_{g,n})$, by equivariance of $f_t$, we have:

 $$f_t(\rho(\alpha)\cdot \eta)=\left(\rho_t(\alpha)\rho(\alpha)^{-1}\right)\rho(\alpha)\cdot f_t(\eta).$$
 Therefore, taking the derivative at $t=0$, we get:
\begin{align*}
V(\rho(\alpha)\cdot \eta)&= \left( \frac{d}{dt}|_{t=0} \rho_t(\alpha) \rho(\alpha)^{-1}   \right)(\rho(\alpha)\cdot \eta)+\mathrm{d}(\mathrm{id}_{\mathbb{H}^2})\left(\frac{d}{dt}|_{t=0} \rho(\alpha)\cdot f_t(\eta) \right)   \\
 &=\tau(\alpha)(\rho(\alpha)\cdot \eta)+\mathrm{d}\rho(\alpha)(V(\eta)).
\end{align*}
Equivalently$$ V=\rho(\alpha)_*V+\tau(\alpha).$$
\end{remark}

We can now state the main result of this section, which can be seen as the analogue of Kerckhoff's Theorem \cite[Theorem 3.5]{kerker}, which states that any tangent vector in the Teichmüller space is tangent to a unique earthquake path.
\begin{prop}\label{kerkhoffvector}
Let $\rho:\pi_1(\Sigma_{g,n})\to\mathrm{Isom}(\mathbb{H}^2)$ be a 
Fuchsian representation and take an admissible cocycle $\tau\in\mathrm{Z}^1_{\mathrm{Ad}\rho}(\pi_1(\Sigma_{g,n}),\mathfrak{isom}(\mathbb{H}^2))$. Then there exists a $(\rho,\tau)$-equivariant left (resp. right) infinitesimal earthquake $\mathcal{E}$ of $\mathbb{H}^2$ which extends continuously to a continuous vector field on $\mathbb{S}^1$. The left (resp. right) $(\rho,\tau)$-infinitesimal earthquake $\mathcal{E}$ is unique except on the leaves of the geodesic lamination $\lambda$ associated to $\mathcal{E}$, where $\mathcal{E}$ is discontinuous. 
\end{prop}

Let $\rho$ be a Fuchsian representation, and consider an admissible cocycle $\tau$. Observe that from the equivariance property \eqref{adequivariance} of the isomorphism $\Lambda: \minko \to \mathfrak{isom}(\mathbb{H}^2)$, the map $\overline{\tau}: \pi_1(\Sigma_{g,n}) \to \minko$ given by $\Lambda \circ \overline{\tau} = \tau$ satisfies the following relation:
$$\overline{\tau}(\alpha\beta)=\overline{\tau}(\alpha)+\rho(\alpha)\overline{\tau}(\beta),$$
for all $\alpha$, $\beta\in \pi_1(\Sigma_{g,n})$. This relation allows us to define an action on Half-pipe space given by:
\begin{equation}\label{dualrepresentation}
    \begin{array}{ccccc}
 \mathrm{Is}(\rho,\overline{\tau})& : & \pi_1(\Sigma_{g,n}) & \to & \isom(\HP) \\
 & & \alpha & \mapsto & \begin{bmatrix}
\rho(\alpha) & 0  \\
^T \overline{\tau}(\alpha)\mathrm{J}\rho(\alpha) & 1 
\end{bmatrix}, \\
\end{array}
\end{equation} see \eqref{groupeduality}. To prove Proposition \ref{kerkhoffvector}, we need this fact due to Nie and Seppi.

\begin{theorem}\cite{affine}\label{nieseppi}
Let $\rho: \pi_1(\Sigma_{g,n}) \to \mathrm{Isom}(\mathbb{H}^2)$ be a Fuchsian representation, and let $\tau$ be an admissible cocycle for $\rho$. Consider the map $\overline{\tau}: \pi_1(\Sigma_{g,n}) \to \minko$ as described above. Then, there exists a bijective correspondence between $[0,+\infty[^n$ (where $n$ is the number of punctures) and lower (resp. upper) semicontinuous functions $\phi: \mathbb{S}^1 \to \mathbb{R}$ with graphs preserved by the representation $\mathrm{Is}(\rho, \overline{\tau})$ as defined in equation \eqref{dualrepresentation}.

Furthermore, among all these functions, there exists a unique continuous function $\phi_{\rho,\tau}: \mathbb{S}^1 \to \mathbb{R}$ with graph preserved by the representation $\mathrm{Is}(\rho, \overline{\tau})$.
\end{theorem}
\begin{proof}[Proof of Proposition \ref{kerkhoffvector}]\label{proofkerkhoffvector}
Let $\rho:\pi_1(\Sigma_{g,n})\to\mathrm{Isom}(\mathbb{H}^2)$ be a Fuchsian representation and $\tau\in\mathrm{Z}^1_{\mathrm{Ad}\rho}(\pi_1(\Sigma_{g,n}),\mathfrak{isom}(\mathbb{H}^2))$ be a cocycle. Then, by Proposition \ref{nieseppi}, there is a unique continuous function $\phi_{\rho,\tau}:\mathbb{S}^1\to \mathbb{R}$ with a graph preserved by the representation $\mathrm{Is}(\rho,\overline{\tau})$, where $\Lambda\circ\overline{\tau}=\tau$. We consider $X$ to be the vector field on the circle for which its support function is given by $\phi_{\rho,\tau}$, namely: $$X(z)=iz\phi_{\rho,\tau}(z).$$ Now we consider $\mathcal{E}_X^-$ and $\mathcal{E}_X^+$, the left and right infinitesimal earthquakes that extend to $X$ continuously. Following Lemma \ref{4.7}, we have:
\begin{equation}
        \mathrm{A}_*\mathcal{E}_{X}^{\pm}+\mathcal{E}^{\pm}_{\Lambda(v)}=\mathcal{E}^{\pm}_{\mathrm{A}_*X+\Lambda(v)},
    \end{equation}
for all $\mathrm{A}\in \mathrm{O}_0(1,2)$ and $v\in\minko$. In particular, we have:
\begin{equation}\label{courage}
\rho(\alpha)_*\mathcal{E}_{X}^{\pm}+\mathcal{E}^{\pm}_{\tau(\alpha)}=\mathcal{E}^{\pm}_{\rho(\alpha)_*X+\tau(\alpha)}.
\end{equation}
However, since the graph of $\phi$ is preserved by $\mathrm{Is}(\rho,\overline{\tau})$, then by Lemma \ref{equii}, we have: $$X=\rho(\alpha)_*X+\tau(\alpha).$$
Now, by Example \ref{exx1}, we have $\mathcal{E}^{\pm}_{\tau(\alpha)}=\tau(\alpha)$, thus equation \eqref{courage} becomes
    $$\mathcal{E}_X^{\pm}=\rho(\alpha)_*\mathcal{E}_{X}^{\pm}+\tau(\alpha).$$ This means that $\mathcal{E}_X^{\pm}$ is $(\rho,\tau)$-equivariant. The uniqueness part follows form Proposition \ref{uniqness_of_earthquake}. This concludes the proof.
\end{proof}

\begin{remark}
For a Fuchsian representation $\rho$ and an admissible cocycle $\tau$, one can prove, in the same way as in Proposition \ref{kerkhoffvector}, that for each lower semi-continuous function $\phi: \mathbb{S} \to \mathbb{R}$ with graph preserved by $\mathrm{Is}(\rho,\overline{\tau})$ (see Theorem \ref{nieseppi}), there exists a left (resp. right) infinitesimal earthquake of $\mathbb{H}^2$ which is $(\rho, \tau)$-equivariant and extends radially to $X(z) = iz\phi(z)$. However, we don't have the uniqueness part since Proposition \ref{uniqness_of_earthquake} holds only for continuous vector fields.  
\end{remark}

\section{Zygmund vector field $\implies$ Finite width}\label{section4}
The purpose of this section is to show that if $X$ is a Zygmund vector field, then the width of $X$ is finite. More precisely, we will show the following quantitative estimate.
\begin{prop}\label{zygmund_to_width}
 Given any Zygmund vector field $X$ of $\mathbb{S}^1$, let $w(X)$ be the width of $X$. Then
 $$w(X)\leq \frac{8}{3}\lVert X\rVert_{cr}.$$
\end{prop}
The key step to prove Proposition \ref{zygmund_to_width} is the following Lemma due to Fan and Hu.

\begin{lemma}\cite[Lemma 4]{FanJun}\label{funju}
    Let $X$ be a Zygmund vector field of $\mathbb{S}^1$ such that $X(1,0)=X(0,1)=X(-1,0)=0$. Then 
    $$\max_{z\in \mathbb{S}^1}\lVert X(z)\rVert_{euc}\leq\frac{4}{3}\lVert X\rVert_{cr},$$ where $\lVert\cdot \rVert_{euc}$ denotes the standard euclidean norm in $\mathbb{R}^2$.
\end{lemma}

\begin{proof}[Proof of Proposition \ref{zygmund_to_width}]
Let $X$ be a Zygmund vector field of $\mathbb{S}^1$ and $\phi_X$ be its support function, then by the definition of the width \eqref{width_with_function}, we have
$$w(X):=\underset{\eta\in \mathbb{D}^2}{\sup} \mathrm{F}_X(\eta),$$
where $\mathrm{F}_X:\mathbb{D}^2\to\mathbb{R}$ is the function defined by $\mathrm{F}(\eta)= \mathrm{L}(\eta,\phi_X^+(\eta))-\mathrm{L}(\eta,\phi_X^-(\eta))$. 
Using the formula \eqref{formule_L} of the height function $\mathrm{L}$, we deduce that: $$\mathrm{F}_X(0,0)=\phi_X^+(0,0)-\phi_X^-(0,0).$$
First, let us assume that $X$ satisfies the hypothesis of Lemma \ref{funju}, namely 
$X(1,0)=X(0,1)=X(-1,0)=0.$ Since
$$-\max_{z\in \mathbb{S}^1}\lvert \phi_X(z)\rvert\leq  \phi_X(z)\leq \max_{z\in \mathbb{S}^1}\lvert \phi_X(z)\rvert,$$ the properties of $\phi_X^{\pm}$ (see Page $11$ in Section \ref{page 8}) then imply that $\phi_X^+\leq \max_{z\in \mathbb{S}^1}\lvert \phi_X(z)\rvert$ and $-\max_{z\in \mathbb{S}^1}\lvert \phi_X(z)\rvert\leq \phi_X^-$
and hence 
$$
    \mathrm{F}_{X}(0,0)\leq 2\max_{z\in \mathbb{S}^1}\lvert \phi_X(z)\rvert.
$$
Since $X(z)=iz\phi_X(z)$, then the euclidean norm of $X$ is the same as the modulus of $\phi_X$ and thus by Proposition \ref{funju}, we get:
\begin{equation}\label{Fx}
    \mathrm{F}_{X}(0,0)\leq \frac{8}{3}\lVert X\rVert_{cr}.
\end{equation}
The equation \eqref{Fx} holds for any Zygmund vector fields which vanishes on $(1,0)$, $(0,1)$ and $(-1,0)$. But since the cross ratio norm and $\mathrm{F}_{X}$ are invariant by translation by Killing vector fields (see Remark \ref{remarqueabc}), then the inequality\eqref{Fx} holds for any Zygmund vector field.
In order to prove Proposition \ref{zygmund_to_width}, we need to show that $\mathrm{F}_X(\eta)\leq\frac{8}{3}\lVert X\rVert_{cr}$ holds for any $\eta\in \mathbb{D}^2,$ and then passing to the the supremum we get the desired inequalities. This follows from the fact that:
\begin{equation}\label{5}
    \mathrm{F}_{\mathrm{A}_*X}(\eta)=\mathrm{F}_X(\mathrm{A}^{-1}\cdot\eta)
\end{equation}
for all $\mathrm{A}\in\mathrm{O}_0(1,2)$ and $\eta\in \mathbb{D}^2.$ Indeed
\begin{align*}
\mathrm{F}_{\mathrm{A}_*X}(\eta)&= \mathrm{L}(\eta,\phi_{\mathrm{A}_*X}^+(\eta))-\mathrm{L}(\eta,\phi_{\mathrm{A}_*X}^-(\eta))\\
&= \mathrm{L}\left(\mathrm{Is}(\mathrm{A}^{-1},0)(\eta,\phi_{\mathrm{A}_*X}^+(\eta)\right)-\mathrm{L}\left(\mathrm{Is}(\mathrm{A}^{-1},0)(\eta,\phi_{\mathrm{A}_*X}^-(\eta)\right)\\
&= \mathrm{L}(\mathrm{A}^{-1}\cdot\eta,\phi_{X}^+(\mathrm{A}^{-1}\cdot\eta))-\mathrm{L}(\mathrm{A}^{-1}\cdot\eta,\phi_{X}^-(\mathrm{A}^{-1}\cdot\eta))\\
&=\mathrm{F}_X(\mathrm{A}^{-1}\cdot\eta),
\end{align*}
where we used in the first equality that the distance along the fiber is invariant by $\mathrm{Isom}(\HP)$ (see equation \eqref{equivariant_length} in the proof of Lemma \ref{width_equivariant}). In the second equality, we used equation \eqref{17} in Lemma \ref{equii}. This completes the proof.
\end{proof}

\section{Finite width $\implies$ Bounded lamination}\label{section5}
In this section we will prove a quantitative estimate between the width of a vector field and the Thurston norm of its infinitesimal earthquake measure. First let us briefly recall the definition of infinitesimal earthquake measure. For a given vector field $X$, we have seen in Proposition \ref{earthquake_properties} that $\mathcal{E}_X^{\pm}$ is an infinitesimal earthquake along a geodesic lamination $\lambda^{\pm}$. Recall that this geodesic lamination is given by the connected components of $S_X^{\pm}(v)\setminus \mathrm{Int}(S_X^{\pm}(v))$, as $\mathrm{P}_v$ varies over all support planes of $\mathrm{gr}(\phi_X^{\pm}|_{\mathbb{D}^2})$, where $$S_X^{\pm}(v)=\{\eta\in\mathbb{D}^2\mid \phi^{\pm}_X(\eta)=\langle(1,\eta),v\rangle_{1,2} \},$$ (see identity \eqref{strata} in the proof of  Proposition \ref{earthquake_properties}). The geodesic lamination $\lambda^{\pm}$ has a naturally associated transverse measure called the \textit{bending measure} of the \textit{pleated surface} $\mathrm{gr}(\phi_X^{\pm}|_{\mathbb{D}^2})$. This is described in detail in \cite[Section 8]{Epstein} in the case of pleated surfaces in hyperbolic space. In our case, the bending measure is described in work of Mess \cite{Mess}, see also \cite[Section 2.5]{BS17}. Here, we recall briefly the construction. We define the \textit{"normal bundle"} of the surface $\mathrm{gr}(\phi_X^{\pm}|_{\mathbb{D}^2})$ as follow:
$$\mathcal{N}\mathrm{gr}(\phi_X^{\pm}|_{\mathbb{D}^2}):=\{ (x,v)\in\mathrm{gr}(\phi_X^{\pm}|_{\mathbb{D}^2})\times\minko \mid \mathrm{P}_v \ \mathrm{is} \ \mathrm{a}\ \mathrm{support}\ \mathrm{plane}\ \mathrm{of}\ \mathrm{gr}(\phi_X^{\pm}|_{\mathbb{D}^2})\ \mathrm{at}\ x      \}.$$ Let $c:[0,1]\to \mathbb{D}^2$ be an arc. A \textit{polygonal approximation} of $c$ on $\mathrm{gr}(\phi_X^{\pm})$ is a sequence 
$$\mathcal{P}=\{(\phi_X^{\pm}(c(t_i)),v_i)  \}\subset \mathcal{N}\mathrm{gr}(\phi_X^{\pm}|_{\mathbb{D}^2}),\ 0=t_0<t_1<\ldots <t_n=1,$$ where $\mathrm{P}_{v_i}$ is a support plane of $\mathrm{gr}(\phi_X^{\pm}|_{\mathbb{D}^2})$ at $\phi_X^{\pm}(c(t_i))$. For each $i=0,\ldots ,n$, denote by $\theta_i$ the angle between $\mathrm{P}_{v_i}$ and $\mathrm{P}_{v_{i+1}}$ (see Proposition \ref{hpangle}). This angle could be identically zero if $\mathrm{P}_{v_i}=\mathrm{P}_{v_{i+1}}$.

\begin{defi}\label{def6.1}
Let $X$ be a continuous vector field of $\mathbb{S}^1$ and consider $c:[0,1]\to \mathbb{D}^2$ an arc, then we define:
 \begin{itemize}
     \item The bending measure of $\mathrm{gr}(\phi_X^-)$ is a transverse measure on $\lambda^-$ defined by: $$ \lambda^{-}(c):=\inf_{\mathcal{P}}\sum_{i=0}^n\theta_i,$$ where $\mathcal{P}$ runs over all polygonal approximations of $c$ on $\mathrm{gr}(\phi_X^{-})$.
      \item Similarly, the bending measure of $\mathrm{gr}(\phi_X^+)$ is a transverse measure on $\lambda^+$ defined by: $$ \lambda^{+}(c):=\inf_{\mathcal{P}}\sum_{i=0}^n\theta_i,$$ where $\mathcal{P}$ runs over all polygonal approximations of $c$ on $\mathrm{gr}(\phi_X^{+})$.
 \end{itemize}
Let $\mathcal{E}_X^{\pm}$ be the infinitesimal earthquake that extends to $X$ continuously. Then we define the \textit{infinitesimal earthquake measure} of $\mathcal{E}_X^{\pm}$ as the bending measure $\lambda^{\pm}$.
\end{defi}

\begin{remark}\label{HPBENDING}
Let $c:[0,1]\to\mathbb{D}^2$ be an arc, a particular case of polygonal approximation of $c$ on $\mathrm{gr}(\phi_X^{\pm})$ is $$\mathcal{P}=\{(\phi_X^{\pm}(c(0)),v_0), (\phi_X^{\pm}(c(1)),v_1)   \},$$ where $\mathrm{P}_{v_i}$ is a support plane of $\mathrm{gr}(\phi_X^{\pm})$ at $\phi_X^{\pm}(c(t_i))$, for $i=0,1.$ Hence by definition we have
$$\lambda^{\pm}(c)\leq \lVert v_0-v_1 \rVert,$$ where we recall that by Proposition \ref{hpangle} the angle between $\mathrm{P}_{v_0}$ and $\mathrm{P}_{v_1}$ is given by $\lVert v_0-v_1 \rVert$, where $\lVert v_0-v_1\rVert= \sqrt{\inner{v_0-v_1,v_0-v_1}_{1,2}}$.
\end{remark}
We have the following estimate.
\begin{prop}\label{lamination_petit_que_width}
Let $X$ be a continuous vector field of finite width. Let $\lambda^{\pm}$ be the infinitesimal earthquake measure of $\mathcal{E}^{\pm}_X$. Then: 
    $$\lVert\lambda^{\pm}\rVert_{\mathrm{Th}}\leq \frac{2\sqrt{2}}{1-\tanh(1)}w(X).$$
\end{prop}
To establish Proposition \ref{lamination_petit_que_width}, we rely on the following fundamental Lemma:
\begin{lemma}\label{simple_result}
    Let $X$ be a continuous vector field of $\mathbb{S}^1$ of finite width $w(X)$. Let $\phi_X:\mathbb{S}^1\to \mathbb{R}$ be its support function. Assume that $\mathbb{D}^2\times\{0\}$ is a support plane of $\mathrm{gr}(\phi_X^-)$ at $(0,0,\phi_X^{-}(0,0))$, (and so $\phi_X^-(0,0)=0$). Then for all $z\in\mathbb{S}^1$ we have:
    $$\phi_X(z)\leq 2 w(X). $$
\end{lemma}
\begin{proof}
First, note that since $\mathbb{D}^2\times\{0\}$ is a support plane of $\mathrm{epi}^+(\phi_X^-)$, then for all $\eta\in \overline{\mathbb{D}^2}$: 
$$\phi_X^-(\eta)\geq 0.$$ In particular, $\phi_X\geq 0.$ Let $z\in \mathbb{S}^1$ and let $\mathcal{L}$ be the segment in $\mathbb{D}^2\times\mathbb{R}$ with endpoints at $(-z,0)$ and $(z,\phi_X(z))$. Let $(0,0,m)$ be the unique intersection point between the segment $\mathcal{L}$ and the vertical line $\{(0,0)\}\times\mathbb{R}$ (see Figure \ref{thales}).
\begin{figure}[htb]
\centering
\includegraphics[width=.5\textwidth]{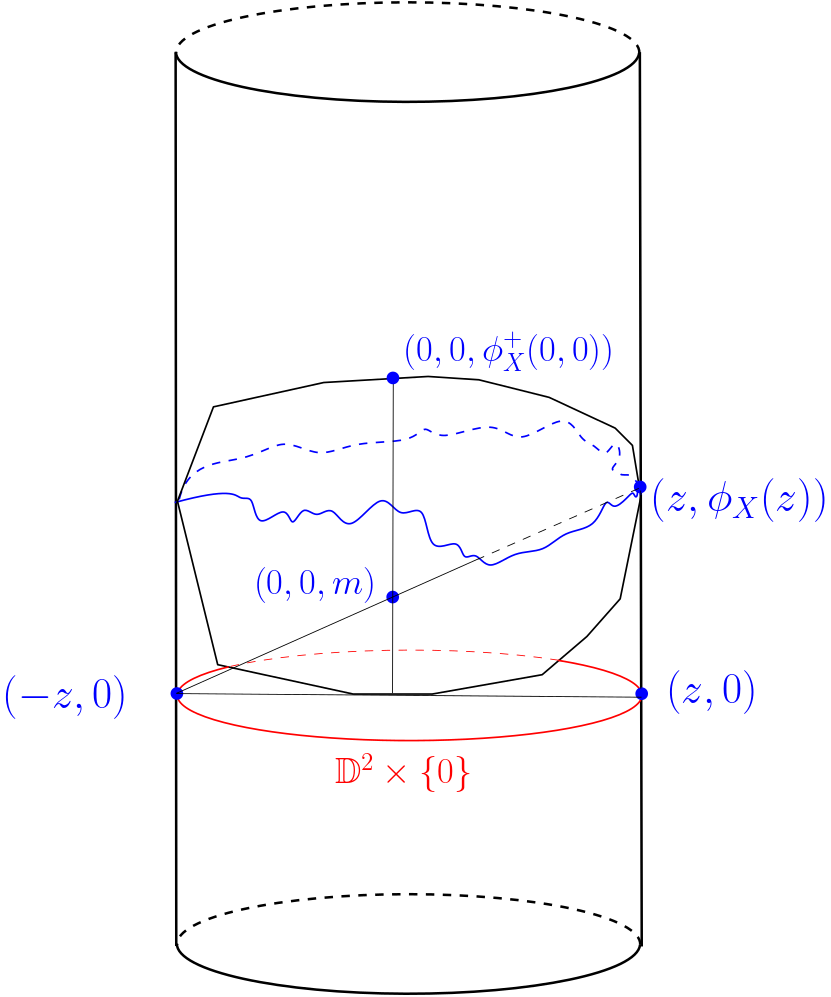}
\caption{Illustration for the proof of Lemma \ref{simple_result}}\label{thales}
\end{figure}

By elementary computation, we find that $m=\frac{1}{2}\phi_X(z)$. Now, due to the concavity of $\phi_X^+$, we have
$$\phi_X^+(0,0)\geq\frac{1}{2}\phi_X^+(-z)+\frac{1}{2}\phi_X^+(z).$$
Since $\phi_X^+=\phi_X$ on $\mathbb{S}^1$, we obtain: $$\phi_X^+(0,0)\geq\frac{1}{2}\phi_X(-z)+\frac{1}{2}\phi_X(z).$$
Hence, by the positivity of $\phi_X$, we deduce that $\phi_X^+(0,0)\geq m$. Since $\mathbb{D}^2\times\{0\}$ is a support plane at the point $(0,0,\phi_X^-(0,0))$, then by definition $\phi_X^{-}(0,0)=0$ and hence $$m= \phi_X^+(0,0)-\phi^-_X(0,0)\leq w(X),$$
where we recall that $w(X):=\underset{\eta\in\mathbb{D}^2}{\sup} \left( \mathrm{L}(\eta,\phi_X^+(\eta))-\mathrm{L}(\eta,\phi_X^{-}(\eta))\right)$. Therefore: $$\phi_X(z)=2m\leq 2w(X).$$ This concludes the proof.
\end{proof}
With Lemma \ref{simple_result} in place, we now possess all the necessary tools to prove Proposition \ref{lamination_petit_que_width}.
\begin{proof}[Proof of Proposition \ref{lamination_petit_que_width}]

Let $\lambda^-$ be the infinitesimal earthquake measure of $\mathcal{E}^-_X$. Let $I=[x,y]$ be a geodesic segment in $\mathbb{D}^2$ of length one and transverse to $\lambda$. Let $v$ and $w\in\minko$ such that $\mathrm{P}_v$ and $\mathrm{P}_w$ are support planes of $\mathrm{epi^+}(\phi_X^-)$ at $(x,\phi_X^-(x))$ and $(y,\phi_X^-(y))$, respectively. As the width of a vector field is invariant under the action of $\mathrm{O}_0(1,2)\ltimes\minko$ (see Lemma \ref{width_equivariant}), we can assume that $v=0$ so that $\mathrm{P}_v=\mathbb{D}^2\times\{0\}$, $x=(0,0)$, and $I$ is contained in the horizontal line joining $(-1,0)$ and $(1,0)$, which implies that $y=(\tanh(1),0)$.
By Lemma \ref{planes_intersect}, the planes $\mathrm{P}_v$ and $\mathrm{P}_w$ intersect in $\HP$, then $v-w$ is spacelike. It follows from Remark \ref{HPBENDING} that:
   \begin{equation}\label{53}
       \lambda^-(I)\leq \lVert v-w\rVert=\lVert w\rVert.\end{equation} If $w=(w_0,w_1,w_2)$ then since $(y,\phi_X^-(y))\in\mathrm{P}_w$ we get:
$$  -w_0+\tanh(1)w_1=\phi_X^-(y),$$
however since $\mathbb{D}^2\times\{0\}$ is a support plane then $\phi_X^-(y)\geq 0$ and so: \begin{equation}\label{54}
-w_0+\tanh(1)w_1\geq0.\end{equation} Now since $\mathrm{P}_w$ is a support plane for $\mathrm{epi^+}(\phi_X^-)$, then for all $\eta\in \overline{\mathbb{D}^2},$ we have:
\begin{equation}\label{55}
    \inner{(1,\eta),w}\leq\phi_X^-(\eta).
\end{equation}
By applying this inequality for $\eta\in\{(1,0),(-1,0),(0,1),(0,-1)\}$ and using Lemma \ref{simple_result}, we derive the following inequalities:
\begin{equation}\label{56}
    -w_0+w_1\leq 2w(X).
\end{equation}
\begin{equation}\label{57}
    -w_0-w_1\leq 2w(X).
\end{equation}
\begin{equation}\label{58}
    -w_0+w_2\leq 2w(X).
\end{equation}
\begin{equation}\label{59}
    -w_0-w_2\leq 2w(X).
\end{equation}
Our goal now is to use these inequalities to bound the norm of each coordinate of the vector $w$ by the width of $X$.
First using inequalities \eqref{54} and \eqref{56}, we obtain:
$$w_0\leq \frac{2\tanh(1)}{1-\tanh(1)}w(X).$$ Adding the inequalities \eqref{56} and \eqref{57}, we get:
$$-2w(X)\leq w_0.$$ Therefore we conclude:

\begin{equation}\label{60}
    \lvert w_0\rvert \leq \frac{2\tanh(1)}{1-\tanh(1)}w(X).
\end{equation}
Now, from inequalities \eqref{54} and \eqref{60}, we have: 
$$\frac{-2}{1-\tanh(1)}w(X)\leq w_1.$$ From inequalities \eqref{56} and \eqref{60}, we get:
$$w_1\leq \frac{2}{1-\tanh(1)}w(X).$$ Hence we can establish:
\begin{equation}\label{61}
    \lvert w_1\rvert \leq \frac{2}{1-\tanh(1)}w(X).
\end{equation}
Now according to inequalities \eqref{58} and \eqref{60} we have:
$$w_2\leq\frac{2}{1-\tanh(1)}w(X).$$
Finally from inequalities \eqref{59} and \eqref{60} we get:
$$\frac{-2}{1-\tanh(1)}w(X)\leq w_2.$$
Therefore: \begin{equation}\label{62}
    \lvert w_2\rvert\leq \frac{2}{1-\tanh(1)}w(X).
\end{equation}
Hence, by combining inequalities \eqref{61} and \eqref{62}, the Minkowski norm of $w$ satisfies:

\begin{align*}
\lVert w\rVert&=\sqrt{-w_0^2+w_1^2+w_2^2}\\
&\leq \sqrt{w_1^2+w_2^2}\\
&\leq \frac{2\sqrt{2}}{1-\tanh(1)}w(X).
\end{align*}
From equation \eqref{53}, we can then deduce that: $$\lambda^-(I)\leq \frac{2\sqrt{2}}{1-\tanh(1)}w(X).$$ The proof for estimating $\lambda^+$ can be done in a similar manner. The proof is complete.
\end{proof}
The proof of Theorem \ref{TH2} is then completed by combining Propositions \ref{lamination_petit_que_width} and \ref{zygmund_to_width}.

\section{Length and Thurston norm}\label{sec7}
In this section we will work on a closed surface $\Sigma_g$ of genus $g\geq 2$. For each Fuchsian representation $\rho:\pi_1(\Sigma)\to\isom(\mathbb{H}^2)$, we denote by $S_{\rho}$ the quotient: $$S_{\rho}:=\mathbb{H}^2/\rho(\pi_1(\Sigma_g)).$$
\begin{defi}
Let $\rho:\pi_1(\Sigma)\to\isom(\mathbb{H}^2)$ be a Fuchsian representation. Then we say that $\lambda$ is a measured geodesic lamination on $S_{\rho}$ if $\lambda$ lifts to a $\rho(\pi_1(\Sigma_g))-$invariant measured geodesic lamination in $\mathbb{H}^2.$   
\end{defi}
We denote by $\mathcal{ML}(S_{\rho})$ the set of measured geodesic laminations on the closed hyperbolic surface $S_{\rho}$. A \textit{weighted multicurve} on $S_{\rho}$ is a geodesic lamination $\lambda$ of the form:
$$\lambda=\sum_{i=1}^ka_i\alpha_i,$$ where $\alpha_i$ are non-trivial simple closed geodesics, pairwise non-homotopic, and $a_i>0$. For each closed geodesic $\alpha$ of $S_{\rho}$, we define the length $\ell_{\rho}(\alpha)$ to be the hyperbolic length of the geodesic $\alpha$. It is known that the set of weighted measured geodesic laminations is dense in $\mathcal{ML}(S_{\rho})$. Kerckhoff \cite{Kerearth} has shown that the length function can be extended to general measured geodesic lamination $\lambda\in\mathcal{ML}(S_{\rho})$. More precisely, for each Fuchsian representation $\rho$, we define the length function of $\rho$
$$\begin{array}{ccccc}
\ell_{\rho} & : & \mathcal{ML}(S_{\rho}) & \to & [0,+\infty[\\
 & & \lambda & \mapsto & \ell_{\rho}(\lambda) \\
\end{array}$$
to be the unique continuous function such that, for every weighted multicurve $\lambda=\sum_{i=1}^ka_i\alpha_i$, we have
$$\ell_{\rho}(\lambda)=\sum_{i=1}^ka_i\ell_{\rho}(\alpha_i).$$
The goal of this section is to prove Proposition \ref{lengthvsThurs}. Let us now fix once and for all a Fuchsian representation $\rho$.
\begin{prop}
Let $\lambda$ be a measured geodesic lamination on the hyperbolic surface $S_{\rho} := \mathbb{H}^2/\rho(\pi_1(\Sigma_g))$. Then, the following inequality holds:
\begin{equation}
\ell_{\rho}(\lambda) \leq \frac{64\pi(3e-1)}{3(e-1)}(g-1) \lVert \lambda\rVert_{\mathrm{Th}}^{\rho}.
\end{equation}
\end{prop}
We recall that the Thurston norm $\lVert \lambda\rVert_{\mathrm{Th}}^{\rho}$
is defined as the Thurston norm (in the sense of Definition \ref{thursnorm2}) of the lift of $\widetilde{\lambda}$ in $\mathbb{H}^2$ by the covering map $\mathbb{H}^2\to S_{\rho}$. Observe that by co compactness of the action of $\rho(\pi_1(\Sigma_g))$ on $\mathbb{H}^2$, the measured geodesic lamination $\widetilde{\lambda}$ is bounded.
In \cite{Mess}, Mess identifies the set of measured geodesic laminations $\mathcal{ML}(S_{\rho})$ with the tangent space $\mathrm{T}_{\rho}\mathcal{T}(\Sigma_g)$ at $\rho$ of Teichm\"uller space, which is $\mathrm{H}^1_{\mathrm{Ad}\rho}(\pi_1(\Sigma),\mathfrak{isom}(\mathbb{H}^2))$, (see Section \ref{equiva_inf_earth}). For the reader’s convenience, we describe the cocycle associated to a measured geodesic lamination $\lambda$, referring to \cite{Mess,BS17,barbotfillastre} for more details.

Let us fix a point $x_0\in\mathbb{H}^2\setminus\widetilde{\lambda}$. For any $\alpha\in\pi_1(\Sigma_g)$, consider $\mathcal{G}[x_0,\rho(\alpha)(x_0)]$, the space of geodesics in $\mathbb{H}^2$ intersecting $[x_0,\rho(\alpha)(x_0)]$. Then we define $\sigma:\mathcal{G}[x_0,\rho(\alpha)(x_0)]\to\minko$ as the map assigning to each geodesic $\ell$ (intersecting $[x_0,\rho(\alpha)(x_0)]$) the unit spacelike vector in $\minko$ orthogonal to $\ell$ and pointing towards $\rho(\alpha)(x_0)$. Then, we set:
\begin{equation}\label{cocycle_de_lamination}
\tau_{\lambda}(\alpha):=\Lambda\left(\displaystyle \int_{\mathcal{G}[x_0,\rho(\alpha)(x_0)]} \sigma  \mathrm{d}\lambda\right)
\end{equation}
It can be proven that $\tau_{\lambda}\in\mathrm{Z}^1_{\mathrm{Ad}\rho}(\pi_1(\Sigma),\mathfrak{isom}(\mathbb{H}^2))$. Moreover, if the basepoint $x_0$ is changed, the new cocycle differs from the preceding one by a coboundary.

Now, the key point to prove Proposition \ref{lengthvsThurs} is an interpretation of the length function in terms of Half-Pipe geometry due to Barbot and Fillastre. Before explaining this relation, we need to fix some notation: For any function $\psi:\mathbb{D}^2\to \mathbb{R}$, we denote by $\overline{\psi}:\mathbb{H}^2\to\mathbb{R}$ the function defined on the hyperboloid $\mathbb{H}^2$ (see \eqref{hyperboloid}) obtained by identifying $\mathbb{H}^2$ with $\mathbb{D}^2$ through the radial projection $\Pi$ defined in \eqref{radial}, namely for $\eta\in\mathbb{D}^2$:

$$\overline{\psi}(\Pi^{-1}(\eta))=\frac{\psi(\eta)}{\sqrt{1-\lVert\eta\rVert_{euc}^2}}.$$

\begin{prop}\cite[Propostion 3.31]{barbotfillastre}\label{Fillastre_barbot_length}
Let $\lambda$ be a measured geodesic lamination on $S_{\rho}$, and let $\tau_{\lambda}$ be the $\rho$-cocycle associated with $\lambda$. Consider $\phi_{\rho,\tau_{\lambda}}:\mathbb{S}^1\to\mathbb{R}$ to be the unique continuous function whose graph is preserved by the representation $\mathrm{Is}(\rho,\tau_{\lambda})$ (see Theorem \ref{nieseppi}). Then there exists a function $\phi_{\rho,\tau_{\lambda}}^m:\mathbb{D}^2\to\mathbb{R}$ with a graph preserved by $\mathrm{Is}(\rho,\tau_{\lambda})$ such that:
    \begin{enumerate}
        \item $\overline{\phi_{\rho,\tau_{\lambda}}^-}\leq\overline{\phi_{\rho,\tau_{\lambda}}^m}\leq\overline{\phi_{\rho,\tau_{\lambda}}^+}$. 
        \item The length of $\lambda$ is given by:
        \begin{equation}\label{944}\ell_{\rho}(\lambda)=\displaystyle \int_{S_{\rho}} (\overline{\phi_{\rho,\tau_{\lambda}}^m}-\overline{\phi_{\rho,\tau_{\lambda}}^-}) \, \mathrm{d}\mathrm{A}_{\rho},\end{equation}
    \end{enumerate}
     where $\mathrm{A}_{\rho}$ is the area form of the hyperbolic surface $S_{\rho}$.
\end{prop}
In fact, the graph of the function $\phi_{\rho,\tau_{\lambda}}^m$ is a surface of zero \textit{mean} curvature in Half-pipe space $\mathbb{D}^2\times\mathbb{R}$ and the integral in \eqref{944} is the Half-Pipe volume of the region bounded between the graph of $\phi_{\rho,\tau_{\lambda}}^-$ and $\phi_{\rho,\tau_{\lambda}}^m$ . In this paper, we will focus on both properties of the function $\phi_{\rho,\tau_{\lambda}}^m$ given in Proposition \ref{Fillastre_barbot_length}, and we refer the reader to \cite[Section 2.3.1]{barbotfillastre} for more details about the construction of $\phi_{\rho,\tau_{\lambda}}^m$. The second result that we need follows from the work of Bonsante and Seppi. We summarize here some results:

\begin{prop}\cite{BS17}\label{BS17_left_inf_earth}
Let $\lambda$ be a measured geodesic lamination on $S_{\rho}$, and let $\widetilde{\lambda}$ be its lift to $\mathbb{H}^2$. Consider the bounded left infinitesimal earthquake of the circle (see \eqref{bounded_infinitesimal_earthquake}) given by:
  $$X(z):=\frac{d}{dt}\big\lvert_{t=0}\mathrm{E}_l^{t\widetilde{\lambda}}(z).$$ Let $\phi_X:\mathbb{S}^1\to\mathbb{R}$ be the support function of $X$, and consider $x_0\in\mathbb{H}^2$ be a point that lies in a stratum of $\widetilde{\lambda}$. By translating by a Killing field, we furthermore assume that $\phi_X^-(x_0)=0$. Then:
  \begin{enumerate}
      \item \label{1BS17}The bending lamination of $\mathrm{gr}(\phi_X^-)$ is $\widetilde{\lambda}$, (see Definition \ref{def6.1}).
      \item\label{2BS17} Let $\tau_{\lambda}$ the cocycle associated to $\lambda$ given by:
      \begin{equation}\label{cocycleBS17}
    \tau_{\lambda}(\alpha)=\Lambda\left(\displaystyle \int_{\mathcal{G}[x_0,\rho(\alpha)(x_0)]} \sigma \, \mathrm{d}\lambda\right)
\end{equation} Then the graph of $\phi_X^-$ is invariant under the representation $\mathrm{Is}(\rho,\tau_{\lambda})$.
  \end{enumerate}
\end{prop}
The first statement \eqref{1BS17} is proven in \cite[Proposition 5.2]{BS17}. It is expressed in terms of the support function of the domain of dependence in Minkowski space and its dual lamination (see Remark \ref{remark_on_support_function}). The second statement \eqref{2BS17} follows from the work of Mess \cite{Mess} and is also formulated in terms of the domain of dependence in Minkowski space; see also \cite[Proposition 2.13]{BS17}.

\begin{proof}[Proof of Proposition \ref{lengthvsThurs}]Let $\lambda$ be a measured geodesic lamination on $S_{\rho}$ and $\widetilde{\lambda}$ its lift to $\mathbb{H}^2$. Consider $X$ the vector field on the circle defined by $X:=\frac{d}{dt}\big\lvert_{t=0}\mathrm{E}_l^{t\widetilde{\lambda}}$. Without loss of generality, we may assume that $X$ satisfies the hypothesis of Proposition \ref{BS17_left_inf_earth} up to translating by a Killing field. According to a result of Hu \cite[Theorem 2]{Junhu_Zygnorm}, the following estimate holds:
\begin{equation}\label{93}
\lVert X\rVert_{cr}\leq \frac{2(3e-1)}{e-1}\lVert \lambda\rVert_{\mathrm{Th}}^{\rho},
\end{equation}
where we recall that $\lVert \lambda\rVert_{\mathrm{Th}}^{\rho}=\lVert \widetilde{\lambda}\rVert_{\mathrm{Th}}$.
By Theorem \ref{TH2}, we have
\begin{equation}\label{94}
w(X)\leq \frac{8}{3}\lVert X\rVert_{cr}.
\end{equation}Using the first statement \eqref{1BS17} of Proposition \ref{BS17_left_inf_earth}, the measured lamination $\widetilde{\lambda}$ is the bending measure of the graph of $\phi_X^-$, and so by definition $\widetilde{\lambda}$ is the infinitesimal earthquake measure of the left infinitesimal earthquake, which extends continuously to $X$. Therefore, combining equations \eqref{93} and \eqref{94}, we get
\begin{equation}\label{95}
w(X)\leq \frac{16(3e-1)}{3(e-1)}\lVert \lambda\rVert_{\mathrm{Th}}^{\rho}.
\end{equation}

Now, by the second statement \eqref{2BS17} of Proposition \ref{BS17_left_inf_earth}, the graph of $\phi_X^-$ is preserved by $\mathrm{Is}(\rho,\tau_{\lambda})$, where $\tau_{\lambda}$ is the cocycle given by \eqref{cocycleBS17}. Therefore, by continuity, the graph of $\phi_X$ is also preserved by $\mathrm{Is}(\rho,\tau_{\lambda})$, and by Theorem \ref{nieseppi}, we necessarily have
$$\phi_X=\phi_{\rho,\tau_{\lambda}}.$$
Thus, from the last equation, the width of $X$ is given by
$$ w(X)=\underset{\eta\in\mathbb{D}^2}{\sup} \left( \mathrm{L}(\eta,\phi_{\rho,\tau_{\lambda}}^+(\eta))-\mathrm{L}(\eta,\phi_{\rho,\tau_{\lambda}}^{-}(\eta))\right)=\underset{x\in\mathbb{H}^2}{\sup} \left( \overline{\phi_{\rho,\tau_{\lambda}}^+}(x)-\overline{\phi_{\rho,\tau_{\lambda}}^{-}}(x)\right),$$
and so $0\leq\overline{\phi_{\rho,\tau_{\lambda}}^+}-\overline{\phi_{\rho,\tau_{\lambda}}^-}\leq w(X)$. Now let $\phi_{\rho,\tau_{\lambda}}^m:\mathbb{D}^2\to\mathbb{R}$ be the function given in Proposition \ref{Fillastre_barbot_length}. From the properties of the function $\phi_{\rho,\tau_{\lambda}}^m$, we deduce that
\begin{equation}\label{96}
\overline{\phi_{\rho,\tau_{\lambda}}^m}-\overline{\phi_{\rho,\tau_{\lambda}}^-}\leq\overline{\phi_{\rho,\tau_{\lambda}}^+}-\overline{\phi_{\rho,\tau_{\lambda}}^-}\leq w(X).
\end{equation}

Next, we will use the formula for the length of $\lambda$ given in Proposition \ref{Fillastre_barbot_length} and equations \eqref{95}, \eqref{96} to obtain:
\begin{align*}
\ell_{\rho}(\lambda)&=\displaystyle \int_{S_{\rho}} (\overline{\phi_{\rho,\tau_{\lambda}}^m}-\overline{\phi_{\rho,\tau_{\lambda}}^-}) \, \mathrm{d}\mathrm{A}_{\rho}\\
 &\leq \displaystyle \int_{S_{\rho}} w(X) \, \mathrm{d}\mathrm{A}_{\rho}\\
 &\leq \frac{16(3e-1)}{3(e-1)}\lVert\lambda\rVert_{\mathrm{Th}}^{\rho}\displaystyle \int_{S_{\rho}} 1 \, \mathrm{d}\mathrm{A}_{\rho}\\
 &=\frac{64\pi(3e-1)}{3(e-1)}(g-1)\lVert\lambda\rVert_{\mathrm{Th}}^{\rho},
\end{align*}
where we used in the last step the Gauss-Bonnet Theorem, this concludes the proof.
\end{proof}

\bibliographystyle{alpha}
\bibliography{zygmund.bib}

\end{document}